\documentclass[a4paper,reqno]{amsart}

\usepackage[utf8]{inputenc}
\usepackage{amssymb}
\usepackage{enumitem}
\usepackage{mathrsfs}
\usepackage{mathtools}
\usepackage{amsmath}
\usepackage [dvipsnames] { xcolor }
\usepackage{multirow}
\usepackage[colorlinks=true]{hyperref}

\hypersetup{linkcolor=OliveGreen,urlcolor=TealBlue,citecolor=TealBlue}

\definecolor{identifiercolor}{rgb}{.4,.6,.56}
\definecolor{stringcolor}{gray}{0.5}
\definecolor{inactivecolor}{rgb}{0.15,0.15,0.5}

\usepackage{listings}
\setlist[enumerate]{label=\emph{(\roman*)}}

\usepackage{environ}

\newtheorem{theorem}{Theorem}[section]

\newtheorem{lemma}[theorem]{Lemma}
\newtheorem{proposition}[theorem]{Proposition}

\theoremstyle{definition}
\newtheorem{definition}[theorem]{Definition}
\newtheorem{remark}[theorem]{Remark}
\numberwithin{equation}{section}
\newcommand{\R}{\mathbb{R}}

\newcommand{\dd}{{\rm d}}

\begin{document}

\title[2D mass-critical ZK equation]{On the near soliton dynamics for the 2D cubic Zakharov-Kuznetsov equations}

\author{Gong Chen}
\address{School of Mathematics, Georgia Institute of Technology, Atlanta, GA, 30332-0160, USA.}
\email{gc@math.gatech.edu}

\author{Yang Lan}
\address{Yau Mathematical Sciences Center, Tsinghua University, 100084 Beijing, P. R. China.}
\email{lanyang@mail.tsinghua.edu.cn}

\author{Xu Yuan}
\address{Department of Mathematics, The Chinese University of Hong Kong, Shatin, N.T., Hong Kong, P.R. China.}
\email{xu.yuan@cuhk.edu.hk}

\begin{abstract}
In this article, we consider the Cauchy problem for the cubic (mass-critical) Zakharov-Kuznetsov equations in dimension two:
\begin{equation*}
\partial_tu+\partial_{x_1}(\Delta u+u^3)=0,\quad (t,x)\in [0,\infty)\times \mathbb{R}^{2}.
\end{equation*}
For the initial data in $H^{1}$ close to the soliton and satisfying a suitable space-decay property, we fully describe the asymptotic behavior of the corresponding solution. More precisely, for such initial data, we show that only three possible behaviors can occur: 1) The solution leaves a tube near soliton in finite time; 2) the solution blows up in finite time; and 3) the solution is global and locally converges to a soliton. In addition, we show that for initial data near a soliton with non-positive energy and above the threshold mass, the corresponding solution will blow up as described in Case 2.

Our proof is inspired by the techniques developed for the mass-critical generalized Korteweg–de Vries (gKdV) equation in a similar context 
by Martel-Merle-Rapha\"el~\cite{MMR}. More precisely, our proof relies on refined modulation estimates and a modified energy-virial Lyapunov functional. The primary challenge in our problem is the lack of coercivity for the Schr\"odinger operator, which appears in the virial-type estimate.
To overcome the difficulty, we apply a transform, which was first introduced in Kenig-Martel~\cite{KenigMartel}, to perform the virial computations after converting the original problem into an adjoint one. The coercivity of the Schr\"odinger operator in the adjoint problem has been numerically verified by  Farah-Holmer-Roudenko-Yang~\cite{FHRY}. 
\end{abstract}

\maketitle

\section{Introduction}
\subsection{Main results}
Consider the 2D cubic Zakharov-Kuznetsov equation,
\begin{equation}\label{CP}
\partial_tu+\partial_{x_1}(\Delta u+u^3)=0,\quad (t,x)\in [0,\infty)\times \mathbb{R}^{2},
\end{equation}
where $x=(x_{1},x_{2})\in \mathbb{R}^{2}$ and $\Delta=\partial_{x_{1}}^{2}+\partial_{x_{2}}^{2}$ is the Laplace operator on $\mathbb{R}^{2}$. Recall that, by the work of \cite{Faminskii,Kin,LinPas,RibaudVento}, the Cauchy problem for equation~\eqref{CP} is locally well-posed in the energy space $H^{1}$: for any initial data $u_{0}\in H^{1}(\R^{2})$, there exists a unique (in a certain sense) maximal solution of~\eqref{CP} in $C\left([0,T):H^{1}(\R^{2})\right)$ with $u_{|t=0}=u_{0}$. Moreover, for this problem, the following blow-up criterion holds:
\begin{equation}\label{equ:Blowcri}
T<\infty\Longrightarrow \lim_{t\uparrow T}\|\nabla u(t)\|_{L^{2}}=\infty.
\end{equation}
For any $H^{1}$ solution $u$, the mass $M$ and energy $E$ are conserved, where 
\begin{equation*}
M(u(t))=\int_{\R^{2}}|u(t,x)|^{2}\dd x\quad \mbox{and}\quad 
E(u(t))=\frac{1}{2}\int_{\R^{2}}\left(|\nabla u(t,x)|^{2}-\frac{1}{2}|u(t,x)|^{4}\right)\dd x.
\end{equation*}
Recall also that, for any solution $u$ of~\eqref{CP} and $\lambda>0$, the scaling symmetry
\begin{equation*}
u_\lambda(t,x)=\lambda u(\lambda^3 t, \lambda x),\quad \mbox{for}\ (t,x)\in [0,\infty)\times\R^{2},
\end{equation*}
again results in a solution to \eqref{CP}. This scaling symmetry keeps the $L^2$-norm invariant so that the problem is \emph{mass-critical}.

Denote by $Q$ the \emph{ground state}, which is the unique radial positive solution of~\eqref{CP}: 
\begin{equation*}
-Q''-\frac{Q'}{r}+Q-Q^{3}=0,\quad Q'(0)=0\quad \mbox{and}\quad \lim_{r\to \infty}Q(r)=0.
\end{equation*}
It is well-known and easily checked that, for any $n\in \mathbb{N}$, 
\begin{equation*}
\left|Q^{(n)}(r)\right|\lesssim r^{-\frac{1}{2}}e^{-r},\quad \mbox{for}\ r>1.
\end{equation*}
We refer to Berestycki-Lions~\cite{BL} for the work related to the soliton $Q$. Using the symmetries of the equation, from $Q$, for any $(\lambda_0,x_{1,0},x_{2,0}) \in \mathbb{R}^+\times \mathbb{R}\times \R$, one can find the family of \emph{{soliton}/traveling wave} solutions to \eqref{CP}:
\begin{equation*}
    u(t,x)=\lambda_0 Q\left((\lambda_0(x_1-\lambda_0^2 t-x_{1,0}), \lambda_0 (x_2-x_{2,0})\right).
\end{equation*}

Based on a variational argument (see~\cite{BL,Kwong,Weinstein}),
the unique radial positive solution $Q$ attains the best constant $C$ in the following Gagliardo-Nirenberg inequality
\begin{equation*}
\|f\|_{L^{4}}^{4}\le C \|f\|_{L^{2}}^{2}\|\nabla f\|_{L^{2}}^{2},\quad \mbox{for any}\ f\in H^{1}(\R^{2}).
\end{equation*}
It follows from the definition of the energy $E$ that 
\begin{equation*}
E(u)\ge \frac{1}{2}\|\nabla u\|_{L^{2}}^{2}\left(1-\frac{\|u\|^{2}_{L^{2}}}{\|Q\|^{2}_{L^{2}}}\right),\quad \mbox{for any}\ u\in H^{1}(\R^{2}).
\end{equation*}
Combining the above estimate with the conservation of the energy and blow-up criterion, we obtain the global existence for any initial data with $\|u_{0}\|_{L^{2}}<\|Q\|_{L^{2}}$.

\smallskip
For the case of $\|u_{0}\|_{L^{2}}\ge \|Q\|_{L^{2}}$, the existence of blow-up solutions (in finite time or infinite time) has been an interesting problem and has attracted people’s attention in recent years. In particular, in this direction, the first result is obtained by Farah-Holmer-Roudenko-Yang~\cite{FHRY} which focuses on the blow-up dynamics for the case that the mass is slightly above the threshold.
More precisely, they show that, there exists $\alpha_{0}>0$ such that, for any initial data $u_{0}\in H^{1}(\R^{2})$ satisfying
\begin{equation*}
E(u_{0})<0\quad \mbox{and}\quad 0<\|u_{0}\|^{2}_{L^{2}}-\|Q\|_{L^{2}}^{2}\le \alpha_{0},
\end{equation*} 
the corresponding solution $u(t)$ blows up in finite or infinite forward time. 

\smallskip
In this article, we study the soliton dynamics of the 2D Cauchy problem~\eqref{CP}: we first prove the rigidity of the solution flow for the initial data near the soliton and then show a blow-up result for such solutions with non-positive energy. We start with the definition for the set of initial data and $L^{2}$-modulated tube near the soliton manifold.
\begin{definition}\label{def:Tube}
For any $\alpha>0$, we define the $L^2$-modulated tube near the soliton manifold as follows:
\begin{equation*}
\mathcal{T}_{\alpha}=\left\{ u\in H^{1}:\inf_{\substack{\lambda_{0}>0\\x_{0}\in \R^{2}}}\left\|u(\cdot)-\frac{1}{\lambda_{0}}Q\left(\frac{\cdot-x_{0}}{\lambda_{0}}\right)\right\|_{L^{2}}<\alpha\right\}.
\end{equation*}
Moreover, for any $\alpha>0$, we define the following initial data set:
\begin{equation*}
\mathcal{A}_{\alpha}=\left\{
u_{0}=Q+\varepsilon_{0}:\|\varepsilon_{0}\|_{H^{1}}<\alpha\ \mbox{and}\ 
\int_{\R}\int_{0}^{\infty}y_{1}^{100}\varepsilon_{0}^{2}(y_{1},y_{2})\dd y_{1}\dd y_{2}<1
\right\}.
\end{equation*}
\end{definition}

Our first main result is the following rigidity of the solution flow in $\mathcal{A}_{\alpha}$.
\begin{theorem}\label{MT}
There exist some universal constants $0<\alpha\ll\alpha^*\ll1$ such that the following is true. Let the initial data $u_{0}\in \mathcal{A}_{\alpha}$. Then for the corresponding solution $u(t)$ of~\eqref{CP}, one of the following three scenarios occurs{\rm{:}}

\begin{description}

\item [Exit] There exists a finite time $T\in (0,\infty)$ such that $u(T)\notin\mathcal{T}_{\alpha^*}$.

\item [Blow-up] The solution $u(t)$ blows up in finite time $T\in (0,\infty)$ with
	\begin{equation*}
	\|\nabla u(t)\|_{L^2}=\frac{\ell(u_{0})+o_{t\uparrow T}(1)}{(T-t)^{\beta}},
	\end{equation*}
    where $\beta\in \left(\frac{5}{7},\frac{5}{6}\right)$ is a universal constant and $\ell(u_{0})>0$ is a constant depending only on $u_{0}$ . Moreover, for all $t\in (0,T)$, we have $u(t)\in\mathcal{T}_{\alpha^*}$.
    
    \item [Soliton] The solution $u(t)$ is globally defined on $[0,\infty)$, and for all $t\in [0,\infty)$, we have $u(t)\in\mathcal{T}_{\alpha^*}$. Moreover, there exist some constants $(\lambda_{\infty},x_{2,\infty})\in [0,\infty)\times \R$ and a $C^1$ function $x(t)=(x_{1}(t),x_{2}(t))$ such that
    \begin{align*}
    &\left\|\lambda_{\infty}u(t,\lambda_{\infty}\cdot+x(t))-Q\right\|_{H^{1}_{{\rm loc}}}\rightarrow 0,\quad \quad \quad \quad\text{ as }t\rightarrow\infty,\\
    &|\lambda_{\infty}-1|\lesssim\delta(\alpha) \ \  \mbox{and}\ \ x(t)\sim \left(\frac{t}{\lambda_{\infty}^{2}},x_{2,\infty}\right),\text{ as }t\rightarrow\infty.
    \end{align*}
\end{description}

\end{theorem}

Our second main result focuses on the blow-up dynamics near the soliton in $\mathcal{A}_{\alpha}$.
\begin{theorem}\label{MT2}
There exists a universal constant $0<\alpha\ll 1$ such that the following is true. Let the initial data $u_{0}\in \mathcal{A}_{\alpha}$ be such that 
\begin{equation*}
E(u_{0})\le 0\quad \mbox{and}\quad  \|Q\|_{L^{2}}< \|u_{0}\|_{L^{2}}.
\end{equation*}
Then the corresponding solution $u(t)$ blows up in finite time $T$ in the regime described by Theorem~\ref{MT}.
\end{theorem}

\begin{remark}
In the statement of Theorem~\ref{MT}, the constant $\beta$ is defined by 
\begin{equation}\label{eq:theta}
\beta=\frac{1}{3-\theta}\quad \mbox{and}\quad 
\theta=2\bigg(\int_{\R}\frac{\big|\widehat{F}(\xi)\big|^{2}}{1+|\xi|^{2}}\dd \xi\bigg)
\bigg/
\left(\int_{\R}\big|\widehat{F}(\xi)\big|^{2}\dd \xi\right).
\end{equation}
Here, the functions $F$ and $\widehat{F}$ are given by 
\begin{equation*}
F(y_{2})=\int_{\R}\Lambda Q(y_{1},y_{2})\dd y_{1}\quad \mbox{and}\quad 
\widehat{F}(\xi)=\frac{1}{\sqrt{2\pi }}\int_{\mathbb{R}}F(y_{2})e^{-iy_{2}\xi}\dd y_{2}.
\end{equation*}
Actually, the value of the constant $\theta$ is quite essential in our analysis, since the blow-up rate is determined by it (see more discussion in \S\ref{SS:Comm}).  Using elementary numerical computations, we find $\theta\approx 1.66\in \left(\frac{8}{5},\frac{9}{5}\right)$ which implies the blow-up rate $\beta\in \left(\frac{5}{7},\frac{5}{6}\right)$. We present the details of the numerical computation in Appendix~\ref{App:A}.
\end{remark}

\begin{remark}
We point out that the \(y_{1}^{100}\) weight in Theorem~\ref{MT} and Theorem~\ref{MT2} is merely a technical restriction, and we do not claim its sharpness. For a similar but more relaxed restriction in the context of the mass-critical gKdV equation, we refer to~\cite[\S 1.3]{MMR}.
\end{remark}

\begin{remark}
A detailed numerical study on blow-up for~\eqref{CP} and its the mass supercritical counterpart $\partial_tu+\partial_{x_1}(\Delta u+u^4)=0$ was presented in~Klein-Roudenko-Stoilov~\cite{KRS}. In the  mass-critical case, the authors of~\cite{KRS} conjectured that the blow-up happens in finite time and  at spatial infinity. Moreover, the authors conjectured that the blow-up rate satisfies $\|\nabla u(t)\|_{L^{2}}\sim \left(T-t\right)^{-\frac{1}{2}}$. In the  mass-supercritical case, the authors of~\cite{KRS} conjectured that the blow-up happens in finite time and  at a finite spatial location. Similar to the  mass-critical problem,  the authors also conjectured that the blow-up rate satisfies $\|\nabla u(t)\|_{L^{2}}\sim \left(T-t\right)^{-\frac{2}{9}}$. It is worth mentioning that our result is different from the conjecture that was checked by numerical blow-up computations. More precisely, our blow-up rate is faster than the one stated in their conjecture.
\end{remark}

\begin{remark}
We mention here that very recently, a similar result was proved independently by the work Bozgan-Ghoul-Masmoudi~\cite{BGM}, using a similar but different method. Most importantly, the two works employ different energy-virial Lyapunov functionals. (see more discussion in \S\ref{SS:Comm} and \S\ref{S:Mono}). 
\end{remark}

\subsection{Related results}\label{SS:Related}
In the last twenty years, there has been remarkable progress towards understanding the near soliton dynamics and singularity formation of nonlinear dispersive equations, particularly for the mass-critical gKdV and nonlinear Schr\"odinger (NLS) equations. For the gKdV equation, in the work of Martel-Merle-Rapha\"el~\cite{MMR,MMR1,MMR2}, the authors gave a complete description and classification of the solution flow near soliton that completes the previous results in~\cite{MartelMerleJMPA,MMGAFA,MMANN,MMJAMS,MMDUKE,Merle}. Then, building on the work~\cite{MMR}, Martel-Merle-Nakanishi-Rapha\"el~\cite{MMNR} constructed the threshold manifold near the soliton for gKdV equation. For the NLS equation, in the work of Bourgain-Wang~\cite{BW}, the authors constructed  a family of conformal blow-up solutions. Then, based on the series works~\cite{MR1,MR2,MR3,MR4,MR5} of Merle-Rapha\"el,  a complete description of singularity formation for the solution flow near soliton was obtained. We refer to~\cite{KK1,KK2,KSnls,LAN1,LAN2,LAN3,Merlenls,MRS,Perelmannls} for some related results of mass-critical models. We also refer to~\cite{JJ1,JJ2,KSTwave,MRRmap,RR,RSmap,Rodnianski} and references therein  for some related results of energy-critical models, which are natural analogies of mass-critical models.  

\smallskip
We now briefly survey the literature related to the Zakharov-Kuznetsov models. The Zakharov-Kuznetsov equations are natural extensions of the gKdV equations in higher dimensions and are of  physical importance. For more historical and physical background, we refer to the introductions by Farah-Holmer-Roudenko-Yang \cite{FHRY1,FHRY}. The local and global well-posedness theory of the Zakharov-Kuznetsov models with various nonlinear powers and in different dimensions has been studied extensively. Without attempting to be exhaustive, we refer to the works \cite{Faminskii,Kin,LinPas,RibaudVento} and the references therein for details.

\smallskip
For the 2D mass-critical problem \eqref{CP}, the instability of the soliton was obtained by Farah-Holmer-Roudenko \cite{FHR}. This matches the situation with the critical gKdV equation. Again in the mass-critical case, when the initial data has negative energy with the mass slightly above the threshold, Farah-Holmer-Roudenko-Yang~\cite{FHRY} showed that the gradient of the solution blows up in finite or infinite forward time. Finally, we would like to mention that for the 2D mass super-critical problems, the instability of soliton was shown by  Farah-Holmer-Roudenko \cite{FHR1}.

\smallskip
For the 2D quadratic Zakharov-Kuznetsov equation, after passing to the adjoint problem without regularization, the asymptotic stability of soliton and stability of multi-solitons have been proven by C\^ote-Mu\~noz-Pilod-Simpson~\cite{CMPS} using virial and monotonicity estimates. Their virial estimates relied on a sign condition verified numerically. However, these numerical computations do not hold for the problem in the 3D case. Recently, Farah-Holmer-Roudenko-Yang derived in \cite{FHRY1} a new virial estimate in the case of 3D, based on different orthogonality conditions, converting to an adjoint problem with regularizations and relying on the numerical analysis of the spectra of a linear operator. This allowed them to extend the asymptotic stability result of soliton to the 3D quadratic Zakharov-Kuznetsov equation. In the work of Mendez-Mu\~noz-Poblete-Pozo \cite{MMPP}, some new virial estimates in the cases of 2D and 3D are used to prove the decay of solutions in large time-dependent spatial regions. We mention here that, the existence and uniqueness of asymptotic multi-solitons were shown by Valet \cite{V}  in the cases of 2D and 3D using the strategy introduced by Martel~\cite{MartelAJM} for the gKdV equation.
 Very recently, the asymptotic stability of multi-solitons to the 2D and 3D quadratic Zakharov-Kuznetsov equations was established in Pilod-Valet \cite{PV2}, and then, the same authors described the collision of  two nearly equal solitary waves on the whole time interval and proved the stability of this phenomenon in \cite{PV1}.

\subsection{Comments on the proof}\label{SS:Comm}
The method of the current article, based on the use of the energy-virial Lyapunov functional, is inspired by the remarkable work \cite{MMR} in a similar context for the mass-critical gKdV equation. The first step of the method is the decomposition of solution $u$ for~\eqref{CP} close to a soliton in $H^{1}$:
\begin{equation*}
u(t,x)=\frac{1}{\lambda(t)}\left(Q_{b(t)}+\varepsilon\right)\left(t,\frac{x-x(t)}{\lambda(t)}\right).
\end{equation*}
Here, $Q_{b}$ is close to $Q$ for $b$ small enough. In the current case (see also~\cite[\S2.2]{MMR} for the case of gKdV equation), we can write the function $Q_{b}$ as 
\begin{equation*}
Q_{b}=Q+bP\phi_{b},
\end{equation*}
where $\phi_{b}$ is a suitable localized profile and $P$ is a non-localized profile to be determined.
As usual in investigating the blow-up phenomenon of mass-critical problem, we introduce the following new variables:
\begin{equation*}
s=\int_{0}^{t}\frac{1}{\lambda^{3}(\tau)}\dd \tau\quad \mbox{and}\quad 
y=\frac{x-x(t)}{\lambda(t)}.
\end{equation*}
Now, the sharp description of the near soliton dynamics relies on the determination of the finite-dimensional dynamical system for a suitable choice of geometrical parameters $(\lambda(t),b(t),x(t))$, coupled to the infinite-dimensional dynamics related to the reminder term $\varepsilon(t)$. Roughly speaking, in our analysis, we handle the finite-dimensional dynamics via standard ODE argument and then handle the infinite-dimensional dynamics via energy-virial Lyapunov functional. 

\smallskip

However, the presence of an additional dimension and the non-explicit soliton expression in our problem introduce several challenges compared to the case of gKdV equation. 
A direct problem we encounter when studying the 2D Cauchy problem, using the general strategy introduced by~\cite{MMR}, is that we have to solve a different elliptic equation when constructing the non-localized profile $P$ since an extra dimension $x_{2}$ exists. This directly leads to the difference in the finite-dimensional system for the geometrical parameters ($\lambda,b$). Roughly speaking, in our case now, the geometrical parameters $(\lambda,b)$ satisfy
\begin{equation}\label{equ:ode}
\frac{\dd s}{\dd t}=\frac{1}{\lambda^{3}},\quad 
\frac{\lambda_{s}}{\lambda}=-b\quad \mbox{and}\quad 
b_{s}+\theta b^{2}=0.
\end{equation}

Equation~\eqref{equ:ode} differs slightly from the case of gKdV equation (see~\cite[Page 67]{MMR}). It is worth mentioning that due to the presence of the constant $\theta$, we obtain a slightly different dynamic for the solutions, which results in a different blow-up rate. This may reflect the influences and interactions arising from the additional dimension. More precisely, from~\eqref{equ:ode}, we directly have 
\begin{equation*}
\frac{\lambda_{t}}{\lambda}=-\frac{b}{\lambda^{3}}\quad \mbox{and}\quad 
b_{t}=-\theta\frac{b^{2}}{\lambda^{3}}\Longrightarrow \frac{\dd }{\dd t}\left(\frac{b}{\lambda^{\theta}}\right)=0.
\end{equation*}
Formally, from a standard ODE argument, we obtain the following three scenarios for the dynamics of the solution with initial data $(\lambda,b)_{|t=0}=\left(1,b_{0}\right)$:
\begin{enumerate}
\item [(i)] For the case of $b_{0}<0$, we have $\lambda(t)=\left(1-(3-\theta)b_{0}t\right)^{\frac{1}{3-\theta}}$ on $[0,\infty)$ and the dynamic is stable.

\smallskip
\item [(ii)] For the case of $b_{0}=0$, we have $\lambda(t)\equiv 1$ on $[0,\infty)$ and the dynamic is unstable.

\item [(iii)] For the case of $b_{0}>0$, we have $\lambda(t)=\left(b_{0}(3-\theta)\left(T-t\right)\right)^{\frac{1}{3-\theta}}$ on $[0,\infty)$ where $T=(b_{0}(3-\theta))^{-1}>0$ and the dynamic is stable.
\end{enumerate}

\smallskip
Another, and more significant, problem we encounter when studying the 2D Cauchy problem is the lack of coercivity for the Schrödinger operator in the energy-virial estimate for the original remainder term  $\varepsilon$. Heuristically, when we compute the time variation of the virial quantity $\int_{\R^{2}}y_{1}\varepsilon^{2}\dd y_1\dd y_2$, the following Schr\"odinger operator $\mathcal{B}$ will appear in the estimate:
\begin{equation*}
\mathcal{B}=-\frac{3}{2}\partial_{y_{1}}^{2}-\frac{1}{2}\partial_{y_{2}}^{2}+\frac{1}{2}
-\frac{3}{2}Q^{2}+3y_{1}Q\partial_{y_{1}}Q.
\end{equation*}

In contrast to the case of the gKdV equation (see~\cite[Proposition 4]{MartelMerleJMPA} and~\cite[Lemma 3.4]{MMR}), the coercivity of the operator $\mathcal{B}$ under suitable orthogonality conditions has not yet been obtained. In particular, the analysis and numerical verification for the 2D operator seems to be significantly more challenging than in the 1D case.  To overcome this difficulty, inspired by the work of~\cite{FHRY}, we employ a transform introduced in Kenig-Martel~\cite{KenigMartel} to reduce the original problem to an adjoint one related to a new term  $\eta$. To the best of our knowledge, this technique was first introduced in the context of the gKdV equation by Martel~\cite{Martel}.
Here, we consider the regularized dual problem $\eta=(1-\gamma\Delta)^{-1}\mathcal{L}\varepsilon$ where $\gamma$ is a sufficiently small constant and $\mathcal{L}=-\Delta+1-3Q^2$ is the linearized operator around the soliton $Q$.  Actually, when we compute the time variation of the virial quantity $\int_{\R^{2}}y_{1}\eta^{2}\dd y_1 \dd y_2$, the following Schr\"odinger operator $\mathcal{A}$ will appear in the estimate:
\begin{equation*}
\begin{aligned}
\mathcal{A}f
&=-\frac{3}{2}\partial_{y_{1}}^{2}f-\frac{1}{2}\partial_{y_{2}}^{2}f+\frac{1}{2}f
-\left(\frac{3}{2}Q^{2}+3y_{1}Q\partial_{y_{1}}Q\right)f\\
&+3\frac{(f,y_{1}Q)}{(Q,Q)}Q^{2}\partial_{y_{1}}Q+3\frac{(f,Q^{2}\partial_{y_{1}}Q)}{(Q,Q)}y_{1}Q,\quad \mbox{for any}\ f\in H^{1}(\R^{2}).
\end{aligned}
\end{equation*}
We mention here that, under suitable orthogonality conditions, the coercivity of the operator $\mathcal{A}$ has been verified in~\cite[\S 16]{FHRY} through numerical computation. Therefore, even though we do not know whether the suitable coercivity of the operator $\mathcal{B}$ exists, we can establish a virial estimate for the regularized dual problem of $\eta$ based on the coercivity of the operator $\mathcal{A}$. Then, by combining the energy estimate of $\varepsilon$ with the virial estimate of  $\eta$, we obtain the energy-virial Lyapunov functional with monotonicity, and thus, we could obtain the control of the infinite-dimensional term over the whole space. We point out that the weight functions appearing in the energy and virial quantities must be chosen carefully, as the constant dependence on the time variation of such quantities differs and should be handled attentively. (see more details in \S\ref{SS:ENERGY}--\S\ref{SS:Mon}). Last, we also point out that the regularized transformation we apply here also reminisces the Darboux transformations which were successfully applied to study the asymptotic stability of kinks and solitons in various problems on any compact interval (see for instance~\cite{KMMV,Mar}).

\subsection{Outline of the article}
The article is organized as follows. First, Section~\ref{S:Pre} introduces the technical tools involved in the choice of the localized profile $Q_{b}$: spectral theory of the linearized operator, the pointwise estimates of the non-localized profile $P$ and the localized profile $Q_{b}$. Then, Section~\ref{S:Modu} introduces the technical tools involved in a dynamical approach to the soliton problem for~\eqref{CP}: geometric decomposition near soliton and the modulation estimates for the geometric parameters. Next, Section~\ref{S:Mono} focuses on the establishment of the energy-virial Lyapunov functional that plays a crucial role in our analysis.
Finally, by the monotonicity of the energy-virial Lyapunov functional and a suitable bootstrap argument, Theorem~\ref{MT} and Theorem~\ref{MT2} are proved in Section~\ref{S:Endproof} and Section~\ref{S:Endproof2}, respectively.

\subsection*{Acknowledgments}
The authors would like to thank Kuang Huang for the helpful discussion on the numerical computation related to the soliton.  The authors would also like to thank Yang Ge for the support of the coding for Mathematica. The authors are grateful to Claudio Mu\~noz, Didier Pilod, Frederic Valet and Kai Yang for valuable comments on the manuscript.

\section{Preliminaries}\label{S:Pre}
\subsection{Notation and Conventions}\label{SS:Nota}
For any $\alpha=(\alpha_{1},\alpha_{2})\in \mathbb{N}^{2}$, we set 
\begin{equation*}
|\alpha|=|\alpha_{1}|+|\alpha_{2}|\quad \mbox{and}\quad 
\partial_{y}^{\alpha}=\frac{\partial^{|\alpha|}}{\partial_{y_{1}}^{\alpha_{1}}\partial_{y_{2}}^{\alpha_{2}}}.
\end{equation*}
We denote by $\mathcal{Y}(\mathbb{R}^d)$ the set of smooth function $f$ on $\mathbb{R}^d$, such that for all $n\in\mathbb{N}$, there exist $r_n>0$ and $C_n>0$ such that
\begin{equation*}
\sum_{|\alpha|=n}|\partial_{y}^{\alpha}f(y)|\leq C_n(1+|y|)^{r_n}e^{-|y|},\quad \mbox{on}\ \R^{d}.
\end{equation*}
We also denote by $\mathcal{Z}(\mathbb{R}^d)$ the set of smooth function $f$ on $\mathbb{R}^d$, such that for all $n\in\mathbb{N}$, there exist $r_n>0$ and $C_n>0$ such that
\begin{equation*}
\sum_{|\alpha|=n}|\partial_{y}^{\alpha}f(y)|\leq C_n(1+|y|)^{r_n}e^{-\frac{|y|}{2}},\quad \mbox{on}\ \R^{d}.
\end{equation*}
For any $(f,g)\in L^{2}(\R)\times L^{2}(\R)$, we introduce the following $L^{2}(\R^{2})$ function
\begin{equation*}
f\otimes g: y=(y_{1},y_{2})\longmapsto f(y_{1})g(y_{2}).
\end{equation*}
The Fourier transform of a function $h\in L^{1}(\mathbb{R}^{d})$, denoted by $\mathcal{F}h$ or $\widehat{h}$, is defined as:
\begin{equation*}
\mathcal{F}h(\xi)=\widehat{h}(\xi)=\frac{1}{(2\pi)^{\frac{d}{2}}}\int_{\R^{d}}h(y)e^{-iy \cdot\xi}\dd y,\quad \mbox{on}\ \mathbb{R}^{d}.
\end{equation*}
The Fourier transform defines a linear isometric operator on $L^{2}(\R^{d})$, that is, 
\begin{equation*}
\int_{\R^{d}}|h(y)|^{2}\dd y=\int_{\R^{d}}|\widehat{h}(\xi)|^{2}\dd \xi,\quad \mbox{for any}\ h\in L^{2}(\R^{d}).
\end{equation*}
The inverse Fourier transform of a function $h$, denoted by $\mathcal{F}^{-1}h$, is defined as:
\begin{equation*}
\mathcal{F}^{-1}h(y)=\frac{1}{(2\pi)^{\frac{d}{2}}}\int_{\R^{d}}h(\xi)e^{iy \cdot\xi}\dd \xi,\quad \mbox{on}\ \mathbb{R}^{d}.
\end{equation*}
Recall that, we denote by $Q(y):=Q(|y|)$ the unique radial positive solution of~\eqref{CP}: 
\begin{equation*}
-Q''-\frac{Q'}{r}+Q-Q^{3}=0,\quad Q'(0)=0\quad \mbox{and}\quad \lim_{r\to \infty}Q(r)=0.
\end{equation*}
Based on the ODE arguments, for any $n\in \mathbb{N}$, there exists $C_{n}>0$ such that 
\begin{equation*}
\sum_{|\alpha|=n}\left|\partial^{\alpha}_{y}Q(y)\right|\lesssim e^{-|y|},\quad \mbox{on}\ \R^{2}.
\end{equation*}
Recall also that, from a variational argument,
the unique radial positive solution $Q$ attains the best constant $C$ in the following Gagliardo-Nirenberg inequality
\begin{equation*}
\|f\|_{L^{4}}^{4}\le C \|f\|_{L^{2}}^{2}\|\nabla f\|_{L^{2}}^{2},\quad \mbox{for any}\ f\in H^{1}(\R^{2}).
\end{equation*}
We denote the linearized operator around $Q$ by
\begin{equation*}
\mathcal{L}f=-\Delta f+f-3Q^{2}f,\quad \mbox{for}\ f\in H^{1}(\R^{2}).
\end{equation*}
Next, we introduce the scaling operator:
\begin{equation*}
\Lambda f= f+y\cdot \nabla f,\quad \mbox{for}\ f\in H^{1}(\R^{2}).
\end{equation*}
In this article, we set 
\begin{equation}\label{equ:defF}
F(y_{2})=\int_{\R}\Lambda Q(y_{1},y_{2})\dd y_{1},\quad \mbox{on}\ \R.
\end{equation}
For a given small constant $\alpha$, we denote by $\delta(\alpha)$ a generic small constant with
\begin{equation*}
\delta(\alpha)\to 0,\quad \mbox{as}\ \alpha\to 0.
\end{equation*}
For any $f\in L^{2}(\R^{2})$ and $g\in L^{2}(\R^{2})$, we denote the $L^2$-inner product by
\begin{equation*}
(f,g)=\int_{\R^{2}}f(y)g(y)\dd y.
\end{equation*}
Then, for any $T\in \mathcal{S}'(\R^{2})$ and $f\in \mathcal{S}(\R^{2})$, we denote by $\left<T,f\right>$ the canonical duality pairing between the distribution $T$ and the test function $f$.

\smallskip

In our analysis, we always need to carefully trace the dependence on a large-scale constant $B$. We set the following conventions:  the implied constants in $\lesssim$ and $O$  are \emph{independent of} $B$  from  Section~\ref{S:Pre} to  Section~\ref{S:Mono} and can \emph{depend on} the large constant $B$ in Section~\ref{S:Endproof} and Section~\ref{S:Endproof2}.

\subsection{The linearized operator}\label{SS:Linopera}
In this subsection, we recall the spectral theory of the linearized operator $\mathcal{L}$ and then introduce the useful function for the construction of the localized profile.
\begin{proposition}[Spectral properties of $\mathcal{L}$]\label{Prop:Spectral}
${\rm{(i)}}$ \emph{Spectrum}. The self-adjoint operator $\mathcal{L}$ has essential spectrum $[1,\infty)$, a unique single negative eigenvalue $-\mu_{0}$ with $\mu_{0}>0$, and its kernel is ${\rm{Span}}\left(\partial_{y_{1}}Q,\partial_{y_{2}}Q\right)$. Let $Y$ be the $L^{2}$ normalized eigenvector of $\mathcal{L}$ corresponding to the eigenvalue $-\mu_{0}$. It holds, for all $\alpha\in \mathbb{N}^{2}$,
\begin{equation*}
\left|\partial_{y}^{\alpha}Y(y)\right|\lesssim e^{-|y|},\quad \mbox{on}\ \ \mathbb{R}^{2}.
\end{equation*}

\smallskip
${\rm{(ii)}}$ \emph{$L^{2}$-scaling identities}. We have 
\begin{equation*}
\mathcal{L}\Lambda Q=-2Q\quad \mbox{and}\quad (Q,\Lambda Q)=0.
\end{equation*}

${\rm{(iii)}}$ \emph{First coercivity.} There exists $\mu>0$ such that, for all $f\in H^{1}(\R^{2})$,
\begin{equation*}
(\mathcal{L}f,f)\ge \mu\|f\|^2_{H^1}-\frac{1}{\mu}\left((f,Y)^2+(f,\partial_{y_{1}}Q)^2+(f,\partial_{y_{2}}Q)^2\right).
\end{equation*}

\smallskip
${\rm{(iv)}}$ \emph{Second coercivity.} There exists $\nu>0$ such that, for all $f\in H^{1}(\R^{2})$,
\begin{equation*}
(\mathcal{L}f,f)\ge \nu\|f\|^2_{H^1}-\frac{1}{\nu}\left((f,Q^3)^2+(f,\partial_{y_{1}}Q)^2+(f,\partial_{y_{2}}Q)^2\right).
\end{equation*}

\smallskip
${\rm(v)}$ \emph{Inversion of $\mathcal{L}$.}
Let $g\in L^{2}(\R^{2})$ be such that $\left|(g,\nabla Q)\right|=0$. Then there exists a unique $f\in H^{2}(\mathbb{R}^{2})$ such that $\mathcal{L}f=g$ and $\left|(f,\nabla Q)\right|=0$. Moreover, if $g$ is even in either $y_{1}$ and $y_{2}$, then $f$ is also even in $y_{1}$ or $y_{2}$. In addition, if $g\in \mathcal{Y}\left(\R^{2}\right)$, then we have $f\in \mathcal{Z}(\R^{2})$.

\end{proposition}

\begin{proof}
Proof of (i)--(ii). The properties of $\mathcal{L}$ in (i)-(ii) are standard and easily checked. We refer to~\cite[Theorem 3.1 and Lemma 3.2]{FHR} for the details of the proofs.

\smallskip
Proof of (iii)--(iv). First, from~\cite[Lemma 3.6]{FHR}, there exists $\tau>0$ such that, for all $f\in H^{1}(\R^{2})$, we have
\begin{equation*}
\left(\mathcal{L}f,f\right)\ge \tau\left\|f\right\|_{L^{2}}^{2}-\frac{1}{\tau}\left((f,Y)^2+(f,\partial_{y_{1}}Q)^2+(f,\partial_{y_{2}}Q)^2\right).
\end{equation*}
Therefore, for any $0<\delta\ll1$, we obtain
\begin{equation*}
\begin{aligned}
\left(\mathcal{L}f,f\right)
&=\delta \left(\mathcal{L}f,f\right)+(1-\delta)\left(\mathcal{L}f,f\right)\\
&\ge \delta \|f\|_{{H}^{1}}^{2}+\tau(1-\delta)\|f\|_{L^{2}}^{2}-3\delta\left(Q^{2},f^{2}\right)\\
&-\frac{1}{\tau}\left((f,Y)^2+(f,\partial_{y_{1}}Q)^2+(f,\partial_{y_{2}}Q)^2\right),
\end{aligned}
\end{equation*}
which completes the proof of (iii) by taking $0<\delta\ll1$ small enough. The proof of (iv) is directly based on a similar argument to the one above and~\cite[Lemma 3.5]{FHR}.

\smallskip
Proof of (v). We define
\begin{equation*}
\begin{aligned}
L^{\perp}&=\left\{v\in L^{2}(\R^{2}):\left|(v,\nabla Q)\right|=\left(v,Y\right)=0\right\},\\
H^{\perp}&=\left\{v\in H^{1}(\R^{2}):\left|(v,\nabla Q)\right|=\left(v,Y\right)=0\right\}.
\end{aligned}
\end{equation*}
First, for any $g\in L^{\perp}$, the map $M_{1}$ which is defined by
\begin{equation*}
\begin{aligned}
M_{1}: h\longmapsto (h,g),\quad \mbox{for}\ h\in H^{\perp},
\end{aligned}
\end{equation*}
is a linear bounded functional. On the other hand, from (iii) of Proposition~\ref{Prop:Spectral}, we know that the map $M_{2}$ which is defined by 
\begin{equation*}
M_{2}: (f_{1},f_{2})\longmapsto \left(\mathcal{L}f_{1},f_{2}\right),\quad \mbox{for}\ (f_{1},f_{2})\in H^{\perp}\times H^{\perp},
\end{equation*}
is an inner product. Therefore, from Riesz representation Theorem, for any $g\in L^{\perp}$, there exists unique $f\in H^{\perp}$ such that 
\begin{equation*}
\left(\mathcal{L}f,h\right)=\left(g,h\right),\quad \mbox{for any}\ h\in H^{\perp}.
\end{equation*}
It follows from $(g,f)\in L^{\perp}\times H^{\perp}$ that 
\begin{equation}\label{equ:Lf=g}
\left(\mathcal{L}f,h\right)=\left(g,h\right),\ \ \mbox{for any}\ h\in H^{1}(\R^{2})\Longrightarrow
\mathcal{L}f=g,\ \ \mbox{on}\ \R^{2}.
\end{equation}
Second, for any $g\in L^{2}(\R^{2})$ with $\left|(g,\nabla Q)\right|=0$, we decompose
\begin{equation*}
g=g^{\perp}+aY,\ \ \mbox{where}\ g^{\perp}\in L^{\perp}\  \mbox{and} \ a=(g,Y).
\end{equation*}
Using~\eqref{equ:Lf=g}, we find, there exists unique $f^{\perp}\in H^{\perp}$ such that $\mathcal{L}f^{\perp}=g^{\perp}$ on $\R^{2}$. It follows that 
\begin{equation*}
f=f^{\perp}-\frac{a}{\mu_{0}}Y \ \mbox{with}\ |\left(f,\nabla Q\right)|=0\Longrightarrow \mathcal{L}f=g,\quad \mbox{on}\ \R^{2}.
\end{equation*}
Then, the uniqueness of $f$ is a direct consequence of ${\rm{Ker}}\mathcal{L}={\rm{Span}}\left\{\partial_{y_{1}}Q,\partial_{y_{2}}Q\right\}$. Note that, from the uniqueness of $f$, we obtain, if $g$ is even in either $y_{1}$ and $y_{2}$, then $f$ is also even in $y_{1}$ or $y_{2}$. 

\smallskip
Third, taking Fourier transforms on the both sides of $\mathcal{L}f=g$, we have
\begin{equation*}
\widehat{f}(\xi)=\frac{1}{1+|\xi|^{2}}\left(3\widehat{Q^{2}f}(\xi)+\widehat{g}(\xi)\right),\quad \mbox{on}\ \R^{2},
\end{equation*}
which implies
\begin{equation}\label{equ:fy}
f(y)=\left[\mathcal{F}^{-1}\left(\frac{1}{1+|\xi|^{2}}\right)\ast (3Q^{2}f+g)\right](y),\quad \mbox{on}\ \R^{2}.
\end{equation}
On the one hand, from standard elliptic arguments, we have 
\begin{equation}\label{est:3Q2f}
f\in \bigcap_{k=1}^{\infty}H^{k}(\R^{2})\Longrightarrow
3Q^{2}{f}+g\in \mathcal{Y}(\R^{2}).
\end{equation}
On the other hand, for any regular function $h\in \mathcal{S}(\R^{2})$, we find
\begin{equation*}
\begin{aligned}
\left<\mathcal{F}^{-1}\left(\frac{1}{1+|\xi|^{2}}\right),\widehat{h}\right>
&=\int_{\R^{2}}\frac{h(\xi)}{1+|\xi|^{2}}\dd \xi=\int_{0}^{\infty}e^{-\rho}\left(\int_{\R^{2}}h(y)e^{-\rho |y|^{2}}\dd y\right)\dd \rho.
\end{aligned}
\end{equation*}
By an elementary computation, 
\begin{equation*}
\mathcal{F}\left(e^{-\rho |y|^{2}}\right)(\xi)=(2\rho)^{-1}e^{-\frac{|\xi|^{2}}{4\rho}},\quad \mbox{on}\ \R^{2}.
\end{equation*}
It follows from the Plancherel Theorem that 
\begin{equation*}
\left<\mathcal{F}^{-1}\left(\frac{1}{1+|\xi|^{2}}\right),\widehat{h}\right>
=\frac{1}{2}\left[\int_{\R^{2}}\left(\int_{0}^{\infty}\rho^{-1}e^{-\left(\rho+\frac{|\xi|^{2}}{4\rho}\right)}\dd \rho\right) \widehat{h}(\xi)\dd \xi\right].
\end{equation*}
Based on the above identity, we know that
\begin{equation*}
\mathcal{F}^{-1}\left(\frac{1}{1+|\xi|^{2}}\right)(y)=\frac{1}{2}\int_{0}^{\infty}\rho^{-1}e^{-\left(\rho+\frac{|y|^{2}}{4\rho}\right)}\dd \rho\in L^{1}(\R^{2})\cap C^{\infty}\left(\R^{2}\setminus \left\{0\right\}\right).
\end{equation*}
Moreover, for any $y\in \R^{2}$ with $|y|>1$, we have 
\begin{equation*}
\begin{aligned}
\left|\mathcal{F}^{-1}\left(\frac{1}{1+|\xi|^{2}}\right)(y)\right|
&\lesssim \int_{0}^{\frac{|y|}{8}}\rho^{-1}e^{-\left(\rho+\frac{|y|^{2}}{4\rho}\right)}\dd \rho\\
&+\int^{2|y|}_{\frac{|y|}{8}}\rho^{-1}e^{-\left(\rho+\frac{|y|^{2}}{4\rho}\right)}\dd \rho+ \int_{2|y|}^{\infty}\rho^{-1}e^{-\left(\rho+\frac{|y|^{2}}{4\rho}\right)}\dd \rho\lesssim e^{-|y|}.
\end{aligned}
\end{equation*}
Combining the above estimate with~\eqref{equ:fy}--\eqref{est:3Q2f}, we conclude that $f\in \mathcal{Z}(\R^{2})$.
\end{proof}

Recall that, we set 
\begin{equation*}
F(y_{2})=\int_{\R}\Lambda Q(y_{1},y_{2})\dd y_{1},\quad \mbox{on}\ \R.
\end{equation*}
By an elementary computation, for any $n\in \mathbb{N}^{+}$, we have 
\begin{equation*}
\begin{aligned}
\left|\frac{\dd^{n}F}{\dd y_{2}^{n}}(y_{2})\right|
&\lesssim \int_{|y_{1}|\le |y_{2}|}\left(1+|y_{1}|+|y_{2}|\right)e^{-\sqrt{y_{1}^{2}+y^{2}_{2}}}\dd y_{1}\\
&+\int_{|y_{1}|> |y_{2}|}\left(1+|y_{1}|+|y_{2}|\right)e^{-\sqrt{y_{1}^{2}+y^{2}_{2}}}\dd y_{1}\lesssim (1+|y_{2}|^{2})e^{-|y_{2}|},
\end{aligned}
\end{equation*}
which means that $F\in \mathcal{Y}(\R)$. 

\smallskip
For future reference,  we denote by $h_{2}\in \mathcal{Y}(\R)$ the even solution of the following second-order ODE:
\begin{equation}\label{def:h2}
-h_{2}''(y_{2})+h_{2}(y_{2})=F''(y_{2}),\quad \mbox{on}\ \R.
\end{equation}
Then, we set 
\begin{equation*}
G: y_{1}\longmapsto \int_{\R}h_{2}(y_{2})Q(y_{1},y_{2})\dd y_{2}.
\end{equation*}
We fix a regular function $h_{1}\in \mathcal{Y}(\R)$ such that $\int_{\R}h_{1}(y_{1})\dd y_{1}=1$ and $h_{1}$ is orthogonal to $G$ in the $L^{2}(\R)$ sense. It follows that 
\begin{equation}\label{equ:h1h2}
\left(h_{1}\otimes h_{2},Q\right)=\int_{\R}h_{1}(y_{1})\left(\int_{\R}h_{2}(y_{2}) Q(y_{1},y_{2})\dd y_{2}\right)\dd y_{1}=0.
\end{equation}

\smallskip
We now introduce the following non-localized profile for future reference.
\begin{lemma}[Non-localized profile]\label{le:Nonloca}
There exists a smooth function $P\in C^{\infty}(\R^{2})$ such that $\partial_{y_{1}}P\in \mathcal{Z}(\R^{2})$ and 
\begin{equation}\label{equ:P1}
\partial_{y_{1}}\mathcal{L}P=\Lambda Q,\quad \lim_{y_{1}\to \infty}\partial_{y_{2}}^{n}P(y_{1},y_{2})=0,\ \forall n\in \mathbb{N},
\end{equation}
\begin{equation}\label{equ:P2}
|(\nabla P,Q)|=0\quad \mbox{and}\quad (P,Q)=\frac{1}{4}\int_{\mathbb{R}}|F(y_{2})|^{2}\dd y_2.
\end{equation}
Moreover, for any $\alpha=(\alpha_{1},\alpha_{2})\in \mathbb{N}^{2}$, there exists $C_{1\alpha}>0$ such that 
\begin{equation}\label{est:P1}
\begin{aligned}
\left|\partial_{y}^{\alpha}P(y_{1},y_{2})\right|&\le C_{1\alpha}e^{-\frac{|y_{2}|}{3}},\quad \mbox{on}\ \R^{2},\\
\left|\partial_{y}^{\alpha}P(y_{1},y_{2})\right|&\le C_{1\alpha}e^{-\frac{|y|}{3}},\quad \ \mbox{on}\ (0,\infty)\times \R.
\end{aligned}
\end{equation}
For any $\alpha=(\alpha_{1},\alpha_{2})\in \mathbb{N}^{2}$ with $\alpha_{1}\ne 0$, there exists $C_{2\alpha}>0$ such that 
\begin{equation}\label{est:P2}
\left|\partial_{y}^{\alpha}P(y_{1},y_{2})\right|\le C_{2\alpha}e^{-\frac{|y|}{3}},\quad \ \mbox{on}\ \mathbb{R}^{2}.
\end{equation}
\end{lemma}
\begin{proof}
We consider $P\in C^{\infty}(\R^{2})$ of the form
\begin{equation*}
P(y_{1},y_{2})=\widetilde{P}(y_{1},y_{2})-\int_{y_{1}}^{\infty}\Lambda Q(\rho,y_{2})\dd \rho-h_{2}(y_{2})\int_{y_{1}}^{\infty}h_{1}(\rho)\dd \rho.
\end{equation*}
with $\widetilde{P}\in \mathcal{Z}(\R^{2})$.
By the definition of $h_{2}$, we see that $\partial_{y_{1}}\mathcal{L}P=\Lambda Q$ is equivalent to
\begin{equation}\label{def:R}
\partial_{y_{1}}\mathcal{L}\widetilde{P}=\Lambda Q+\partial_{y_{1}}\mathcal{L}\int_{y_1}^{+\infty}\left(\Lambda Q(\rho,y_2)+h_1(\rho)h_2(y_2)\right)\dd \rho=\partial_{y_{1}}R,
\end{equation}
where
\begin{equation*}
\begin{aligned}
R(y)=&\partial_{y_{1}}\Lambda Q-\int_{y_{1}}^{\infty}\partial_{y_{2}}^{2}\Lambda Q(\rho,y_{2})\dd \rho-3Q^{2}\int_{y_{1}}^{\infty}\Lambda Q(\rho,y_{2})\dd \rho\\
&+h'_{1}(y_{1})h_{2}(y_{2})+F''(y_{2})\int_{y_{1}}^{\infty}h_{1}(\rho)\dd \rho-3Q^{2}\int_{y_{1}}^{\infty}h_{1}(\rho)h_{2}(y_{2})\dd \rho.
\end{aligned}
\end{equation*}
First, on $(y_{1},y_{2})\in (1,\infty)\times \R$, from $\left(h_{1},h_{2},F\right)\in \mathcal{Y}(\R)\times \mathcal{Y}(\R)\times \mathcal{Y}(\R)$ and the definition of $\Lambda Q$, for any $n\in \mathbb{N}$, there exists $r_{n}>0$, such that 
\begin{equation*}
\begin{aligned}
\sum_{|\alpha|=n}\left|\partial_{y}^{\alpha}R(y)\right|&\lesssim \int^{\infty}_{y_{1}}\left(1+\left(\rho^{2}+y_{2}^{2}\right)^{\frac{r_{n}}{2}}\right)e^{-\sqrt{\rho^{2}+y_{2}^{2}}}\dd \rho\\
&+|y|^{r_{n}}e^{-|y|}
\lesssim (1+|y|)^{r_{n}+1}e^{-|y|}.
\end{aligned}
\end{equation*}
Second, on $(y_{1},y_{2})\in (-\infty,-1)\times \R$, from the definition of $h_{1}$ and $F$, we see that 
\begin{equation*}
\lim_{y_{1}\to -\infty}\left(F''(y_{2})\int_{y_{1}}^{\infty}h_{1}(\rho)\dd \rho-\int_{y_{1}}^{\infty}\partial_{y_{2}}^{2}\Lambda Q(\rho,y_{2})\dd \rho\right)=0.
\end{equation*}
Based on the above identity and the Fundamental theorem, on $(y_{1},y_{2})\in (-\infty,-1)\times \R$, for any $n\in \mathbb{N}$, there exists $r_{n}>0$ such that 
\begin{equation*}
\begin{aligned}
\sum_{|\alpha|=n}\left|\partial_{y}^{\alpha}R(y)\right|
&\lesssim \int_{-\infty}^{y_{1}}\left(1+\left(\rho^{2}+y_{2}^{2}\right)^{\frac{r_{n}}{2}}\right)e^{-\sqrt{\rho^{2}+y_{2}^{2}}}\dd \rho\\
&+|y|^{r_{n}}e^{-|y|}
\lesssim (1+|y|)^{r_{n}+1}e^{-|y|}.
\end{aligned}
\end{equation*}
Combining the above two estimates, we obtain $R\in \mathcal{Y}(\R^{2})$. 

\smallskip
On the other hand, since $R(y_{1},y_{2})$ and $Q(y_{1},y_{2})$ are both even in $y_{2}$, we have $(R,\partial_{y_{2}}Q)=0$. Then, using~\eqref{def:R}, $\partial_{y_{1}} Q\in {\rm{Ker}}\mathcal{L}$ and (ii) of Proposition~\eqref{Prop:Spectral}, 
\begin{equation*}
\left(R,\partial_{y_{1}}Q\right)=-(\Lambda Q,Q)+\left(\int_{y_1}^{+\infty}\left(\Lambda Q(\rho,y_2)+h_1(\rho)h_2(y_2)\right)\dd \rho,\mathcal{L}\partial_{y_{1}}Q\right)=0.
\end{equation*}
Therefore, from (v) of Proposition~\ref{Prop:Spectral}, there exists $\widetilde{P}\in \mathcal{Z}(\R^{2})$ such that 
\begin{equation*}
\mathcal{L}\widetilde{P}=R\ \  \mbox{with}\ \left|\left(\widetilde{P},\nabla Q\right)\right|=0\Longrightarrow \partial_{y_{1}}\mathcal{L}\widetilde{P}=\partial_{y_{1}}R\ \ \mbox{with}\ \left|\left(\widetilde{P},\nabla Q\right)\right|=0.
\end{equation*}
Moreover, from $h_{2}(y_{2})$ and $R(y_{1},y_{2})$ are even in $y_{2}$ and (v) of Proposition~\ref{Prop:Spectral}, we see that $P$ is also even in $y_{2}$.

\smallskip
Note that, from the definition of $P$ and $\widetilde{P}\in \mathcal{Z}(\R^{2})$, for any $n\in \mathbb{N}$, we have 
\begin{equation*}
\begin{aligned}
\lim_{y_{1}\to\infty}\partial_{y_{2}}^{n}P(y_{1},y_{2})
&=\lim_{y_{1}\to \infty}\partial_{y_{2}}^{n}\widetilde{P}(y_{1},y_{2})-\lim_{y_{1}\to \infty}h^{(n)}_{2}(y_{2})\int_{y_{1}}^{\infty}h_{1}(\rho)\dd \rho\\
&-\lim_{y_{1}\to \infty}\int_{y_{1}}^{\infty}\partial_{y_{2}}^{n}\Lambda Q(\rho,y_{2})\dd \rho=0.
\end{aligned}
\end{equation*}
Note also that, using~\eqref{equ:h1h2} and the definition $P$, we see that 
\begin{equation*}
\left(\partial_{y_{1}}P,Q\right)=-\left(\widetilde{P},\partial_{y_{1}}Q\right)+\left(\Lambda Q, Q\right)+\left(h_{1}\otimes h_{2},Q\right)=0.
\end{equation*}
In addition, from $P(y_{1},y_{2})$ and $Q(y_{1},y_{2})$ are even in $y_{2}$, we have $(\partial_{y_{2}}P,Q)=0$.
Next, from (ii) of Proposition~\ref{Prop:Spectral}, $\partial_{y_{1}}\mathcal{L}P=\Lambda Q$ and $\mathcal{L}P\to 0$ as $y_{1}\to \infty$, 
\begin{equation*}
\begin{aligned}
\left(P,Q\right)=-\frac{1}{2}\left(\mathcal{L}P,\Lambda Q\right)
&=\frac{1}{2}\left(\int_{y_{1}}^{\infty}\Lambda Q(\rho,y_{2})\dd \rho,\Lambda Q\right)\\
&=\frac{1}{2}\int_{\R}\left(\int_{\R^{2}}\Lambda Q(\rho,y_{2})\Lambda Q(y_{1},y_{2}){\textbf{1}}_{\{y_{1}\le \rho\}}\dd y_{1}\dd \rho\right)\dd y_{2}\\
&=\frac{1}{4}\int_{\R}\left(\int_{\R}\Lambda Q(y_{1},y_{2})\dd y_{1}\right)^{2}\dd y_{2}=\frac{1}{4}\int_{\R}|F(y_{2})|^{2}\dd y_{2}.
\end{aligned}
\end{equation*}
Last, the estimate \eqref{est:P1} with $\alpha_{1}\ne 0$, the second line of estimate~\eqref{est:P1} and the estimate~\eqref{est:P2} are direct consequences of $\partial_{y_{1}}\mathcal{L}P\in \mathcal{Z}(\R^{2})$. On the other hand, for the case of $\alpha_{1}=0$, from $(h_{1},h_{2},\Lambda Q,\widetilde{P})\in \mathcal{Y}(\R)\times \mathcal{Y}(\R)\times \mathcal{Y}(\R^{2})\times\mathcal{Z}(\R^{2})$, there exists $r_{\alpha_{2}}>0$ such that 
\begin{equation*}
\begin{aligned}
\left|\partial_{y}^{\alpha}P(y_{1},y_{2})\right|
&\lesssim (1+|y_{1}|^{r_{\alpha_{2}}}+|y_{2}|^{r_{\alpha_{2}}})e^{-\frac{\sqrt{y_{1}^{2}+y^{2}_{2}}}{2}}+(1+|y_{2}|^{r_{\alpha_{2}}})e^{-|y_{2}|}\\
&+\int^{\infty}_{y_{1}}\left(1+|\rho|^{r_{\alpha_{2}}}+|y_{2}|^{r_{\alpha_{2}}}\right)e^{-\sqrt{\rho^{2}+y^{2}_{2}}}\dd \rho\lesssim e^{-\frac{|y_{2}|}{3}}.
\end{aligned}
\end{equation*}
The proof of Lemma~\ref{le:Nonloca} is complete.
\end{proof}

\subsection{The localized profile}\label{SS:loca}
In this subsection, we introduce a localized profile to avoid the growth of $P$ as $y_{1}\to -\infty$.
Let $\phi\in C^{\infty}(\mathbb{R})$ be such that $\phi\in [0,1]$ with $\phi'\ge 0$ on $\mathbb{R}$ and
\begin{equation*}
\phi(y_{1})=
\left\{
\begin{aligned}
&0,\quad \mbox{for}\ y_{1}<-2,\\
&1,\quad \mbox{for}\ y_{1}>-1.
\end{aligned}\right.
\end{equation*}
For any $|b|\ll1$, we now define the localized profile
\begin{equation*}
\phi_{b}(y_{1})=\phi(|b|^{\frac{3}{4}}y_{1})\quad \mbox{and}\quad 
Q_{b}(y)=Q(y)+bP(y)\phi_{b}(y_{1}).
\end{equation*}

Then the following estimates related to the localized profile hold.

\begin{lemma}\label{le:Ketest}
There exists a small constant $0<b^*\ll1 $ such that for any $|b|<b^*$, the following estimates hold.
\begin{enumerate}
\item \emph{Estimate of $Q_b$}. For all $y\in\mathbb{R}^2$ and $k\in\mathbb{N}^{+}$, we have 
\begin{equation}\label{est:Qb}
\begin{aligned}
|Q_{b}(y)|&\lesssim e^{-\frac{|y|}{3}}
+|b|e^{-\frac{|y_{2}|}{3}}\mathbf{1}_{[-2,0]}(|b|^{\frac{3}{4}}y_1),\\
|\partial_{y_1}^kQ_b(y)|&\lesssim e^{-\frac{|y|}{3}}+ |b|^{1+\frac{3k}{4}}e^{-\frac{|y_{2}|}{3}}\mathbf{1}_{[-2,-1]}(|b|^{\frac{3}{4}}y_1).
\end{aligned}
\end{equation}
\item \emph{Estimate for the error term}. Let 
\begin{equation}\label{def:Psib}
\Psi_{b}=-b\Lambda Q_{b}+\partial_{y_{1}}\left(-\Delta Q_{b}+Q_{b}-Q_{b}^3\right).
\end{equation}
Then, for all $y\in\mathbb{R}^2$ and $k\in\mathbb{N}^{+}$, we have 
\begin{equation}\label{est:Psib}
\begin{aligned}
\left|\Psi_{b}(y)\right|\lesssim& |b|^{2}\left(e^{-\frac{|y|}{3}}+e^{-\frac{|y_{2}|}{3}}\mathbf{1}_{[-2,0]}(|b|^{\frac{3}{4}}y_1)\right)\\
&+|b|^{\frac{7}{4}}e^{-\frac{|y_{2}|}{3}}\mathbf{1}_{[-2,-1]}(|b|^{\frac{3}{4}}y_1),\\
|\partial_{y_1}^k\Psi_b(y)|\lesssim& |b|^{2}e^{-\frac{|y|}{3}}+|b|^{1+\frac{3}{4}(k+1)}e^{-\frac{|y_{2}|}{3}}\mathbf{1}_{[-2,-1]}(|b|^{\frac{3}{4}}y_1).
\end{aligned}
\end{equation}
\item \emph{Scalar product with $Q$}. We have 
 \begin{equation}\label{equ:PsiQ}
(\Psi_b, Q)=-\frac{b^2}{2}\int_{\mathbb{R}}\frac{|\widehat{F}(\xi)|^2}{1+|\xi|^2}\,\dd \xi+O\left(|b|^3\right).
\end{equation}
\item \emph{Energy and mass of $Q_b$}. We have 
\begin{equation}\label{est:EMQb}
\left|E(Q_b)+b(P,Q)\right|\lesssim b^2\ \ \mbox{and}\ \
\left|\int_{\R^{2}} Q_b^2\dd y-\int_{\R^{2}} Q^2\dd y-2b(P,Q)\right|\lesssim |b|^{\frac{5}{4}}.
\end{equation}
\end{enumerate}
\end{lemma}
\begin{proof}
Proof of (i). The estimate~\eqref{est:Qb} follows directly from Lemma~\ref{le:Nonloca}.

\smallskip
Proof of (ii). By~\eqref{equ:P1} and an elementary computation, we have 
\begin{equation*}
\begin{aligned}
\Psi_{b}
=&-b(1-\phi_{b})\Lambda Q+b\left(\left(\mathcal{L}P-2\partial_{y_{1}}^{2}P\right)\phi'_{b}-3(\partial_{y_{1}}P)\phi''_{b}-P\phi'''_{b}\right)\\
&-b^{2}\left(3\partial_{y_{1}}\left(QP^{2}\phi^{2}_{b}\right)+(\Lambda P)\phi_{b}+y_{1}P\phi'_{b}\right)
-b^{3}\partial_{y_{1}}\left(P^{2}\phi_{b}^{3}\right).
\end{aligned}
\end{equation*}
Then, the estimate~\eqref{est:Psib} also follows directly from Lemma~\ref{le:Nonloca}.

\smallskip 
Proof of (iii). From  the definition of $\Psi_b$ and decay properties of $P$ and $Q$, 
\begin{equation*}
(\Psi_b, Q)=-b^2\left(\partial_{y_{1}}\left(3QP^2\right)+\Lambda P,Q\right)+O\left(|b|^3\right).
\end{equation*}
Using~\eqref{equ:P1} and integration by parts, we have
\begin{equation*}
\begin{aligned}
(\Lambda P,Q)=-(P,\Lambda Q)&=\left(P, \Delta \partial_{y_{1}}P-\partial_{y_{1}}P+\partial_{y_{1}}(3Q^2P)\right)\\
&=\left(P, \Delta \partial_{y_{1}}P-\partial_{y_{1}}P\right)-\left(\partial_{y_{1}} \left(3QP^2\right),Q\right).
\end{aligned}
\end{equation*}
From~\eqref{def:h2} and the definition of $P$ in Lemma~\ref{le:Nonloca}, we know that
\begin{equation*}
\begin{aligned}
\left(P,\partial_{y_1}^3P\right)&=-\frac{1}{2}\int_{\mathbb{R}^2}\partial_{y_{1}}\left((\partial_{y_{1}}P)^{2}\right)\dd y=0,\\
\left(P,\partial_{y_{1}}P\right)&=\frac{1}{2}\int_{\mathbb{R}}\left(P^2(y_1,y_2)\bigg|^{y_1=+\infty}_{y_1=-\infty}\right)\,\dd y_2=-\frac{1}{2}\int_{\mathbb{R}}\left(F(y_2)+h_2(y_2)\right)^2\dd y_2,\\
(P,\partial_{y_2}^2\partial_{y_1}P)&=-\frac{1}{2}\int_{\mathbb{R}}\left(\left(\partial_{y_{2}}P\right)^2(y_1,y_2)\bigg|^{y_1=+\infty}_{y_1=-\infty}\right)\dd y_2=\frac{1}{2}\int_{\mathbb{R}}\left(F'(y_2)+h'_2(y_2)\right)^2\dd y_2.
\end{aligned}
\end{equation*}
Taking the Fourier transform on the both sides of~\eqref{def:h2}, we deduce that 
\begin{equation*}
\left(1+|\xi|^{2}\right)\widehat{h_{2}}(\xi)=-|\xi|^{2}\widehat{F}(\xi),\ \ \mbox{on}\ \R\Longrightarrow
\widehat{F}(\xi)+\widehat{h_{2}}(\xi)=\frac{\widehat{F}(\xi)}{1+|\xi|^2}, \ \ \mbox{on}\ \R.
\end{equation*}
Combining the above identities with the Plancherel theorem, 
\begin{equation*}
(P, \Delta \partial_{y_{1}}P-\partial_{y_{1}}P)=
\frac{1}{2}\int_{\mathbb{R}}(1+|\xi|^2)\left(\widehat{F}(\xi)+\widehat{h_{2}}(\xi)\right)^2\dd \xi=\frac{1}{2}\int_{\mathbb{R}}\frac{|\widehat{F}(\xi)|^2}{1+|\xi|^2}\,\dd \xi.
\end{equation*}
We see that~\eqref{equ:PsiQ} follows from the above identities.

\smallskip
Proof of (iv). By an elementary computation and integration by parts, we have 
\begin{equation*}
\begin{aligned}
E(Q_{b})&=E(Q)+b\left(P\phi_{b},-\Delta Q-Q^{3}\right)+O(|b|^{2}),\\
\int_{\R^{2}}Q_{b}^{2}\dd y&=\int_{\R^{2}}Q^{2}\dd y+b^{2}\int_{\R^{2}}P^{2}\phi_{b}^{2}\dd y+2b\left(P\phi_{b},Q\right).
\end{aligned}
\end{equation*}
Combining the above identities with~\eqref{est:P1}, $E(Q)=0$, $-\Delta Q+Q-Q^{3}=0$ and the decay properties of $Q$, we complete the proof of~\eqref{est:EMQb}. 
\end{proof}

\section{Modulation estimates} \label{S:Modu}
\subsection{Geometric decomposition and Bootstrap assumptions}\label{SS:Geo}
 In this subsection, we recall a standard decomposition result on solutions of~\eqref{CP} that are close to the soliton manifold. More precisely, we assume that there exist $(\overline{\lambda}(t),\overline{x}(t),\overline{\varepsilon}(t))\in(0,+\infty)\times\mathbb{R}^2\times H^{1}(\R^{2})$ such that, for all $t\in[0,t_0)$, the solution $u(t)$ of \eqref{CP} satisfies
\begin{equation}\label{est:decomposition}
u(t,x)=\frac{1}{\overline{\lambda}(t)}\left[{Q}+\bar{\varepsilon}(t)\right]\left(\frac{x-\bar{x}(t)}{\bar{\lambda}(t)}\right),\quad \mbox{with}\quad \|\bar{\varepsilon}(t)\|_{L^2}\le \kappa\le \kappa^{*},
\end{equation}
where $0<\kappa^{*}\ll1$ is small enough universal constant.

\smallskip
We now recall the following modulation result for solutions of~\eqref{CP}.
\begin{proposition}\label{Prop:decomposition}
Let $u(t)$ be a solution of~\eqref{CP} satisfying~\eqref{est:decomposition} on $[0,t_{0})$.
Then there exist $C^1$ functions $(\lambda(t),x(t),b(t))\in(0,\infty)\times\mathbb{R}^3$ such that, for all $t\in [0,t_{0})$, $\varepsilon(t)$ being defined by
\begin{equation}\label{def:var}
\varepsilon(t,y)=\lambda(t)u(t,\lambda(t)y+x(t))-Q_{b(t)}(y),
\end{equation}
it satisfies the orthogonality conditions
\begin{equation}\label{equ:orth}
(\varepsilon(t),Q)=(\varepsilon(t), Q^3)=|(\varepsilon(t),\nabla Q)|=0.
\end{equation}
Moreover, we have
\begin{equation*}
\|\varepsilon(t)\|_{L^2}+|b(t)|+\bigg|1-\frac{\overline{\lambda}(t)}{\lambda(t)}\bigg|\lesssim\delta(\kappa)\quad \mbox{and}\quad 
\|\varepsilon(t)\|_{H^1}\lesssim \delta(\|{\varepsilon}(0)\|_{H^1}).
\end{equation*}
\end{proposition}
\begin{proof}
The proof of the decomposition proposition relies on a standard argument based on Proposition~\ref{Prop:Spectral}, Lemma~\ref{le:Nonloca} and the implicit function Theorem.
For the sake of completeness, we provide a sketch here. Define the functional
\begin{equation*}
(u,\Gamma) \longmapsto \Theta(u,\Gamma):= \left((\varepsilon,Q),(\varepsilon,Q^3),(\varepsilon,\partial_{x_1} Q),(\varepsilon,\partial_{x_2} Q)\right) \in \mathbb R^{4},
\end{equation*}
where $\Gamma=\left(\lambda,x,b\right)$.
We compute the Jacobian matrix of the above mapping with respect to $(\lambda,x,b)$ and evaluate it at $(u,\lambda,x,b)=(Q,\overline{\lambda},\overline{x},0)$. Up to a rescaling and translations,
the heart of the proof is the invertibility of the Jacobian matrix:
\begin{equation*}
\mathcal{M}=\left(\begin{array}{cccc}
\left(\Lambda Q,Q^{3}\right)& \left(\Lambda Q, Q\right) &\left(\Lambda Q, \partial_{y_{1}}Q\right) & \left(\Lambda Q, \partial_{y_{2}}Q\right)\\
\left(P,Q^{3}\right)& \left(P, Q\right) &\left(P, \partial_{y_{1}}Q\right) & \left(P, \partial_{y_{2}}Q\right)\\
\left(\partial_{y_{1}}Q,Q^{3}\right)& \left(\partial_{y_{1}} Q, Q\right) &\left(\partial_{y_{1}} Q, \partial_{y_{1}}Q\right) & \left(\partial_{y_{1}}Q, \partial_{y_{2}}Q\right)\\
\left(\partial_{y_{2}} Q,Q^{3}\right)& \left(\partial_{y_{2}} Q, Q\right) &\left(\partial_{y_{2}} Q, \partial_{y_{1}}Q\right) & \left(\partial_{y_{2}} Q, \partial_{y_{2}}Q\right)
\end{array}\right).
\end{equation*}
Actually, by an elementary computation, we obtain
\begin{equation*}
\mathcal{M}=\left(\begin{array}{cccc}
\frac{1}{2}\|Q\|_{L^{4}}^{4}& 0 & 0 &0\\
\left(P,Q^{3}\right)& \frac{1}{4}\|F\|_{L^{2}}^{2} & 0 & \left(P, \partial_{y_{2}}Q\right)\\
0& 0 &\|\partial_{y_{1}} Q\|^{2}_{L^{2}} & 0\\
0 & 0 & 0 & \|\partial_{y_{2}} Q\|^{2}_{L^{2}}
\end{array}\right),
\end{equation*}
which implies $\mathcal{M}$ is invertible. See more details in the proof of~\cite[Lemma 4.4]{FHRY} and also see the proof of~\cite[Lemma 2.5]{MMR}.
\end{proof}
As usual in investigating the blow-up phenomenon of mass-critical dispersive equations, we introduce the following new time variable
\begin{equation*}
 s=\int_0^t\frac{1}{\lambda^3(\sigma)}\dd \sigma\Longleftrightarrow
 \frac{\dd s}{\dd t}=\frac{1}{\lambda^{3}(t)}.
\end{equation*} 
Recall that, we define
\begin{equation*}
\Psi_{b}=-b\Lambda Q_{b}+\partial_{y_{1}}\left(-\Delta Q_{b}+Q_{b}-Q_{b}^3\right).
\end{equation*}
In addition, we set 
\begin{equation*}
\begin{aligned}
{\rm{Mod}}&=\left(\frac{\lambda_s}{\lambda}+b\right)\left(\Lambda Q_{b}+\Lambda\varepsilon\right)
-b_s\frac{\partial Q_{b}}{\partial b}\\
&+\left(\frac{x_{1s}}{\lambda}-1\right)(\partial_{y_{1}}Q_{b}+\partial_{y_{1}}\varepsilon)+\frac{x_{2s}}{\lambda}(\partial_{y_{2}}Q_{b}+\partial_{y_{2}}\varepsilon).
\end{aligned}
\end{equation*}
We now deduce the equation of $\varepsilon$ from~\eqref{CP} and~\eqref{def:var}.
\begin{lemma}[Equation of $\varepsilon$] \label{le:equvar}
The function $\varepsilon$ satisfies
\begin{equation*}
	\partial_{s}\varepsilon=\partial_{y_{1}}\mathcal{L}\varepsilon-b\Lambda\varepsilon+
	{\rm{Mod}}+\Psi_{b}-\partial_{y_{1}}R_{b}-\partial_{y_{1}}R_{\rm NL},
	\end{equation*}
	where
	\begin{equation*}
	R_{b}=3(Q_{b}^2-Q^2)\varepsilon\quad \mbox{and}\quad 
	R_{{\rm {NL}}}=(Q_{b}+\varepsilon)^3-3Q_{b}^{2}\varepsilon-Q_{b}^3.
	\end{equation*}
\end{lemma}
\begin{proof}
We denote 
\begin{equation*}
v(t,y)=\lambda(t) u(t,\lambda(t)y+x(t)).
\end{equation*}
Using~\eqref{CP}, we see that 
\begin{equation*}
\lambda^{3}\partial_{t}v+\partial_{y_{1}}\left(\Delta v+v^{3}\right)-\lambda^{2}\lambda_{t}\Lambda v-\lambda^{2}x_{t}\cdot \nabla v=0.
\end{equation*}
Based on the above identity and the definition of the time variable $s$, we have 
\begin{equation*}
\partial_{s}v+\partial_{y_{1}}\left(\Delta v+v^{3}\right)-\frac{\lambda_{s}}{\lambda}\Lambda v-\frac{x_{s}}{\lambda}\cdot \nabla v=0.
\end{equation*}
Therefore, from $v(s,y)=\varepsilon(s,y)+Q_{b(s)}(y)$, we conclude that 
\begin{equation*}
\begin{aligned}
\partial_{s}\varepsilon
&=\partial_{y_{1}}\left(-\Delta (Q_{b}+\varepsilon)+(Q_{b}+\varepsilon)-(Q_{b}+\varepsilon)^{3}\right)+
{\rm{Mod}}-b(\Lambda Q_{b}+\Lambda \varepsilon)\\
&=\partial_{y_{1}}\mathcal{L}\varepsilon-b\Lambda\varepsilon+
	{\rm{Mod}}+\Psi_{b}-\partial_{y_{1}}R_{b}-\partial_{y_{1}}R_{\rm NL}.
\end{aligned}
\end{equation*}
\end{proof}
For $i=0,1,2$, we define the smooth function $\varphi_{i}\in C^{\infty}(\R)$ as follows,
\begin{equation*}
\begin{aligned}
\vartheta_{i}(y_1)=
\begin{cases}
\frac{1}{2},&\mbox{for}\ y_1\in (-\infty,\frac{1}{2}),\\
y_1^{i+6},&\mbox{for}\ y_1\in(1,+\infty),
\end{cases}
\quad \vartheta'(y_1)\ge0,\ \ \mbox{on}\ \mathbb{R}.
\end{aligned}
\end{equation*}
Moreover, we define the smooth even function $\zeta\in C^{\infty}(\R)$ with $\zeta\in (0,1]$ as follows,
\begin{equation*}
\begin{aligned}
\zeta(y_1)=
\begin{cases}
e^{2y_{1}},&\mbox{for}\ y_1\in(-\infty,-\frac{1}{6}),\\
1,&\mbox{for}\ y_1\in(-\frac{1}{10},\frac{1}{10}),\\
e^{-2y_{1}},&\mbox{for}\ y_1\in(\frac{1}{6},\infty),\\
\end{cases}
\quad \int_{\R}\zeta (y_{1})\dd y_{1}=1.
\end{aligned}
\end{equation*}
Let $B>100$ be a large enough universal constant to be chosen later. For $i=0,1,2$, we define the following weight function,
\begin{equation*}
\vartheta_{i,B}(y_1)=\vartheta_{i}\left(\frac{y_1}{B^{10}}\right),\quad \mbox{on}\ \R.
\end{equation*} 
We also define a smooth function $\psi_B\in C^\infty(\mathbb{R})$ such that 
\begin{equation*}
\lim_{y_1\rightarrow-\infty}\psi_B(y_1)=0 \ \ \mbox{and}\ \ 
\psi_{B}'(y_1)=
\begin{cases}
\frac{1}{B}\zeta\left(\frac{y_{1}}{B}+\frac{1}{3}-\frac{1}{2}B^{-\frac{1}{3}}\right),&\text{ for }y_1<-\frac{1}{3}B,\\
\frac{1}{B}\zeta\left(\frac{y_1}{B^{\frac{2}{3}}}+\frac{1}{3}B^{\frac{1}{3}}\right),&\text{ for }y_1\geq-\frac{1}{3}B.
\end{cases}
\end{equation*}
Last, for $i=0,1,2$, we set 
\begin{equation*}
\varphi_{i,B}(y_{1})=\sqrt{2\psi_{B}(y_{1})}\vartheta_{i,B}(y_{1}),\quad \mbox{on}\ \R.
\end{equation*}

\begin{lemma}\label{le:psiphi}
For all large enough $B>100$, the following estimates hold.
\begin{enumerate}
\item We know that $\psi_B$ is strictly increasing and $\psi_B(y_1)\to\frac{1}{2}$ as $y_{1}\to \infty$.

\item For all $y_1\in (-\infty,-B)$, we have 
\begin{equation*}
e^{\frac{2y_{1}}{B}}\le \psi_{B}(y_{1})\le 2e^{\frac{2y_{1}}{B}}\quad 
\mbox{and}\quad \frac{\sqrt{2}}{2}e^{\frac{y_{1}}{B}}\le \varphi_{i,B}(y_{1})\le e^{\frac{y_{1}}{B}}.
\end{equation*}

\item For all $y_{1}\in \left(-\frac{B}{4},\frac{B}{4}\right)$, we have
\begin{equation*}
\psi_B'(y_1)+\left|\psi_{B}(y_{1})-\frac{1}{2}\right|
+\sum_{i=1,2}\left(\varphi'_{i,B}(y_{1})+\left|\varphi_{i,B}(y_1)-\frac{1}{2}\right|\right)\lesssim e^{-\frac{1}{6}B^{\frac{1}{3}}}.
\end{equation*}
\item For all $y_1\in\mathbb{R}$ and $i=0,1,2$, we have $\psi_B(y_1) \le \varphi_{i,B}(y_1)$.
\end{enumerate}
\end{lemma}

\begin{proof}
Proof of (i). First, from $\psi'_B>0$ on $\mathbb{R}$, we know that $\psi_B$ is strictly increasing.
Then, from $\zeta$ is an even function and $\int_{\R}\zeta(y_{1})\dd y_{1}=1$, we obtain
\begin{equation*}
\int_{-\infty}^{0}\zeta(y_{1})\dd y_{1}=\int_{0}^{\infty}\zeta (y_{1})\dd y_{1}=\frac{1}{2},
\end{equation*}
which implies
\begin{align*}
&\lim_{y_1\rightarrow+\infty}\psi_B(y_1)=\int_{\mathbb{R}}\psi'_B(y_1)\,\dd y_1\\
&=\int_{-\infty}^{-\frac{B}{3}}\frac{1}{B}\zeta\left(\frac{y_{1}}{B}+\frac{1}{3}-\frac{1}{2}B^{-\frac{1}{3}}\right)\dd y_{1}
+\int_{-\frac{B}{3}}^{\infty}\frac{1}{B}\zeta\left(\frac{y_{1}}{B^{\frac{2}{3}}}+\frac{1}{3}B^{\frac{1}{3}}\right)\dd y_{1}\\
&=\int_{-\infty}^{-\frac{1}{2}B^{-\frac{1}{3}}}\zeta(y_1)\dd y_{1}+B^{-\frac{1}{3}}\int_{0}^{\infty}\zeta(y_1)\dd y_{1}=\frac{1}{2}-\int_{-\frac{1}{2}B^{-\frac{1}{3}}}^0\zeta(y_1)\,\dd y_1+\frac{1}{2}B^{-\frac{1}{3}}=\frac{1}{2}.
\end{align*}

Proof of (ii). For all $y_1<-B$, we have $\frac{y_{1}}{B}+\frac{1}{3}-\frac{1}{2}B^{-\frac{1}{3}}\le -\frac{1}{6}$. It follows that 
\begin{equation*}
\begin{aligned}
\psi_B(y_1)&=\int_{-\infty}^{y_1}\zeta\left(\frac{\rho}{B}+\frac{1}{3}-\frac{1}{2}B^{-\frac{1}{3}}\right)\dd \rho\\
&=\int_{-\infty}^{y_1}\frac{1}{B}\exp\left(\frac{2\rho}{B}+\frac{2}{3}-B^{-\frac{1}{3}}\right)\dd \rho\\
&=\frac{1}{2}\exp\left(\frac{2}{3}-B^{-\frac{1}{3}}\right)e^{\frac{2y_{2}}{B}}\in \left[e^{\frac{2y_{1}}{B}},2e^{\frac{2y_{1}}{B}}\right].
\end{aligned}
\end{equation*}
Based on the above estimate and the definition of $\varphi_{i,B}$, for $y_{1}<-B$, we find 
\begin{equation*}
\varphi_{i,B}(y_{1})=\frac{1}{2}\sqrt{2\psi_{B}(y_{1})}\in \left[\frac{\sqrt{2}}{2}e^{\frac{y_{1}}{B}},e^{\frac{y_{1}}{B}}\right].
\end{equation*}

Proof of (iii). For all $y_{1}\in \left(-\frac{B}{4},\frac{B}{4}\right)$, we have 
\begin{equation*}
\frac{y_{1}}{B^{2}}\in \left(-\frac{1}{4B},\frac{1}{4B}\right)\subset \left(-\infty,\frac{1}{2}\right)\quad \mbox{and}\quad \frac{y_{1}}{B^{\frac{2}{3}}}+\frac{1}{3}B^{\frac{2}{3}}\in\left(\frac{1}{12}B^{\frac{1}{3}},\frac{7}{12}B^{\frac{1}{3}}\right)\subset \left(\frac{1}{6},\infty\right).
\end{equation*}
It follows from the definition of $\psi_{B}$ and $\varphi_{i,B}$ that 
\begin{equation*}
\psi'_{B}(y_{1})+\sum_{i=1,2}\varphi'_{i,B}(y_{1})\lesssim \frac{1}{B}\exp \left(-\frac{2y_{1}}{B^{\frac{2}{3}}}-\frac{2}{3}B^{\frac{1}{3}}\right)\lesssim e^{-\frac{1}{6}B^{\frac{1}{3}}},
\end{equation*}
and
\begin{equation*}
\begin{aligned}
\left|\psi_{B}(y_{1})-\frac{1}{2}\right|
&\lesssim \int_{-\frac{B}{4}}^{\infty}\frac{1}{B}\exp\left(-\frac{2\rho}{B^{\frac{2}{3}}}-\frac{2}{3}B^{\frac{1}{3}}\right)\dd \rho\lesssim e^{-\frac{1}{6}B^{\frac{1}{3}}},\\
\sum_{i=1,2}\left|\varphi_{i,B}(y_{1})-\frac{1}{2}\right|&=\left|\sqrt{2\psi_{B}(y_{1})}-1\right|=\left|\frac{2\psi_{B}(y_{1})-1}{\sqrt{2\psi_{B}(y_{1})}+1}\right|\lesssim e^{-\frac{1}{6}B^{\frac{1}{3}}}.
\end{aligned}
\end{equation*}

Proof of (iv). From $0<\psi_{B}<\frac{1}{2}$ on $\R$ and the definition of $\psi_{B}$ and $\varphi_{i,B}$, we complete the proof of (iv).
\end{proof}

Based on the above lemma, we obtain the following technical lemma related to the pointwise estimates $\psi_{B}$ and $\varphi_{i,B}$ and their derivatives.
\begin{lemma}\label{le:psiphi2}
The following estimates hold.
\begin{enumerate}
\item \emph{First-type estimates of derivatives of $\psi_{B}$}. We have 
\begin{equation*}
B^{\frac{2}{3}}|\psi''_{B}|+B^{\frac{4}{3}}|\psi'''_{B}|\lesssim \psi'_{B},\quad \mbox{on}\ \R.
\end{equation*}

\item \emph{Second-type estimates of derivatives of $\psi_{B}$.} For $i=1,2$, we have 
\begin{equation*}
\sqrt{B\psi'_{B}}\lesssim B\varphi'_{i,B}+\psi_{B},\quad \mbox{on}\ \R.
\end{equation*}

\item {\emph{Third-type estimates of derivatives of $\psi_{B}$.}} For $i=1,2$, we have
\begin{equation*}
|y_{1}|\psi'_{B}\lesssim \sqrt{\psi_{B}}\lesssim B\varphi'_{i,B}+\psi_{B},\quad \mbox{on}\ \R.
\end{equation*}

\item \emph{First-type estimates of derivatives of $\varphi_{i,B}$}. For $i=1,2$, we have 
\begin{equation*}
\begin{aligned}
|\varphi''_{i,B}|&\lesssim B^{-\frac{2}{3}}\varphi'_{i,B}+B^{-20}\psi_{B},\quad \mbox{on}\ \R,\\
|\varphi'''_{i,B}|&\lesssim B^{-\frac{4}{3}}\varphi'_{i,B}+B^{-30}\psi_{B},\quad \mbox{on}\ \R.
\end{aligned}
\end{equation*}

\item {\emph{Second-type estimates on derivatives of $\varphi_{i,B}$.}} For $i=1,2$, we have 
\begin{equation*}
B\varphi'_{i,B}+\psi_{B}\lesssim  \varphi_{i-1,B}\lesssim B^{10}\varphi'_{i,B}+\psi_{B},\quad \mbox{on}\ \R.
\end{equation*}

\item {\emph{Third-type  estimates on derivatives of $\varphi_{i,B}$.}} For $i=1,2$, we have 
\begin{equation*}
\varphi_{i,B}\lesssim B\varphi'_{i,B}+\psi_{B}+|y_{1}|\varphi'_{i,B}\mathbf{1}_{[B^{10},\infty)},\quad \mbox{on}\ \R.
\end{equation*}
\end{enumerate}
\end{lemma}

\begin{proof}
Proof of (i). From the definition of $\psi_{B}$, we see that 
\begin{equation*}
\psi_{B}''(y_1)=
\begin{cases}
\frac{1}{B^{2}}\zeta'\left(\frac{y_{1}}{B}+\frac{1}{3}-\frac{1}{2}B^{-\frac{1}{3}}\right),&\text{ for }y_1<-\frac{1}{3}B,\\
\frac{1}{B^{\frac{5}{3}}}\zeta'\left(\frac{y_1}{B^{\frac{2}{3}}}+\frac{1}{3}B^{\frac{1}{3}}\right),&\text{ for }y_1\geq-\frac{1}{3}B,
\end{cases}
\end{equation*}
\begin{equation*}
\psi_{B}'''(y_1)=
\begin{cases}
\frac{1}{B^{3}}\zeta''\left(\frac{y_{1}}{B}+\frac{1}{3}-\frac{1}{2}B^{-\frac{1}{3}}\right),&\text{ for }y_1<-\frac{1}{3}B,\\
\frac{1}{B^{\frac{7}{3}}}\zeta''\left(\frac{y_1}{B^{\frac{2}{3}}}+\frac{1}{3}B^{\frac{1}{3}}\right),&\text{ for }y_1\geq-\frac{1}{3}B.
\end{cases}
\end{equation*}
We see that the estimate in (i) directly follows from the above identities.

Proof of (ii) and (iii). Using again the definition of $\psi_{B}$ and $\varphi_{i,B}$, we have 
\begin{equation*}
B\psi_{B}'(y_1)=
\begin{cases}
\zeta\left(\frac{y_{1}}{B}+\frac{1}{3}-\frac{1}{2}B^{-\frac{1}{3}}\right),&\text{ for }y_1<-\frac{1}{3}B,\\
\zeta\left(\frac{y_1}{B^{\frac{2}{3}}}+\frac{1}{3}B^{\frac{1}{3}}\right),&\text{ for }y_1\geq-\frac{1}{3}B,
\end{cases}
\end{equation*}
\begin{equation*}
|y_{1}|\psi_{B}'(y_1)=
\begin{cases}
\frac{|y_{1}|}{B}\zeta\left(\frac{y_{1}}{B}+\frac{1}{3}-\frac{1}{2}B^{-\frac{1}{3}}\right),&\text{ for }y_1<-\frac{1}{3}B,\\
\frac{|y_{1}|}{B}\zeta\left(\frac{y_1}{B^{\frac{2}{3}}}+\frac{1}{3}B^{\frac{1}{3}}\right),&\text{ for }y_1\geq-\frac{1}{3}B,
\end{cases}
\end{equation*}
and
\begin{equation*}
\varphi'_{i,B}(y_{1})
=\frac{\psi'_{B}(y_{1})}{\sqrt{2\psi_{B}(y_{1})}}\vartheta_{i}\left(\frac{y_{1}}{B^{10}}\right)
+\frac{1}{B^{10}}\sqrt{2\psi_{B}(y_{1})}\vartheta'_{i}\left(\frac{y_{1}}{B^{10}}\right).
\end{equation*}
Therefore, from (i) of Lemma~\ref{le:psiphi}, we complete the proof of (ii) and (iii).

Proof of (iv). Using again the definition of $\varphi_{i,B}$, we have 
\begin{equation*}
\begin{aligned}
\frac{\varphi''_{i,B}(y_{1})}{\sqrt{2\psi_{B}(y_{1})}}
&=\left(\frac{\psi''_{B}}{2\psi_{B}}-\left(\frac{\psi'_{B}}{2\psi_{B}}\right)^{2}\right)\vartheta_{i}\left(\frac{y_{1}}{B^{10}}\right)\\
&+\frac{2}{B^{10}}\frac{\psi'_{B}}{2\psi_{B}}\vartheta_{i}'\left(\frac{y_{1}}{B^{10}}\right)+\frac{1}{B^{20}}\vartheta''_{i}\left(\frac{y_{1}}{B^{10}}\right),\qquad \qquad \qquad \qquad \qquad 
\end{aligned}
\end{equation*}
\begin{equation*}
\begin{aligned}
\frac{\varphi'''_{i,B}(y_{1})}{\sqrt{2\psi_{B}(y_{1})}}
&=
\left(\frac{\psi'''_{B}(y_{1})}{2\psi_{B}(y_{1})}-3\frac{\psi'_{B}(y_{1})\psi''_{B}(y_{1})}{(2\psi_{B}(y_{1}))^{2}}+3\left(\frac{\psi'_{B}(y_{1})}{2\psi_{B}(y_{1})}\right)^{3}\right)\vartheta_{i}\left(\frac{y_{1}}{B^{10}}\right)\\
&+\frac{3}{B^{10}}\left(\frac{\psi''_{B}(y_{1})}{2\psi_{B}(y_{1})}-\left(\frac{\psi'_{B}(y_{1})}{2\psi_{B}(y_{1})}\right)^{2}\right)\vartheta'_{i}\left(\frac{y_{1}}{B^{10}}\right)\\
&+\frac{3}{B^{20}}\frac{\psi'_{B}(y_{1})}{2\psi_{B}(y_{1})}\vartheta''_{i}\left(\frac{y_{1}}{B^{10}}\right)+\frac{1}{B^{30}}\vartheta'''_{i}\left(\frac{y_{1}}{B^{10}}\right).
\end{aligned}
\end{equation*}
Therefore, using again the definition of $\psi_{B}$ and $\varphi_{i,B}$, we complete the proof of (iv).

Proof of (v)--(vi). Combining the above identities with (i) and (ii) of Lemma~\ref{le:psiphi}, we complete the proof of (v)--(vi).
\end{proof}

Recall that, from the definition of $Q_{b}$, we have 
\begin{equation*}
\frac{\partial Q_{b}}{\partial {b}}(y)=P(y)\frac{\partial}{\partial{b}}
\left(b\phi\left(|b|^{\frac{3}{4}}y_{1}\right)\right)=
P(y)\left(\phi+\frac{3}{4}y_{1}\phi'\right)\left(|b|^{\frac{3}{4}}y_{1}\right).
\end{equation*}
It follows from Lemma~\ref{le:Nonloca} that
\begin{equation*}
\sum_{|\alpha|=0}^{2}\left|\partial_{y}^{\alpha}\frac{\partial Q_{b}}{\partial b}\right|\lesssim 
e^{-\frac{|y_{2}|}{3}}\mathbf{1}_{[-2,0]}(|b|^{\frac{4}{3}}y_{1})+e^{-\frac{|y|}{3}}\mathbf{1}_{[0,\infty)}(y_{1}).
\end{equation*}

Based on the above estimate, Lemma~\ref{le:Nonloca}, Lemma~\ref{le:psiphi} and Lemma~\ref{le:psiphi2}, we obtain the following pointwise estimates.

\begin{lemma}\label{le:psiphi3}
The following estimates hold.
\begin{enumerate}
\item \emph{First-type weighted estimates for $Q_{b}$.} It holds
\begin{equation*}
\begin{aligned}
&\sum_{|\alpha|=0}^{2}\left|\partial_{y}^{\alpha}Q_{b}\right|
\left(\psi_B'+\varphi'_{i,B}+|\psi_B-\varphi_{i,B}|\right)\\
&\lesssim e^{-\frac{|y_{2}|}{4}} \left(B^{-30}+|b|\right)\left(B\varphi'_{i,B}+\psi_{B}\right).
\end{aligned}
\end{equation*}

\item \emph{Second-type weighted estimates for $Q_{b}$.} It holds
\begin{equation*}
\begin{aligned}
&\sum_{|\alpha|=0}^{2}\left|\partial_{y}^{\alpha}\frac{\partial{Q}_{b}}{\partial{b}}\right|\left(\psi_{B}+\psi'_{B}+\varphi_{i,B}\right)\\
&\lesssim \left(e^{-\frac{|y_{2}|}{4}}\mathbf{1}_{[-2,0]}(|b|^{\frac{4}{3}}y_{1})+e^{-\frac{|y|}{4}}\mathbf{1}_{[0,\infty)}(y_{1})\right)\left(B\varphi'_{i,B}+\psi_{B}\right).
\end{aligned}
\end{equation*}

\item \emph{Third-type weighted estimates for $Q_{b}$.} For any $\Gamma\in \left\{\Lambda,\nabla \right\}$, we have 
\begin{equation*}
\begin{aligned}
&\sum_{|\alpha|=0}^{2}\left|\partial_{y}^{\alpha}(\Gamma Q_{b}-\Gamma Q)\right|\left(\psi_{B}+\psi'_{B}+\varphi_{i,B}\right)\\
&\lesssim
 |b|\left(e^{-\frac{|y_{2}|}{4}}\mathbf{1}_{[-2,0]}(|b|^{\frac{4}{3}}y_{1})+e^{-\frac{|y|}{4}}\mathbf{1}_{[0,\infty)}(y_{1})\right)\left(B\varphi'_{i,B}+\psi_{B}\right).
\end{aligned}
\end{equation*}

\item\emph{Weighted estimates for $\Psi_{b}$.} It holds
\begin{equation*}
\begin{aligned}
&\sum_{|\alpha|=0}^{2}\left|\partial_{y}^{\alpha}\Psi_{b}\right|
\left(\psi_B'+\varphi'_{i,B}+|\psi_B-\varphi_{i,B}|\right)\\
&\lesssim e^{-\frac{|y_{2}|}{4}} \left(B^{-30}+|b|\right)\left(B\varphi'_{i,B}+\psi_{B}\right).
\end{aligned}
\end{equation*}

\end{enumerate}

\end{lemma}

\begin{proof}
The proof of the above estimates relies on an argument based on~Lemma~\ref{le:Ketest}, Lemma~\ref{le:psiphi} and Lemma~\ref{le:psiphi2}, and we omit it.
\end{proof}

For $i=0,1,2$, we now define the following weighted $H^1$ norms of $\varepsilon$,
\begin{equation*}
\mathcal{N}_i(s)=\int_{\mathbb{R}^2}\left(|\nabla\varepsilon(s,y)|^2\psi_B(y_1)
+|\varepsilon(s,y)|^{2}\varphi_{i,B}(y_1)\right)\dd y.
\end{equation*}

Let $u(t)$ be a solution of \eqref{CP} satisfying \eqref{est:decomposition} on $[0,t_0]$, and hence the geometrical decomposition in Proposition~\ref{Prop:decomposition} holds on $[0,t_0]$. Let $0<\kappa\ll 1$ be a small enough universal constant. We denote by $s_0=s(t_0)$ and assume the following priori bounds hold for all $s\in[0,s_0]$:
\begin{enumerate}
\item [{\rm(i)}] Scaling invariant bounds. We assume
\begin{equation}\label{est:Boot1}
|b(s)|+\mathcal{N}_2(s)+\|\varepsilon(s)\|_{L^2}\leq\kappa.
\end{equation}
\item [{\rm(ii)}] Bounds related to $H^{\frac{\theta}{2}}$ scaling. We assume
\begin{equation}\label{est:Boot2}
\frac{|b(s)|+\mathcal{N}_2(s)}{\lambda^{\theta}(s)}\le\kappa.
\end{equation}
Here, the value of $\theta$ is given by~\eqref{eq:theta}.
\item [{\rm{(iii)}}]Decay assumption on the $y_1$-variable. We assume
\begin{equation}\label{est:Boot3}
\int_{\R}\int_{0}^{\infty}y_1^{100}\varepsilon^2(s,y)\dd y_{1}\dd y_{2}\leq 10\bigg(1+\frac{1}{\lambda^{100}(s)}\bigg).
\end{equation}
\end{enumerate}

We mention here that, the bootstrap assumption \eqref{est:Boot1}--\eqref{est:Boot3} plays a crucial role in our proof. See more details related to the bootstrap assumption in Section~\ref{S:Mono}--\ref{S:Endproof}.

\subsection{Modulation estimates}
In this subsection, we deduce the modulation estimates for the geometric parameters from the equation of $\varepsilon$ and the orthogonal conditions. Recall that, in this article, we still assume that $u(t)$ is a solution of~\eqref{CP} which satisfies~\eqref{est:decomposition} on $[0,t_{0}]$ and thus admits on $[0,t_{0}]$ a decomposition~\eqref{def:var} as in Proposition~\ref{Prop:decomposition}. Recall also that, we always denote $s_{0}=s(t_{0})$.

We start with the following standard energy and modulation estimates.
\begin{lemma}\label{le:modu1}
Assume that for all $s\in[0,s_0]$, the solution $u(t)$ with initial data $u_{0}$ satisfies the bootstrap assumption \eqref{est:Boot1}--\eqref{est:Boot3}.
Then the following estimates hold.
\begin{enumerate}
	\item \emph{Estimate induced by the conservation law.} We have 
	\begin{equation}\label{CL}
	\begin{aligned}
	\|\varepsilon\|_{L^2}^{2}
	&\lesssim|b|+\left|\int_{\mathbb{R}^2} \left(u^2_0- Q^2\right)\dd y\right|,\\
  	\left|2\lambda^2 E(u_0)-\|\nabla \varepsilon\|_{L^2}^2\right| &\lesssim|b|
  	+\int_{\R^{2}} \varepsilon^2e^{-\frac{|y|}{10}}\dd y+(\|\varepsilon\|_{L^2}^2+|b|^{\frac{1}{4}})\|\nabla \varepsilon\|_{L^2}^2.
  	\end{aligned}
	\end{equation}
	\item \emph{Standard modulation estimates}. We have
 	\begin{equation*}
 	\begin{aligned}
 	|b_s|&\lesssim b^2+\int_{\R^{2}} \varepsilon^2e^{-\frac{|y|}{10}}\dd y,\\
	\left|\frac{\lambda_s}{\lambda}+b\right|+\left|\frac{x_{1s}}{\lambda}-1\right|
	&+\left|\frac{x_{2s}}{\lambda}\right|
	\lesssim b^{2}+\left(\int_{\R^{2}} \varepsilon^2e^{-\frac{|y|}{10}}\dd y\right)^{\frac{1}{2}}.
	\end{aligned}
	\end{equation*}
	    \item \emph{Weighted $L^1$ estimate}. For all $f\in\mathcal{Z}(\R^{2})$, we have
        \begin{equation*}
         \int_{\mathbb{R}^2}\left(\left|\varepsilon(y_1,y_2)\right|\left|\int_{-\infty}^{y_1}f(\rho,y_2)\dd \rho\right| \right)\dd y_1\dd y_2\lesssim \left(B^{10}\int_{\R^{2}}\varepsilon^2\varphi_{0,B}\dd y\right)^{\frac{1}{2}}.
        \end{equation*}
	\end{enumerate}
	\end{lemma}
	
	\begin{proof}
	Proof of (i). First, from the mass conservation law, we find
	\begin{equation*}
	\begin{aligned}
	\int_{\R^{2}}u_{0}^{2}\dd y
	&=\int_{\R^{2}}Q_{b}^{2}\dd y-\int_{\R^{2}}Q^{2}\dd y-2b(P,Q)\\
	&+\int_{\R^{2}}\varepsilon^{2}\dd y+2b(P,Q)+\int_{\R^{2}}Q^{2}\dd y+2b(\varepsilon,\phi_{b}P).
	\end{aligned}
	\end{equation*}
	Combining the above identity with~\eqref{est:EMQb}, $\|\phi_{b}P\|_{L^{2}}^{2}\lesssim |b|^{-\frac{3}{4}}$ and the Cauchy-Schwarz inequality, we obtain the first estimate in (i).
	
 Second, from the energy conservation law, we find
\begin{equation*}
\begin{aligned}
2\lambda^2 E(u_0)=&\int_{\R^{2}}|\nabla \varepsilon|^{2}\dd y+2\int_{\R^{2}}\varepsilon\left(-\Delta (Q_{b}-Q)-(Q_{b}^{3}-Q^{3})\right)\dd y-2b(P,Q)\\
&+\left(2E(Q_{b})+2b(P,Q)\right)-\frac{1}{2}\int_{\R^{2}}\left((Q_{b}+\varepsilon)^{4}-Q_{b}^{4}-4Q_{b}^{3}\varepsilon\right)\dd y.
\end{aligned}
\end{equation*}
Note that 
\begin{equation*}
\begin{aligned}
&-\Delta (Q_{b}-Q)=-b\phi_{b}\Delta P-2b\phi'_{b}\partial_{y_{1}}P-b\phi''_{b}P\\
&-(Q_{b}^{3}-Q^{3})=-3b\phi_{b}PQ^{2}-3b^{2}\phi_{b}^{2}P^{2}Q-b^{3}\phi_{b}^{3}P^{3}.
\end{aligned}
\end{equation*}
Hence, from~\eqref{est:Boot1} and the decay property of $P$ in Lemma~\ref{le:Nonloca}, we deduce that 
\begin{equation*}
\begin{aligned}
&\left|\int_{\R^{2}}\varepsilon\left(-\Delta (Q_{b}-Q)\right)\dd y\right|\\
&\lesssim |b|\|\nabla \varepsilon\|_{L^2}\|\nabla(\phi_bP)\|_{L^2}\lesssim |b|^{\frac{5}{8}}\|\nabla\varepsilon\|_{L^2}\lesssim |b|+|b|^{\frac{1}{4}}\|\nabla\varepsilon\|_{L^2}^2
\end{aligned}
\end{equation*}
and
\begin{align*}
&\left|\int_{\R^{2}}\varepsilon\left(-(Q_{b}^{3}-Q^{3})\right)\dd y\right|\\
&\lesssim|b|\left(\int_{\mathbb{R}^2}\varepsilon^2e^{-\frac{|y|}{10}}\right)^{\frac{1}{2}}+|b|^3\|\varepsilon\|_{L^2}\|\phi_b^3P^3\|_{L^2}\lesssim |b|
  	+\int_{\R^{2}} \varepsilon^2e^{-\frac{|y|}{10}}\dd y.
\end{align*}
Next, from the Gagliardo-Nirenberg's inequality and the definition of $Q_{b}$, 
\begin{align*}
&\left|\int_{\R^{2}}\left((Q_{b}+\varepsilon)^4-Q_{b}^4-4\varepsilon Q_{b}^3\right)\dd y\right|\\
&\lesssim|b|\int_{\R^{2}}\varepsilon^2\dd y+
\int_{\R^{2}} \varepsilon^2Q^2\dd y+\int_{\R^{2}}\varepsilon^4\dd y\lesssim \int_{\R^{2}}\varepsilon^2e^{-\frac{|y|}{10}}\dd y+|b|+\|\varepsilon\|_{L^2}^2\|\nabla\varepsilon\|_{L^2}^2.
\end{align*}
Combining the above estimates, we obtain the second estimate in (i).
	
	\smallskip
	Proof of (ii). First, differentiating the orthogonality conditions $\left(\varepsilon,Q^{3}\right)=|(\varepsilon,\nabla Q)|=0$ in \eqref{equ:orth} and then using~\eqref{est:Psib} and~Lemma~\ref{le:equvar}, we obtain
\begin{equation*}
\begin{aligned}
&\left(1+O(b)+O\left(\|\varepsilon e^{-\frac{|y|}{6}})\|_{L^{2}}\right)\right)\left(\left|\frac{\lambda_s}{\lambda}+b\right|
+\left|\frac{x_{1s}}{\lambda}-1\right|
+\left|\frac{x_{2s}}{\lambda}\right|\right)\\
&\lesssim b^{2}+|b_s|+\int_{\R^{2}}|\varepsilon|e^{-\frac{|y|}{3}}\dd y +\int_{\R^{2}}\varepsilon^2e^{-\frac{|y|}{3}}\dd y+\int_{\R^{2}}\varepsilon^3e^{-\frac{|y|}{3}}\dd y.
\end{aligned}
\end{equation*}
From the 2D Sobolev embedding inequality and~\eqref{est:Boot1}, we see that
\begin{equation*}
\begin{aligned}
\left(\int_{\R^{2}}\varepsilon^3e^{-\frac{|y|}{3}}\dd y\right)
&\lesssim \|\varepsilon e^{-\frac{|y|}{9}}\|_{L^{2}}\|\varepsilon e^{-\frac{|y|}{9}}\|^{2}_{L^{4}}\\
&\lesssim\|\varepsilon e^{-\frac{|y|}{9}}\|_{L^2}^2
\|\nabla\left(\varepsilon e^{-\frac{|y|}{9}}\right)\|_{L^2}\\
&\lesssim \left(\int_{\R^{2}} \varepsilon^2e^{-\frac{|y|}{10}}\dd y\right)\mathcal{N}_0^{\frac{1}{2}}\lesssim\int_{\R^{2}} \varepsilon^2e^{-\frac{|y|}{10}}\dd y.
\end{aligned}
\end{equation*}
Next, differentiating the orthogonality condition $\left(\varepsilon, Q\right)=0$ and then using~\eqref{equ:PsiQ} and the fact that $\left(\partial_{y_{1}}\mathcal{L}\varepsilon,Q\right)=\left(\varepsilon, \mathcal{L}\partial_{y_{1}}Q\right)=\left(\Lambda Q,Q\right)=\left|\left(\nabla Q,Q\right)\right|=0$, we obtain
\begin{equation*}
|b_s|\lesssim  |b|\left(\left|\frac{\lambda_s}{\lambda}+b\right|+\left|\frac{x_{1s}}{\lambda}-1\right|+\left|\frac{x_{2s}}{\lambda}\right|+|b|\right)+\int_{\R^{2}}\varepsilon^2e^{-\frac{|y|}{3}}\dd y
+\int_{\R^{2}}\varepsilon^3e^{-\frac{|y|}{3}}\dd y.
\end{equation*}
Combining the above estimates, we complete the proof of (ii).
	
	\smallskip
	Proof of (iii). Note that, for any $f\in \mathcal{Z}(\R^{2})$, we have 
	\begin{equation*}
	\left|\int_{-\infty}^{y_{1}}f(\rho,y_{2})\dd \rho\right|\lesssim 
	 \int_{-\infty}^{y_{1}}e^{-\frac{\sqrt{\rho^{2}+y^{2}_{2}}}{4}}\dd \rho\lesssim \left(e^{-\frac{|y_{1}|}{10}}{\textbf{1}}_{(-\infty,0)}(y_{1})+{\textbf{1}}_{[0,\infty)}(y_{1})\right)e^{-\frac{|y_{2}|}{10}}.
	\end{equation*}
	On the other hand, from the definition of $\varphi_{0,B}$, we have 
	\begin{equation*}
	e^{\frac{y_{1}}{B}}\textbf{1}_{(-\infty,0)}(y_{1})+\left(1+\left(\frac{y_{1}}{B^{10}}\right)^{6}\right)\textbf{1}_{[0,\infty)}(y_{1})\lesssim \varphi_{0,B}(y_{1}),\quad \mbox{on}\ \R.
	\end{equation*}
	It follows that 
	\begin{equation*}
	\frac{\left|\int_{-\infty}^{y_{1}}f(\rho,y_{2})\dd \rho\right|^{2}}{\varphi_{0,B}(y_{1})}\lesssim e^{-\frac{|y|}{10}}\textbf{1}_{(-\infty,0)}(y_{1})
	+\frac{e^{-\frac{|y_{2}|}{10}}}{\left(1+\left(\frac{y_{1}}{B^{10}}\right)^{6}\right)}\textbf{1}_{[0,\infty)}(y_{1}).
	\end{equation*}
	Combining the above estimates with the Cauchy-Schwarz inequality, we conclude that 
	\begin{equation*}
	\begin{aligned}
	 &\int_{\mathbb{R}^2}\left(\left|\varepsilon(y_1,y_2)\right|\left|\int_{-\infty}^{y_1}f(\rho,y_2)\dd \rho\right| \right)\dd y_1\dd y_2\\
	 &\lesssim \left\|\varepsilon\sqrt{\varphi_{0,B}}\right\|_{L^{2}}\left(\int_{\R^{2}}\frac{\left|\int_{-\infty}^{y_{1}}f(\rho,y_{2})\dd \rho\right|^{2}}{\varphi_{0,B}(y_{1})}\dd y_{1}\dd y_{2}\right)^{\frac{1}{2}}
	 \lesssim \left(B^{10}\int_{\R^{2}}\varepsilon^{2}\varphi_{0,B}\dd y\right)^{\frac{1}{2}}.
	 \end{aligned}
	\end{equation*}
	\end{proof}
	
	Note that, for any $f\in \mathcal{Z}(\R^{2})$, from Lemma~\ref{le:equvar} and integration by parts, we have 
	\begin{equation*}
	\begin{aligned}
	\frac{\dd }{\dd s}\left(\varepsilon(s),\int_{-\infty}^{y_{1}}f(\rho,y_{2})\dd \rho\right)
	&=-\left(\varepsilon,\mathcal{L}f\right)
	+\left(\Psi_{b},\int_{-\infty}^{y_{1}}f(\rho,y_{2})\dd \rho\right)+(R_{b},f)\\
	&+\left({\rm{Mod}}-b\Lambda \varepsilon,\int_{-\infty}^{y_{1}}f(\rho,y_{2})\dd \rho\right)+\left(R_{NL},f\right).
	\end{aligned}
	\end{equation*}
	Therefore, from~\eqref{est:P1},~\eqref{est:P2},~\eqref{est:Boot1}, Lemma~\ref{le:Ketest}, Lemma~\ref{le:equvar}, Lemma~\ref{le:modu1} and the 2D Sobolev embedding inequality, we see that 
	\begin{equation}\label{equ:varf}
	\begin{aligned}
	&\frac{\dd }{\dd s}\left(\varepsilon(s),\int_{-\infty}^{y_{1}}f(\rho,y_{2})\dd \rho\right)\\
	=&-\left(\varepsilon,\mathcal{L}f\right)+\left(\frac{\lambda_{s}}{\lambda}+b\right)\left(\Lambda Q, \int_{-\infty}^{y_{1}}f(\rho,y_{2})\dd \rho\right)\\
	&+\frac{x_{2s}}{\lambda}\left(\partial_{y_{2}}Q,\int_{-\infty}^{y_{1}}f(\rho,y_{2})\dd \rho\right)
	-b_{s}\left(\frac{\partial Q_{b}}{\partial b },\int_{-\infty}^{y_{1}}f(\rho,y_{2})\dd \rho\right)\\
	&-\left(\frac{x_{1s}}{\lambda}-1\right)\left(Q,f\right)+O\left(B^{5}b^{2}+B^{5}\int_{\R^{2}}\varepsilon^{2}\varphi_{0,B}\dd y\right).
	\end{aligned}
	\end{equation}
	
	Last, we deduce the refined modulation estimates for the geometric parameters from the above identity and Lemma~\ref{le:modu1}.
	\begin{lemma}\label{le:modu2}
	In the context of Lemma~\ref{le:modu1}, the following estimates hold.
	\begin{enumerate}  
   
		\item \emph{Law of $\lambda$.} Let
		\begin{equation*}
		\sigma_{1}(y)=\frac{1}{\|F\|^2_{L^2}}\int_{-\infty}^{y_1}\Lambda Q(\rho,y_2)\,\dd \rho\quad 
		\mbox{and}\quad  J_1(s)=(\varepsilon(s),\sigma_{1}).
		\end{equation*}
		Then we have
		\begin{equation*}
		\left|\frac{\lambda_{s}}{\lambda}+b-2J_{1s}\right|\lesssim B^{5}b^2+B^{5}\int_{\R^{2}}\varepsilon^2\varphi_{0,B}\dd y.
		\end{equation*}
		\item \emph{Law of $b$}. Let
		\begin{equation*}
		 \begin{aligned}
		\sigma_{2}(y)&=\frac{1}{(P,Q)}\left(P(y)+F(y_{2})+h_{2}(y_{2})\right)\\
		&+\frac{(\Lambda P,Q)}{(P,Q)(Q^3,\Lambda Q)}Q^{3}(y)-c_1\int_{-\infty}^{y_{1}}\Lambda Q(\rho,y_{2})\dd \rho,
		\end{aligned}
		\end{equation*}
		where $F$ and $h_{2}$ are defined in~\eqref{equ:defF} and~\eqref{def:h2} respectively. In addition, the constant $c_1\in  \mathbb{R}$ is chosen to ensure that
		\begin{equation*}
		\left(\sigma_{2},\Lambda Q\right)=\frac{1}{(P,Q)}\left(F+h_{2},\Lambda Q\right)-c_{1}\left(\int_{-\infty}^{y_{1}}\Lambda Q(\rho,y_{2})\dd \rho, \Lambda Q\right)=0.
		\end{equation*}
		We set 
		\begin{equation*}
		g_{2}(y)=\partial_{y_{1}}\sigma_{2}(y)\quad \mbox{and}\quad 
		J_2(s)=(\varepsilon(s),\sigma_{2}).
        \end{equation*}
		Then we have
		\begin{equation*}
		|b_s+\theta b^2+bJ_{2s}|\lesssim  B^{5}|b|^3+\left(B^{5}|b|+1\right)\int_{\R^{2}}\varepsilon^2\varphi_{0,B}\dd y.
		\end{equation*}
		Here, the value of $\theta$ is given by \eqref{eq:theta}.
		\item \emph{Law of $\frac{b}{\lambda^{\theta}}$}. Let
		\begin{equation*}
		\sigma=2\theta\sigma_{1}+\sigma_{2}\quad \mbox{and}\quad J(s)=(\varepsilon(s),\sigma).
        \end{equation*}
		Then we have
		\begin{equation*}
		\left|\frac{\dd}{\dd s}\left(\frac{b}{\lambda^{\theta}}\right)+\frac{b}{\lambda^{\theta}}J_s\right|\lesssim \frac{1}{\lambda^{\theta}}\left(B^{5}|b|^3+\left(B^{5}|b|+1\right)\int_{\R^{2}}\varepsilon^2\varphi_{0,B}\dd y\right).
		\end{equation*}
        \item \emph{Law of $x_2$}. Let
        \begin{equation*}
        \sigma_{3}(y)=\frac{1}{c_{2}}\int_{-\infty}^{y_1}\partial_{y_{2}}Q(\rho,y_2)\dd \rho \quad \mbox{and}\quad J_3(s)=(\varepsilon(s),\sigma_{3}),
        \end{equation*}
        where
        \begin{equation*}
        c_2
        =\frac{1}{2}\int_{\mathbb{R}}\left(\int_{\mathbb{R}}\partial_{y_{2}}Q(y_1,y_2)\dd y_{1}\right)^2 \dd y_{2}.
        \end{equation*}
     Then we have
        \begin{equation*}
        \left|\frac{x_{2s}}{\lambda}-J_{3s}\right|\lesssim B^{5}b^2+B^{5}\int_{\R^{2}}\varepsilon^2\varphi_{0,B}\dd y.
        \end{equation*}
	\end{enumerate}
\end{lemma}

\begin{proof}
Proof of (i). Let $f=\Lambda Q$ in~\eqref{equ:varf}. Note that 
\begin{equation*}
\mathcal{L}\Lambda Q=-2Q\quad \mbox{and}\quad 
\left(\Lambda Q,Q\right)=\left(\partial_{y_{2}}Q,\int_{-\infty}^{y_{1}}\Lambda Q(\rho,y_{2})\dd \rho\right)=0.
\end{equation*}
Therefore, from~\eqref{equ:varf} and Lemma~\ref{le:modu1}, we obtain
\begin{equation*}
\begin{aligned}
\|F\|_{L^{2}}^{2}J_{1s}
&=\left(\frac{\lambda_{s}}{\lambda}+b\right)\left(\Lambda Q,\int_{-\infty}^{y_{1}}\Lambda Q(\rho,y_{2})\dd \rho\right)\\
&+O\left(B^{5}b^2+B^{5}\int_{\R^{2}}\varepsilon^2\varphi_{0,B}\dd y\right)\\
&=\frac{1}{2}\|F\|_{L^{2}}^{2}\left(\frac{\lambda_{s}}{\lambda}+b\right)+O\left(B^{5}b^2+B^{5}\int_{\R^{2}}\varepsilon^2\varphi_{0,B}\dd y\right),
\end{aligned}
\end{equation*}
which completes the proof of (i).

\smallskip
Proof of (ii). We claim that 
\begin{equation}\label{est:bsb2}
\begin{aligned}
&b_s+\theta b^2-\frac{b}{(P,Q)}\frac{(\Lambda P,Q)}{(Q^3,\Lambda Q)}(\varepsilon,\mathcal{L}\partial_{y_{1}}Q^3)\\
&-\frac{b}{(P,Q)}(\varepsilon,\mathcal{L}\partial_{y_{1}}P)=O\left(|b|^3+\int_{\R^{2}} \varepsilon^2e^{-\frac{|y|}{10}}\dd y\right).
\end{aligned}
\end{equation}
Indeed, from Lemma~\ref{le:Ketest},  Lemma~\ref{le:equvar} and $\left(\varepsilon,Q^{3}\right)=0$, we have 
\begin{equation*}
\left(\frac{\lambda_{s}}{\lambda}+b\right)=\frac{
\left(\varepsilon,\mathcal{L}\partial_{y_{1}}Q^{3}\right)}{\left(\Lambda Q,Q^{3}\right)}+O\left(b^{2}+\int_{\R^{2}}\varepsilon^{2}e^{-\frac{|y|}{10}}\dd y\right).
\end{equation*}
Then, using again Proposition~\ref{Prop:Spectral}, Lemma~\ref{le:Nonloca}, Lemma~\ref{le:Ketest},  Lemma~\ref{le:equvar} and $\left(\varepsilon,Q\right)=|\left(\varepsilon,\nabla Q\right)|=0$, we obtain
\begin{equation*}
\begin{aligned}
&b_{s}+\theta b^{2}-\frac{b}{\left(P,Q\right)}\left(\varepsilon,\Lambda Q+6PQ\partial_{y_{1}}Q\right)\\
&-b\left(\frac{\lambda_{s}}{\lambda}+b\right)\frac{\left(\Lambda P, Q\right)}{\left(P,Q\right)}=O\left(|b|^{3}+\int_{\R^{2}}\varepsilon^{2}e^{-\frac{|y|}{10}}\dd y\right).
\end{aligned}
\end{equation*}
Note that, from Lemma~\ref{le:Nonloca} and an elementary computation,
\begin{equation*}
\partial_{y_{1}}\mathcal{L}P=\Lambda Q\Longrightarrow
\mathcal{L}\partial_{y_{1}}P=\Lambda Q+6PQ\partial_{y_{1}}Q.
\end{equation*}
We see that~\eqref{est:bsb2} follows from the above three identities.

On the other hand, from the definition of $g_{2}$ and $\sigma_{2}$ and $P$ and $Q$ are even in $y_{2}$, 
\begin{equation*}
\left(\Lambda Q,\sigma_{2}\right)=\left(\partial_{y_{2}}Q,\sigma_{2}\right)=(Q,g_{2})=0.
\end{equation*}
Moreover, using again the definition of $g_{2}$ and $\sigma_{2}$, we see that 
\begin{equation*}
\lim_{y_{1}\to \infty}\sigma_{2}(y_{1},y_{2})=0,\quad \mbox{for any}\ y_{2}\in \mathbb{R}\Longrightarrow \sigma_{2}\left(y_{1},y_{2}\right)=\int_{-\infty}^{y_{1}}g_{2}(\rho,y_{2})\dd \rho.
\end{equation*}
Therefore, from~\eqref{equ:varf}, $\mathcal{L}\Lambda Q=-2Q$ and $\left(\varepsilon,Q\right)=0$, we see that 
\begin{equation*}
\begin{aligned}
bJ_{2s}
=&-\frac{b}{(P,Q)}\frac{(\Lambda P,Q)}{(Q^3,\Lambda Q)}(\varepsilon,\mathcal{L}\partial_{y_{1}}Q^3)-\frac{b}{(P,Q)}(\varepsilon,\mathcal{L}\partial_{y_{1}}P)\\
&+O\left(B^{5}|b|^{3}+B^{5}|b|\int_{\R^{2}}\varepsilon^{2}\varphi_{0,B}\dd y\right).
\end{aligned}
\end{equation*}
Combining~\eqref{est:bsb2} with the above inequality, we complete the proof of (ii).

\smallskip
Proof of (iii). The estimate in (iii) is a direct consequence of (i) and (ii).

\smallskip
Proof of (iv). Let $f=\partial_{y_{2}}Q$ in~\eqref{equ:varf}. Note that
\begin{equation*}
\mathcal{L}\partial_{y_{2}}Q=0\quad \mbox{and}\quad 
\left(Q,\partial_{y_{2}}Q\right)=\left(\Lambda Q,\int_{-\infty}^{y_{1}}\partial_{y_{2}}Q(\rho,y_{2})\dd \rho\right)=0.
\end{equation*}
Therefore, using again~\eqref{equ:varf} and Lemma~\ref{le:modu1}, we obtain
\begin{equation*}
\begin{aligned}
c_{2}J_{3s}&=\frac{x_{2s}}{\lambda}\left(\partial_{y_{2}}Q,\int_{-\infty}^{y_{1}}\partial_{y_{2}}Q(\rho,y_{2})\dd \rho\right)\\
&+O\left(B^{5}b^2+B^{5}\int_{\R^{2}}\varepsilon^2\varphi_{0,B}\dd y\right)\\
&=c_{2}\frac{x_{2s}}{\lambda}+O\left(B^{5}b^2+B^{5}\int_{\R^{2}}\varepsilon^2\varphi_{0,B}\dd y\right),
\end{aligned}
\end{equation*}
which completes the proof of (iv).
\end{proof}

\section{Monotonicity formula}\label{S:Mono}
\subsection{Energy estimate}\label{SS:ENERGY}
In this subsection, we introduce the weighted energy estimate for the function $\varepsilon$. For $(i,j)\in \left\{1,2\right\}^{2}$, we denote 
\begin{equation}\label{equ:defJij}
\mathcal{J}_{i,j}=(1-J_1)^{-2\theta(j-1)-2i-12}-1,
\end{equation}
where $J_{1}$ is defined in (i) of Lemma~\ref{le:modu2}. Similar to the case of the mass-critical gKdV equation (see \emph{e.g.}~\cite[Section 3.1]{MMR}), for all $(i,j)\in \left\{1,2\right\}^{2}$, we define the following energy functionals of $\varepsilon$,
\begin{equation}\label{equ:defFij}
\mathcal{F}_{i,j}=\int_{\R^{2}}\big(|\nabla\varepsilon|^2\psi_B+(1+\mathcal{J}_{i,j})\varepsilon^2\varphi_{i,B}-\frac{1}{2}\psi_B((Q_{b}+\varepsilon)^4-Q_{b}^4-4 Q_{b}^3\varepsilon)\big)\dd y.
\end{equation}
The following qualitative estimate of the time variation of $\mathcal{F}_{i,j}$ plays an important role in our analysis (see more details in Section~\ref{SS:Mon} and Section~\ref{S:Endproof}).

\begin{proposition}\label{Prop:dsFij}
There exist some universal constants $B>100$ large enough and $0<\kappa_{1}< \min\left\{\kappa^{*},B^{-100}\right\}$ small enough such that the following holds. Assume that for all $s\in[0,s_0]$, the solution $u(t)$ with initial data $u_{0}$ satisfies the bootstrap assumption \eqref{est:Boot1}--\eqref{est:Boot3} with $0<\kappa<\kappa_{1}$. Then for all $(i,j)\in\{1,2\}^2$ and $s\in[0,s_0]$, we have 
\begin{equation}\label{est:dsFij}
	\begin{aligned}
	&\lambda^{\theta(j-1)}\frac{\dd}{\dd s}\left(\frac{\mathcal{F}_{i,j}}{\lambda^{\theta(j-1)}}\right)+\frac{1}{4}\int_{\R^{2}}\left(|\nabla \varepsilon|^2+\varepsilon^2\right)\varphi'_{i,B}\dd y\\
    &\le \frac{C_0}{B^{30}}
    \int_{\R^{2}}\left(|\nabla \varepsilon|^2+\varepsilon^2\right)\psi_{B}\dd y+
    C_1b^4.
	\end{aligned}
	\end{equation}
Here, $C_0>1$ is a universal constant independent of $B$ and $C_{1}=C_1(B)>1$ is a constant depending only on $B$.
\end{proposition}

To complete the proof of Proposition~\ref{Prop:dsFij}, we first recall the following weighted Sobolev estimates on $\mathbb{R}$ introduced in \cite{Merle}.
\begin{lemma}[\cite{Merle}]\label{le:weight1D}
Let $\omega: \mathbb{R}\rightarrow (0,\infty)$ be a $C^1$ function such that $\|\omega'/\omega\|_{L^\infty(\mathbb{R})}\lesssim 1$. Then, for all $f\in H^1(\mathbb{R})$, we have
\begin{equation*}
\begin{aligned}
\left\|f^2\sqrt{\omega}\right\|_{L^\infty(\mathbb{R})}^2&\lesssim \|f\|_{L^2(\mathbb{R})}^2\left(\int_{\mathbb{R}} \left(|f'|^2+|f|^2\right)\omega\dd r\right),\\
\left\|f^2\sqrt{\omega}\right\|_{L^\infty(\mathbb{R})}^2&\lesssim \|f\|_{L^2(\mathbb{R})}^2\left(\int_{\mathbb{R}} |f'|^2\omega\dd r+\int_{\R^{2}}|f|^2\omega\left(\frac{\omega'}{\omega}\right)^{2}\dd r\right).
\end{aligned}
\end{equation*}
\end{lemma}

\begin{proof}
The proof relies on a standard argument based on the Fundamental Theorem and the Cauchy-Schwarz inequality (see \emph{e.g.}~\cite[Lemma 6]{Merle}), and we omit it.
\end{proof}

Next, we generalize the above 1D weighted Sobolev estimate to the case of 2D.

\begin{lemma}\label{le:weight2D}
Let $\omega:\mathbb{R}^2\rightarrow (0,\infty)$ be a $C^2$ function such that 
\begin{equation}\label{est:omega}
\left\|\frac{\nabla \omega}{\omega}\right\|_{L^{\infty}(\R^{2})}+
\sum_{|\alpha|=2}\left\|\frac{\partial_{y}^{\alpha}\omega}{\omega}\right\|_{L^{\infty}(\R^{2})}\lesssim 1.
\end{equation}
 Then, for all $ f\in H^2(\mathbb{R}^2)$, we have
 \begin{equation}\label{est:weight1}
\begin{aligned}
\|f^2\sqrt{\omega}\|_{L^\infty(\mathbb{R}^2)}^2
\lesssim&\left(\int_{\mathbb{R}^2}(|f|^2+|\nabla f|^2)\sqrt{\omega}\dd y\right)^2\\
&+ \|f\|_{L^2(\mathbb{R}^2)}^2\left(\int_{\R^{2}}|\partial_{y_1}\partial_{y_2}f|^2\omega\dd y\right).
\end{aligned}
\end{equation}
Moreover, for all $f\in H^1(\mathbb{R}^2)$, we have
\begin{equation}\label{est:weight2}
\|f^2\sqrt{\omega}\|_{L^2(\mathbb{R}^2)}^2
\lesssim \|f\|_{L^2(\mathbb{R}^2)}^2\left(\int_{\mathbb{R}^2} (|\nabla f|^2+|f|^2)\omega\dd y\right).
\end{equation}
\end{lemma}

\begin{proof}
First, by an elementary computation, we see that 
\begin{equation*}
\begin{aligned}
\partial_{y_{1}}\partial_{y_{2}}\left(f^{2}\sqrt{\omega}\right)
&=f\sqrt{\omega}\left(\partial_{y_{1}}f\frac{\partial_{y_{2}}\omega}{\omega}+\partial_{y_{2}}f\frac{\partial_{y_{1}}\omega}{\omega}\right)+f^{2}\sqrt{\omega}\frac{\partial_{y_{1}}\partial_{y_{2}}\omega}{2\omega}\\
&+2\left(\partial_{y_{1}}f\partial_{y_{2}}f+f\partial_{y_{1}}\partial_{y_{2}}f\right)\sqrt{\omega}
-f^{2}\sqrt{\omega}\left(\frac{\partial_{y_{1}}\omega}{2\omega}\right)\left(\frac{\partial_{y_{2}}\omega}{2\omega}\right).
\end{aligned}
\end{equation*}
Combining the above identity with~\eqref{est:omega}, the Fundamental Theorem and the Cauchy-Schwarz inequality, we complete the proof of~\eqref{est:weight1}.

\smallskip

Second, using~\eqref{est:omega} and Lemma~\ref{le:weight1D}, we have 
\begin{equation*}
\begin{aligned}
\left\|(f^{2}\sqrt{\omega})(\cdot,y_{2})\right\|_{L^{\infty}(\R)}^{2}&\lesssim \left\|f(\cdot,y_{2})\right\|^{2}_{L^{2}(\R)}\int_{\R}\left(|\partial_{y_{1}}f(\rho,y_{2})|^{2}+|f(\rho,y_{2})|^{2}\right)\omega(\rho,y_{2})\dd \rho,\\
\left\|(f^{2}\sqrt{\omega})(y_{1},\cdot)\right\|_{L^{\infty}(\R)}^{2}&\lesssim \left\|f(y_{1},\cdot)\right\|^{2}_{L^{2}(\R)}\int_{\R}\left(|\partial_{y_{2}}f(y_{1},\rho)|^{2}+|f(y_{1},\rho)|^{2}\right)\omega(y_{1},\rho)\dd \rho.
\end{aligned}
\end{equation*}
It follows from the H\"older inequality that
\begin{equation*}
\begin{aligned}
&\int_{\R}\left\|(f^{2}\sqrt{\omega})(\cdot,y_{2})\right\|_{L^{\infty}(\R)}\dd y_{2}
+\int_{\R}\left\|(f^{2}\sqrt{\omega})(y_{1},\cdot)\right\|_{L^{\infty}(\R)}\dd y_{1}\\
&\lesssim \|f\|_{L^{2}(\R^{2})}\left(\int_{\R^{2}}\left(|\nabla f(y)|^{2}+|f(y)|^{2}\right)\omega(y)\dd y\right)^{\frac{1}{2}}.
\end{aligned}
\end{equation*}
Based on the above estimates, we obtain 
\begin{equation*}
\begin{aligned}
\int_{\mathbb{R}^2}f^4(y)\omega(y)\dd y
&\lesssim
\int_{\mathbb{R}}\int_{\mathbb{R}}\left\|(f^2\sqrt{\omega})(\cdot,y_2)\right\|_{L^\infty(\mathbb{R})}\left\|(f^2\sqrt{\omega})(y_1,\cdot)\right\|_{L^\infty(\mathbb{R})}\dd y_{1}\dd y_{2}\\
&\lesssim \left(\int_{\mathbb{R}}\left\|(f^2\sqrt{\omega})(\cdot,y_2)\right\|_{L^\infty(\mathbb{R})}\dd y_2\right)
\times\left(\int_{\mathbb{R}}\left\|(f^2\sqrt{\omega})(y_1,\cdot)\right\|_{L^\infty(\mathbb{R})}\dd y_1\right)\\
&\lesssim \|f\|_{L^2(\mathbb{R}^2)}^2\left(\int_{\mathbb{R}^2} 
\left(|\nabla f(y)|^2+|f(y)|^2\right)\omega\dd y\right),
\end{aligned}
\end{equation*}
which completes the proof of~\eqref{est:weight2}.
\end{proof}
From the definitions of $\varphi_{i,B}$ and $\psi_{B}$, we can easily check that, for $i=0,1,2$, 
\begin{equation}\label{est:psiphi3}
\begin{aligned}
\left\|\frac{\nabla \psi_{B}}{\psi_{B}}\right\|_{L^{\infty}(\R^{2})}+\left\|\frac{\nabla \varphi_{i,B}}{\varphi_{i,B}}\right\|_{L^{\infty}(\R^{2})}
+\left\|\frac{\nabla \sqrt{\psi'_{B}}}{\sqrt{\psi'_{B}}}\right\|_{L^{\infty}(\R^{2})}
&\lesssim 1,\\
\sum_{|\alpha|=2}\left(\left\|\frac{\partial_{y}^{\alpha}\psi_{B}}{\psi_{B}}\right\|_{L^{\infty}(\R^{2})}+
\left\|\frac{\partial_{y}^{\alpha}\varphi_{i,B}}{\varphi_{i,B}}\right\|_{L^{\infty}(\R^{2})}
+\left\|\frac{\partial_{y}^{\alpha}\sqrt{\psi'_{B}}}{{\sqrt{\psi'_{B}}}}\right\|_{L^{\infty}(\R^{2})}
\right)&\lesssim 1.
\end{aligned}
\end{equation}
On the other hand, using again the definitions of $\varphi_{i,B}$ and $\psi_{B}$, the H\"older inequality and the bootstrap assumption~\eqref{est:Boot3}, we deduce that
\begin{equation}\label{est:y17}
\begin{aligned}
&\int_{\R}\int_{0}^{\infty}y_1^7\varepsilon^2\dd y_{1}\dd y_{2}\\
&\lesssim B^{70}\int_{\R^{2}}\psi_{B}\varepsilon^2\dd y+ 
\left(\int_{\R}\int_{B^{10}}^{\infty}y_1^{100}\varepsilon^2\dd y_{1}\dd y_{2}\right)^{\frac{1}{94}}
\left(\int_{\R}\int_{B^{10}}^{\infty}y_1^{6}\varepsilon^2\dd y_{1}\dd y_{2}\right)^{\frac{93}{94}}\\
&\lesssim
B^{70}\int_{\R^{2}}\psi_{B}\varepsilon^2\dd y+
 \left(1+\frac{1}{\lambda^{\frac{50}{47}}}\right)\left(B^{70}\int_{\R}\int_{B^{10}}^{\infty}\varphi'_{1,B}\varepsilon^2\dd y_{1}\dd y_{2}\right)^{\frac{93}{94}},
\end{aligned} 
\end{equation}
and
\begin{equation}\label{est:y18}
\begin{aligned}
&\int_{\R}\int_{0}^{\infty}y_1^8\varepsilon^2\dd y_{1}\dd y_{2}\\
&\lesssim B^{80}\int_{\R^{2}}\psi_{B}\varepsilon^2\dd y+
\left(\int_{\R}\int_{B^{10}}^{\infty}y_1^{100}\varepsilon^2\dd y_{1}\dd y_{2}\right)^{\frac{1}{93}}
\left(\int_{\R}\int_{B^{10}}^{\infty}y_1^{7}\varepsilon^2\dd y_{1}\dd y_{2}\right)^{\frac{92}{93}}\\
&\lesssim B^{80}\int_{\R^{2}}\psi_{B}\varepsilon^2\dd y
+\left(1+\frac{1}{\lambda^{\frac{100}{93}}}\right)
\left(B^{80}\int_{\R}\int_{B^{10}}^{\infty}\varphi'_{2,B}\varepsilon^2\dd y_{1}\dd y_{2}\right)^{\frac{92}{93}}.
\end{aligned} 
\end{equation}

We now give a complete proof of Proposition~\ref{Prop:dsFij}.
\begin{proof}[Proof of Proposition~\ref{Prop:dsFij}]
For all $(i,j)\in \left\{1,2\right\}^{2}$, we decompose
\begin{equation}\label{equ:dsFij}
\lambda^{\theta(j-1)}\frac{\dd }{\dd s}\left(\frac{\mathcal{F}_{i,j}}{\lambda^{\theta(j-1)}}\right)
=\frac{\dd}{\dd s}\mathcal{F}_{i,j}-\theta(j-1)\frac{\lambda_{s}}{\lambda}\mathcal{F}_{i,j}
=\mathcal{I}_{1}+\mathcal{I}_{2}+\mathcal{I}_{3},
\end{equation}
where
\begin{equation*}
\begin{aligned}
&\mathcal{I}_{1}=2\int_{\R^{2}}
\left(\partial_{s}\varepsilon-\frac{\lambda_s}{\lambda}\Lambda\varepsilon\right)\left(-\nabla\cdot(\psi_B\nabla\varepsilon)+\varphi_{i,B}\varepsilon-\psi_B\left((Q_b+\varepsilon)^3-Q_b^3\right)\right)\dd y,\\
&\mathcal{I}_{2}=2\mathcal{J}_{i,j}\int_{\R^{2}}\left(\partial_{s}\varepsilon-\frac{\lambda_s}{\lambda}\Lambda\varepsilon\right)\varphi_{i,B}\varepsilon\dd y-2\int_{\R^{2}}\psi_B\partial_{s}Q_{b}
\left(3Q_{b}\varepsilon^{2}+\varepsilon^{3}\right)\dd y,\\
&\mathcal{I}_{3}=2\frac{\lambda_s}{\lambda}\int_{\R^{2}}\Lambda\varepsilon
\left(-\nabla\cdot(\psi_B\nabla\varepsilon)+(1+\mathcal{J}_{i,j})\varphi_{i,B}\varepsilon-\psi_B\left((Q_b+\varepsilon)^3-Q_b^3\right)\right)\dd y\\
&\quad 
+\left(\frac{\dd}{\dd s}\mathcal{J}_{i,j}\right)\int_{\R^{2}}\varphi_{i,B}\varepsilon^2\dd y
-\theta(j-1)\frac{\lambda_s}{\lambda}\mathcal{F}_{i,j}.
\end{aligned}
\end{equation*}

\textbf{Step 1.} Estimate on $\mathcal{I}_{1}$. We claim that, there exist some universal constants $B>100$ large enough and $0<\kappa_{1}<B^{-100}$ small enough, such that 
\begin{equation}\label{est:I1}
\begin{aligned}
\mathcal{I}_{1}
&\le-\frac{1}{2}\int_{\R^{2}}(|\nabla\varepsilon|^2+\varepsilon^2)\varphi'_{i,B}\dd y\\
&+ \frac{C_{2}}{B^{30}}\int_{\R^{2}}\left(|\nabla\varepsilon|^2+\varepsilon^2\right)\psi_B\dd y+{C_{3}}b^4.
\end{aligned}
\end{equation}
Here, $C_{2}>1$ is a universal constant independent of $B$ and $C_{3}=C_{3}(B)>1$ is a constant depending only on $B$. 

Indeed, we use~\eqref{equ:varf} to rewrite
\begin{equation*}
\partial_{s}\varepsilon-\frac{\lambda_{s}}{\lambda}\Lambda \varepsilon=\partial_{y_{1}}\left(-\Delta \varepsilon+\varepsilon-(Q_{b}+\varepsilon)^{3}+Q_{b}^{3}\right)+{\rm{Mod}}-\left(\frac{\lambda_{s}}{\lambda}+b\right)\Lambda \varepsilon+\Psi_{b}.
\end{equation*}
Based on the above identity, we decompose
\begin{equation*}
\mathcal{I}_{1}=\mathcal{I}_{1,1}+\mathcal{I}_{1,2}+\mathcal{I}_{1,3},
\end{equation*}
where
\begin{equation*}
\begin{aligned}
\mathcal{I}_{1,1}&=2\int_{\R^{2}}
\Psi_{b}\left(-\nabla\cdot(\psi_B\nabla\varepsilon)+\varphi_{i,B}\varepsilon-\psi_B\left((Q_b+\varepsilon)^3-Q_b^3\right)\right)\dd y,\\
\mathcal{I}_{1,2}&=2\int_{\R^{2}}
\partial_{y_{1}}\left(-\Delta \varepsilon+\varepsilon-(Q_{b}+\varepsilon)^{3}+Q_{b}^{3}\right)\\
&\quad \quad\quad \quad  \times\left(-\nabla\cdot(\psi_B\nabla\varepsilon)+\varphi_{i,B}\varepsilon-\psi_B\left((Q_b+\varepsilon)^3-Q_b^3\right)\right)\dd y,\\
\mathcal{I}_{1,3}&=2\int_{\R^{2}}
\left({\rm{Mod}}-\left(\frac{\lambda_{s}}{\lambda}+b\right)\Lambda \varepsilon\right)\\
&\quad \quad \quad \quad \times\left(-\nabla\cdot(\psi_B\nabla\varepsilon)+\varphi_{i,B}\varepsilon-\psi_B\left((Q_b+\varepsilon)^3-Q_b^3\right)\right)\dd y.
\end{aligned}
\end{equation*}
\emph{Estimate on $\mathcal{I}_{1,1}$.} Note that, from Lemma~\ref{le:Ketest}, the definition of $\varphi_{i,B}$ and $\psi_{B}$, 
\begin{equation*}
\begin{aligned}
&\int_{\R^{2}}\left(|\nabla \Psi_{b}|^{2}+\left|\Psi_{b}\right|^{2}\right)\left(\varphi'_{i,B}+\psi_{B}+|y_{1}|^{2}\varphi'_{i,B}\mathbf{1}_{[B^{10},\infty)}\right)\dd y\\
&\lesssim |b|^{\frac{7}{2}}\int_{\R}\int_{-\infty}^{-|b|^{-\frac{4}{3}}}e^{\frac{y_{1}}{B}}e^{-\frac{|y_{2}|}{3}}\dd y_{1}\dd y_{2}+b^{4}\int_{\R}\int_{0}^{\infty}e^{-\frac{|y|}{3}}\left(1+|y_{1}|^{10}\right)\dd y_{1}\dd y_{2}\\
&+b^{4}\int_{\R}\int_{-\infty}^{0}e^{\frac{y_{1}}{B}}e^{-\frac{|y_{2}|}{3}}\dd y_{1}\dd y_{2}\lesssim Bb^{4}.
\end{aligned}
\end{equation*}
It follows from~\eqref{est:Boot1}, \eqref{est:psiphi3}, Lemma~\ref{le:Ketest}, Lemma~\ref{le:psiphi2} and Lemma~\ref{le:weight2D} that,
\begin{equation}\label{est:I11}
\begin{aligned}
\mathcal{I}_{1,1}
&\lesssim B^\frac{1}{2}b^{2}\left( \int_{\R^{2}}
\left(|\nabla \varepsilon|^{2}+\varepsilon^{2}\right)
(\varphi'_{i,B}+{\psi'_B}+\psi_{B})\dd y\right)^{\frac{1}{2}}\\
&+|b|^{\frac{7}{4}}\int_{\R^{3}}\varepsilon^{2}\psi_{B}\dd y+|b|^{\frac{7}{4}}\int_{\R^{2}}\varepsilon^{4}\psi_{B}\dd y\\
&\le \frac{C_{4}}{B^{30}}\int_{\R^{2}}\left(|\nabla \varepsilon|^{2}+\varepsilon^{2}\right)(B\varphi'_{i,B}+\psi_{B})\dd y
+C_{5}b^{4}.
\end{aligned}
\end{equation}
Here, $C_{4}>1$ is a universal constant independent of $B$ and $C_{5}=C_{5}(B)>1$ is a constant depending only on $B$.

\smallskip
\emph{Estimate on $\mathcal{I}_{1,2}$.} By an elementary computation, we decompose 
\begin{equation*}
\mathcal{I}_{1,2}=\mathcal{I}_{1,2,1}+\mathcal{I}_{1,2,2}+\mathcal{I}_{1,2,3},
\end{equation*}
where 
\begin{equation*}
\begin{aligned}
\mathcal{I}_{1,2,1}&=2\int_{\R^{2}}\left(\partial_{y_{1}}\left(-\Delta\varepsilon+\varepsilon\right)\right)\left(-\psi'_{B}\partial_{y_{1}}\varepsilon+\left(\varphi_{i,B}-\psi_{B}\right)\varepsilon\right)\dd y,\\
\mathcal{I}_{1,2,2}&=2\int_{\R^{2}}\left(\partial_{y_{1}}
\left((Q_{b}+\varepsilon)^{3}-Q_{b}^{3}\right)\right)\left(\psi'_{B}\partial_{y_{1}}\varepsilon-\left(\varphi_{i,B}-\psi_{B}\right)\varepsilon
\right)\dd y,\\
\mathcal{I}_{1,2,3}&=2\int_{\R^{2}}\left(\partial_{y_{1}}\left(-\Delta \varepsilon+\varepsilon-(Q_{b}+\varepsilon)^{3}+Q_{b}^{3}\right)\right)\left(-\Delta \varepsilon+\varepsilon-(Q_{b}+\varepsilon)^{3}+Q_{b}^{3}\right)\psi_{B}\dd y.
\end{aligned}
\end{equation*}
From integration by parts, we directly deduce that 
\begin{equation*}
\begin{aligned}
\mathcal{I}_{1,2,1}=
&-2\int_{\R^{2}}\psi'_{B}\left((\partial^{2}_{y_{1}}\varepsilon)^{2}+(\partial_{y_{1}}\partial_{y_{2}}\varepsilon)^{2}\right)\dd y-\int_{\R^{2}}(\partial_{y_{1}}\varepsilon)^{2}\left(3\varphi'_{i,B}-\psi'_{B}-\psi'''_{B}\right)\dd y\\
&-\int_{\R^{2}}\varepsilon^{2}\left((\varphi'_{i,B}-\psi'_{B})-(\varphi'''_{i,B}-\psi'''_{B})\right)\dd y-\int_{\R^{2}}(\partial_{y_{2}}\varepsilon)^{2}\left(\varphi'_{i,B}-\psi'_{B}\right)\dd y,
\end{aligned}
\end{equation*}
\begin{equation*}
\begin{aligned}
\mathcal{I}_{1,2,2}
&=-2\int_{\R^{2}}\left(\varphi_{i,B}-\psi_{B}\right)(\partial_{y_{1}}Q_{b})\left(3Q_{b}\varepsilon^{2}+\varepsilon^{3}\right)\dd y\\
&+\frac{1}{2}\int_{\R^{2}}\left(\varphi'_{i,B}-\psi'_{B}\right)(6Q^{2}_{b}\varepsilon^{2}+8Q_{b}\varepsilon^{3}+3\varepsilon^{4})\dd y\\
&+6\int_{\R^{2}}\left(\psi'_{B}\partial_{y_{1}}\varepsilon\right)
\left((\partial_{y_{1}}Q_{b})(2Q_{b}\varepsilon+\varepsilon^{2})+(\partial_{y_{1}}\varepsilon)(Q_{b}+\varepsilon)^{2}\right)\dd y,\quad \quad \quad \quad \quad \quad 
\end{aligned}
\end{equation*}
and
\begin{equation*}
\begin{aligned}
\mathcal{I}_{1,2,3}
&=-\int_{\R^{2}}\psi'_{B}
\bigg(2|\nabla \varepsilon|^{2}+\sum_{|\alpha|=2}|\partial_{y}^{\alpha}\varepsilon|^{2}\bigg)\dd y
+\int_{\R^{2}}\left(\psi'''_{B}(\partial_{y_{2}}\varepsilon)^{2}
-\left(\psi'_{B}-\psi'''_{B}\right)\varepsilon^{2}
\right)\dd y\\
&-
\int_{\R^{2}}\left(\left(-\Delta \varepsilon+\varepsilon-(Q_{b}+\varepsilon)^{3}+Q_{b}^{3}\right)^{2}-
\left(-\Delta \varepsilon+\varepsilon\right)^{2}
\right)
\psi'_{B}\dd y.
\end{aligned}
\end{equation*}
Based on the above identities, we rewrite the term $\mathcal{I}_{1,2}$ by
\begin{equation*}
\mathcal{I}_{1,2}=\mathcal{I}_{1,2,2}+\mathcal{I}_{1,2,4}+\mathcal{I}_{1,2,5}+\mathcal{I}_{1,2,6},
\end{equation*}
where 
\begin{equation*}
\begin{aligned}
\mathcal{I}_{1,2,4}&=-\int_{\R^{2}}\psi'_{B}
\left(3(\partial^{2}_{y_{1}}\varepsilon)^{2}+4(\partial_{y_{1}}\partial_{y_{2}}\varepsilon)^{2}+(\partial^{2}_{y_{2}}\varepsilon)^{2}\right)\dd y,\\
\mathcal{I}_{1,2,5}&=-\int_{\R^{2}}(\partial_{y_{1}}\varepsilon)^{2}\left(3\varphi'_{i,B}+\psi'_{B}-
\psi'''_{B}\right)\dd y\\
&-\int_{\R^{2}}(\partial_{y_{2}}\varepsilon)^{2}\left(\varphi'_{i,B}+\psi'_{B}-
\psi'''_{B}\right)\dd y
-\int_{\R^{2}}\varepsilon^{2}\left(\varphi'_{i,B}-\varphi'''_{i,B}\right)\dd y,\\
\mathcal{I}_{1,2,6}&=-
\int_{\R^{2}}\left(\left(-\Delta \varepsilon+\varepsilon-(Q_{b}+\varepsilon)^{3}+Q_{b}^{3}\right)^{2}-
\left(-\Delta \varepsilon+\varepsilon\right)^{2}\right)
\psi'_{B}\dd y.
\end{aligned}
\end{equation*}

First, from~\eqref{est:Boot1},~\eqref{est:psiphi3}, Lemma~\ref{le:psiphi2}, Lemma~\ref{le:psiphi3} and Lemma~\ref{le:weight2D}, we have
\begin{equation*}
\mathcal{I}_{1,2,2}\lesssim \frac{1}{B^{30}}\int_{\R^{2}}\left(|\nabla \varepsilon|^{2}+\varepsilon^{2}\right)\left(B\varphi'_{i,B}+\psi_{B}\right)\dd y+\int_{\R^{2}}\psi'_{B}\varepsilon^{2}(\partial_{y_{1}}\varepsilon)^{2}\dd y.
\end{equation*}
Note that, using again~\eqref{est:psiphi3} and Lemma~\ref{le:weight2D}, we deduce that 
\begin{equation*}
\begin{aligned}
\int_{\R^{2}}\varepsilon^{2}|\nabla\varepsilon|^{2}\psi'_{B}\dd y
&\lesssim \left\|\varepsilon^{2}\sqrt{\psi'_{B}}\right\|_{L^{\infty}}
\left(\int_{\R^{2}}|\nabla \varepsilon|^{2}\sqrt{\psi'_{B}}\dd y\right)\\
&\lesssim \left(\int_{\R^{2}}|\nabla \varepsilon|^{2}\sqrt{\psi'_{B}}\dd y\right)
\left(\int_{\R^{2}}(|\nabla \varepsilon|^{2}+\varepsilon^{2})\sqrt{\psi'_{B}}\dd y\right)\\
&+\|\varepsilon\|_{L^{2}}\left(\int_{\R^{2}}|\nabla \varepsilon|^{2}\sqrt{\psi'_{B}}\dd y\right)
\left(\sum_{|\alpha|=2}\left|\partial_{y}^{\alpha}\varepsilon\right|^{2}\psi'_{B}\dd y\right)^{\frac{1}{2}}.
\end{aligned}
\end{equation*}
By integration by parts,~\eqref{est:psiphi3} and Lemma~\ref{le:psiphi2},
\begin{equation*}
\begin{aligned}
\int_{\R^{2}}\varepsilon^{2}\sqrt{\psi'_{B}}\dd y
&\lesssim \int_{\R^{2}}\varepsilon^{2}\left(B^{\frac{1}{2}}\varphi'_{i,B}+B^{-\frac{1}{2}}\psi_{B}\right)\dd y\lesssim \mathcal{N}_{2},\\
\int_{\R^{2}}|\nabla \varepsilon|^{2}\sqrt{\psi'_{B}}\dd y
&=-\int_{\R^{2}}\varepsilon\Delta \varepsilon\sqrt{\psi'_{B}}\dd y+\frac{1}{2}\int_{\R^{2}}\varepsilon^{2}\partial_{y_{1}}^{2}\sqrt{\psi'_{B}}\dd y\\
&\lesssim \|\varepsilon\|_{L^{2}}\sum_{|\alpha|=2}\left(\int_{\R^{2}}\left|\partial_{y}^{\alpha}\varepsilon\right|^{2}\psi'_{B}\dd y\right)^{\frac{1}{2}}+\int_{\R^{2}}\varepsilon^{2}\sqrt{\psi'_{B}}\dd y.
\end{aligned}
\end{equation*}
It follows from~\eqref{est:Boot1} that 
\begin{equation*}
\int_{\R^{2}}\psi'_{B}\varepsilon^{2}(\partial_{y_{1}}\varepsilon)^{2}\dd y
\lesssim \frac{1}{B^{30}}\sum_{|\alpha|=2}\int_{\R^{2}}\left|\partial_{y}^{\alpha}\varepsilon\right|^{2}\psi'_{B}\dd y+\frac{1}{B^{30}}\int_{\R^{2}}\varepsilon^{2}\left(B\varphi'_{i,B}+\psi_{B}\right)\dd y.
\end{equation*}
Combining the above estimates, we obtain
\begin{equation}\label{est:I122}
\begin{aligned}
\mathcal{I}_{1,2,2}
&\lesssim \frac{1}{B^{30}}\int_{\R^{2}}\left(|\nabla \varepsilon|^{2}+\varepsilon^{2}\right)\left(B\varphi'_{i,B}+\psi_{B}\right)\dd y\\
&+\frac{1}{B^{30}}\sum_{|\alpha|=2}\int_{\R^{2}}\left|\partial_{y}^{\alpha}\varepsilon\right|^{2}\psi'_{B}\dd y.
\end{aligned}
\end{equation}
Second, from the definition of $\mathcal{I}_{1,2,5}$ and Lemma~\ref{le:psiphi2}, we directly have 
\begin{equation}\label{est:I125}
\mathcal{I}_{1,2,5}+\frac{9}{10}\int_{\R^{2}}\left(|\nabla \varepsilon|^{2}+\varepsilon^{2}\right)\varphi'_{i,B}\dd y\lesssim \frac{1}{B^{30}}\int_{\R^{2}}\left(|\nabla \varepsilon|^{2}+\varepsilon^{2}\right)\psi_{B}\dd y.
\end{equation}
Based on a similar argument to the one in the estimate of $\mathcal{I}_{1,2,2}$, we deduce that 
\begin{equation}\label{est:I126}
\begin{aligned}
\mathcal{I}_{1,2,6}
&\lesssim \frac{1}{B^{30}}\int_{\R^{2}}\left(|\nabla \varepsilon|^{2}+\varepsilon^{2}\right)\left(B\varphi'_{i,B}+\psi_{B}\right)\dd y\\
&+\frac{1}{B^{30}}\sum_{|\alpha|=2}\int_{\R^{2}}\left|\partial_{y}^{\alpha}\varepsilon\right|^{2}\psi'_{B}\dd y.
\end{aligned}
\end{equation}
Here, we use the fact that 
\begin{equation*}
\left\|\varepsilon^{2}\sqrt{\psi'_{B}}\right\|_{L^{\infty}}^{2}\lesssim \|\varepsilon\|_{L^{2}}^{2}\sum_{|\alpha|=2}\int_{\R^{2}}\left|\partial_{y}^{\alpha}\varepsilon\right|^{2}\psi'_{B}\dd y+\mathcal{N}_{2}\int_{\R^{2}}\varepsilon^{2}\sqrt{\psi'_{B}}\dd y.
\end{equation*}
Combining estimates~\eqref{est:I122}--\eqref{est:I126} with the definition of $\mathcal{I}_{1,2,4}$, we conclude that 
\begin{equation}\label{est:I12}
\begin{aligned}
\mathcal{I}_{1,2}&\le -\frac{8}{9}\sum_{|\alpha|=2}\int_{\R^{2}}|\partial^{\alpha}_{y}\varepsilon|^{2}\psi'_{B}\dd y\\
&-\frac{8}{9}\int_{\R^{2}}
(|\nabla \varepsilon|^{2}+\varepsilon^{2})\varphi'_{i,B}\dd y+\frac{C_{6}}{B^{30}}\int_{\R^{2}}\left(|\nabla \varepsilon|^{2}+\varepsilon^{2}\right)\psi_{B}\dd y.
\end{aligned}
\end{equation}
Here, $C_{6}>1$ is a universal constant independent of $B$.

\smallskip
\emph{Estimate on $\mathcal{I}_{1,3}$.} By an elementary computation, we decompose
\begin{equation*}
\mathcal{I}_{1,3}=\mathcal{I}_{1,3,1}+\mathcal{I}_{1,3,2}+\mathcal{I}_{1,3,3}+\mathcal{I}_{1,3,4}+\mathcal{I}_{1,3,5},
\end{equation*}
where 
\begin{equation*}
\mathcal{I}_{1,3,1}=2\int_{\R^{2}}\widetilde{\rm{Mod}}\left(-\nabla \cdot (\psi_{B}\nabla \varepsilon)+\varphi_{i,B}\varepsilon
-\psi_{B}((Q_{b}+\varepsilon)^{3}-Q^{3}_{b})\right)\dd y, \quad \quad \quad \quad \
\end{equation*}
\begin{equation*}
\mathcal{I}_{1,3,2}=2\left(\frac{\lambda_{s}}{\lambda}+b\right)\int_{\R^{2}}\Lambda Q\left(-\nabla \cdot (\psi_{B}\nabla \varepsilon)+\varphi_{i,B}\varepsilon
-\psi_{B}((Q_{b}+\varepsilon)^{3}-Q^{3}_{b})\right)\dd y, \ \
\end{equation*}
\begin{equation*}
\mathcal{I}_{1,3,3}=2\left(\frac{x_{1s}}{\lambda}-1\right)
\int_{\R^{2}}\partial_{y_{1}}Q\left(-\nabla \cdot (\psi_{B}\nabla \varepsilon)+\varphi_{i,B}\varepsilon
-\psi_{B}((Q_{b}+\varepsilon)^{3}-Q^{3}_{b})\right)\dd y, \
\end{equation*}
\begin{equation*}
\mathcal{I}_{1,3,4}=2\left(\frac{x_{1s}}{\lambda}-1\right)
\int_{\R^{2}}\partial_{y_{1}}\varepsilon\left(-\nabla \cdot (\psi_{B}\nabla \varepsilon)+\varphi_{i,B}\varepsilon
-\psi_{B}((Q_{b}+\varepsilon)^{3}-Q^{3}_{b})\right)\dd y, \
\end{equation*}

\begin{equation*}
\mathcal{I}_{1,3,5}=2\frac{x_{2s}}{\lambda}\int_{\R^{2}}(\partial_{y_{2}}Q+\partial_{y_{2}}\varepsilon)\left(-\nabla \cdot (\psi_{B}\nabla \varepsilon)+\varphi_{i,B}\varepsilon
-\psi_{B}((Q_{b}+\varepsilon)^{3}-Q^{3}_{b})\right)\dd y.\
\end{equation*}
Here, we denote 
\begin{equation*}
\begin{aligned}
\widetilde{\rm{Mod}}&={\rm{Mod}}-\left(\frac{x_{1s}}{\lambda}-1\right)(\partial_{y_{1}}Q+\partial_{y_{1}}\varepsilon)\\
&-\frac{x_{2s}}{\lambda}(\partial_{y_{2}}Q+\partial_{y_{2}}\varepsilon)-\left(\frac{\lambda_{s}}{\lambda}+b\right)
\left(\Lambda Q+\Lambda \varepsilon\right).
\end{aligned}
\end{equation*}

Indeed, by integration by parts, we deduce that,
\begin{equation*}
\mathcal{I}_{1,3,1}=-2\int_{\R^{2}}\nabla \cdot \left(\psi_{B}\nabla {\widetilde{\rm{Mod}}}\right
)\varepsilon\dd y+
2\int_{\R^{2}}\widetilde{\rm{Mod}}\left(\varphi_{i,B}\varepsilon
-\psi_{B}((Q_{b}+\varepsilon)^{3}-Q^{3}_{b})\right)\dd y.
\end{equation*}

Based on Lemma~\ref{le:psiphi3} and Lemma~\ref{le:modu1}, we have 
\begin{equation*}
\begin{aligned}
&\left|\widetilde{{\rm{Mod}}}\right|+\left|\nabla\widetilde{{\rm{Mod}}}\right|+\sum_{|\alpha|=2}\left|\partial_{y}^{\alpha}\widetilde{{\rm{Mod}}}\right|\\
&\lesssim \left(b^{2}+\mathcal{N}_{2}^{\frac{1}{2}}\right)\left(e^{-\frac{|y_{2}|}{4}}\mathbf{1}_{[-2,0]}(|b|^{\frac{3}{4}}y_{1})+e^{-\frac{|y|}{4}}\mathbf{1}_{[0,\infty)}(y_{1})\right)\left(B\varphi'_{i,B}+\psi_{B}\right).
\end{aligned}
\end{equation*}

Therefore, from~\eqref{est:psiphi3}, Lemma~\ref{le:psiphi2} and~Lemma~\ref{le:weight2D}, we obtain
\begin{equation}\label{est:I131}
\begin{aligned}
\mathcal{I}_{1,3,1}
&\lesssim 
\left(b^{2}+\mathcal{N}_{2}^{\frac{1}{2}}\right)\int_{\R}\int_{-\infty}^{0}e^{-\frac{|y_{2}|}{4}}|\varepsilon|\left(B\varphi'_{i,B}+\psi_{B}\right)\dd y_{1}\dd y_{2}\\
&+\left(b^{2}+\mathcal{N}_{2}^{\frac{1}{2}}\right)\int_{\R}\int_{0}^{\infty}e^{-\frac{|y|}{4}}|\varepsilon|\left(B\varphi'_{i,B}+\psi_{B}\right)\dd y_{1}\dd y_{2}\\
&+\left(b^{2}+\mathcal{N}_{2}^{\frac{1}{2}}\right)\int_{\R^{2}}\varepsilon^{2}\psi_{B}\dd y+\left(b^{2}+\mathcal{N}_{2}^{\frac{1}{2}}\right)\int_{\R^{2}}\varepsilon^{4}\psi_{B}\dd y\\
&\le \frac{C_{7}}{B^{30}}\int_{\R^{2}}\left(|\nabla\varepsilon|^{2}+\varepsilon^{2}\right)\left(B\varphi'_{i,B}+\psi_{B}\right)\dd y+C_{8}b^{4}.
\end{aligned}
\end{equation}

Here, $C_{7}>1$ is a universal constant independent of $B$ and $C_{8}=C_{8}(B)>1$ is a constant dependent only on $B$.

By direct computations, we check the following identity
\begin{equation}\label{equ:KeyI13}
\begin{aligned}
&2\left(-\nabla \cdot (\psi_{B}\nabla \varepsilon)+\varphi_{i,B}\varepsilon
-\psi_{B}((Q_{b}+\varepsilon)^{3}-Q^{3}_{b})\right)\\
&=\mathcal{L}\varepsilon+2(\varphi_{i,B}-1)\varepsilon-\left(3(Q_{b}^{2}-Q^{2})\varepsilon+3Q_{b}\varepsilon^{2}+\varepsilon^{3}\right)\\
&+(2\psi_{B}-1)\left(-\Delta \varepsilon-3Q_{b}^{2}\varepsilon-3Q_{b}\varepsilon^{2}-\varepsilon^{3}\right)-2\psi'_{B}\partial_{y_{1}}\varepsilon.
\end{aligned}
\end{equation}
From the orthogonality condition~\eqref{equ:orth},~\eqref{equ:KeyI13} and integration by parts, 
\begin{equation*}
\begin{aligned}
\mathcal{I}_{1,3,2}
&=-2\left(\frac{\lambda_{s}}{\lambda}+b\right)\int_{\R^{2}}\Lambda Q(2\psi_{B}-1)\left(3Q_{b}^{2}\varepsilon+3Q_{b}\varepsilon^{2}+\varepsilon^{3}\right)\dd y\\
&+2\left(\frac{\lambda_{s}}{\lambda}+b\right)\int_{\R^{2}}\Lambda Q\left(2(\varphi_{i,B}-1)\varepsilon-(2\psi_{B}-1)\Delta \varepsilon-2\psi'_{B}\partial_{y_{1}}\varepsilon\right)\dd y\\
&-2\left(\frac{\lambda_{s}}{\lambda}+b\right)\int_{\R^{2}}\Lambda Q(3(Q_{b}^{2}-Q^{2})\varepsilon+3Q_{b}\varepsilon^{2}+\varepsilon^{3})\dd y.
\end{aligned}
\end{equation*}
Then, from the definition of $\varphi_{i,B}$, Lemma~\ref{le:psiphi} and Lemma~\ref{le:psiphi2},
\begin{equation*}
\begin{aligned}
&|\Lambda Q|\left(|2\psi_{B}-1|+|2\varphi_{i,B}-1|+|\psi''_{B}|\right)+|\partial_{y_{1}}\Lambda Q||\psi'_{B}|\lesssim e^{-\frac{1}{6}B^{\frac{1}{3}}}e^{-\frac{|y|}{10}}\psi_{B}.
\end{aligned}
\end{equation*}
It follows from Lemma~\ref{le:modu1} and Lemma~\ref{le:weight2D} that 
\begin{equation*}
\begin{aligned}
&\left|\left(\frac{\lambda_{s}}{\lambda}+b\right)\int_{\R^{2}}\Lambda Q(2\psi_{B}-1)\left(3Q_{b}^{2}\varepsilon+3Q_{b}\varepsilon^{2}+\varepsilon^{3}\right)\dd y\right|\\
&+\left|\left(\frac{\lambda_{s}}{\lambda}+b\right)\int_{\R^{2}}\Lambda Q\left(2(\varphi_{i,B}-1)\varepsilon-(2\psi_{B}-1)\Delta \varepsilon-2\psi'_{B}\partial_{y_{1}}\varepsilon\right)\dd y\right|\\
&\lesssim \left(b^{2}+\left(\int_{\R^{2}}\varepsilon^{2}e^{-\frac{|y|}{10}}\dd y\right)^{\frac{1}{2}}\right)
\left(|b|+e^{-\frac{1}{6}B^{\frac{1}{3}}}\right)\int_{\R^{2}}e^{-\frac{|y|}{10}}\left(|\varepsilon|+|\varepsilon|^{4}\right)\psi_{B}\dd y\\
&\le \frac{C_{9}}{B^{30}}\int_{\R^{2}}\left(|\nabla \varepsilon|^{2}+\varepsilon^{2}\right)(B\varphi'_{i,B}+\psi_{B})\dd y+C_{10}b^{4}.
\end{aligned}
\end{equation*}
Here, $C_{9}>1$ is a universal constant independent of $B$ and $C_{10}=C_{10}(B)>1$ is a constant dependent only on $B$. Next, using again~\eqref{est:Boot1}, Lemma~\ref{le:modu1} and Lemma~\ref{le:weight2D}, 
\begin{equation*}
\begin{aligned}
&\left|\left(\frac{\lambda_{s}}{\lambda}+b\right)\int_{\R^{2}}\Lambda Q(3(Q_{b}^{2}-Q^{2})\varepsilon+3Q_{b}\varepsilon^{2}+\varepsilon^{3})\dd y\right|\\
&\lesssim  \left(b^{2}+\left(\int_{\R^{2}}\varepsilon^{2}e^{-\frac{|y|}{10}}\dd y\right)^{\frac{1}{2}}\right)
\int_{\R^{2}}e^{-\frac{|y|}{10}}\left(|b||\varepsilon|+\varepsilon^{2}+|\varepsilon|^{3}\right)\dd y\\
&\le \frac{C_{11}}{B^{30}}\int_{\R^{2}}\left(|\nabla \varepsilon|^{2}+\varepsilon^{2}\right)(B\varphi'_{i,B}+\psi_{B})\dd y+C_{12}b^{4}.
\end{aligned}
\end{equation*}
Here, $C_{11}>1$ is a universal constant independent of $B$ and $C_{12}=C_{12}(B)>1$ is a constant dependent only on $B$. 

Combining the above estimates, we deduce that 
\begin{equation}\label{est:I132}
\mathcal{I}_{1,3,2}\le  \frac{C_{13}}{B^{30}}\int_{\R^{2}}\left(|\nabla \varepsilon|^{2}+\varepsilon^{2}\right)(B\varphi'_{i,B}+\psi_{B})\dd y+C_{14}b^{4}.
\end{equation}
Here, $C_{13}>1$ is a universal constant independent of $B$ and $C_{14}=C_{14}(B)>1$ is a constant dependent only on $B$.

Based on a similar argument to the one in the estimate of $\mathcal{I}_{1,3,2}$, we deduce that 
\begin{equation}\label{est:I133}
\mathcal{I}_{1,3,3}\le \frac{C_{15}}{B^{30}}\int_{\R^{2}}\left(|\nabla \varepsilon|^{2}+\varepsilon^{2}\right)\left(B\varphi'_{i,B}+\psi_{B}\right)\dd y+C_{16}b^{4}.
\end{equation}
Here, $C_{15}>1$ is a universal constant independent of $B$ and $C_{16}=C_{16}(B)>1$ is a constant dependent only on $B$.

Note also that, by integration by parts,
\begin{equation*}
\begin{aligned}
&\int_{\R^{2}}\partial_{y_{1}}\varepsilon\left(-\nabla \cdot (\psi_{B}\nabla \varepsilon)+\varphi_{i,B}\varepsilon
-\psi_{B}((Q_{b}+\varepsilon)^{3}-Q^{3}_{b})\right)\dd y\\
&=-\frac{1}{2}\int_{\R^{2}}\left(|\nabla \varepsilon|^{2}\psi'_{B}+\varepsilon^{2}\varphi'_{i,B}
-3\varepsilon^{2}\partial_{y_{1}}(Q^{2}_{b}\psi_{B})-2\varepsilon^{3}\partial_{y_{1}}(Q_{b}\psi_{B})-\frac{1}{2}\varepsilon^{4}\psi'_{B}\right)\dd y.
\end{aligned}
\end{equation*}
It follows from~\eqref{est:Boot1}, Lemma~\ref{le:psiphi3} and Lemma~\ref{le:weight2D} that
\begin{equation*}
\begin{aligned}
&\int_{\R^{2}}\partial_{y_{1}}\varepsilon\left(-\nabla \cdot (\psi_{B}\nabla \varepsilon)+\varphi_{i,B}\varepsilon
-\psi_{B}((Q_{b}+\varepsilon)^{3}-Q^{3}_{b})\right)\dd y\\
&\lesssim \int_{\R^{2}}\left(|\nabla \varepsilon|^{2}+\varepsilon^{2}\right)\left(B\varphi'_{i,B}+\psi_{B}\right)\dd y+b^{4}.
\end{aligned}
\end{equation*}
Therefore, from~\eqref{le:modu1}, we obtain
\begin{equation}\label{est:I134}
\mathcal{I}_{1,3,4}\le \frac{C_{17}}{B^{30}}\int_{\R^{2}}\left(|\nabla \varepsilon|^{2}+\varepsilon^{2}\right)\left(B\varphi'_{i,B}+\psi_{B}\right)\dd y+C_{18}b^{4}.
\end{equation}
Here, $C_{17}>1$ is  a universal constant independent of $B$ and $C_{18}=C_{18}(B)>1$ is a constant dependent only on $B$.

\smallskip
Based on a similar argument to the one in the estimate of $\mathcal{I}_{1,3,2}$ and $\mathcal{I}_{1,3,4}$, we deduce that 
\begin{equation}\label{est:I135}
\mathcal{I}_{1,3,5}\le\frac{C_{19}}{B^{30}}\int_{\R^{2}}\left(|\nabla \varepsilon|^{2}+\varepsilon^{2}\right)\left(B\varphi'_{i,B}+\psi_{B}\right)\dd y+C_{20}b^{4}.
\end{equation}
Here, $C_{19}>1$ is a universal constant independent of $B$ and $C_{20}=C_{20}(B)>1$ is a constant dependent only on $B$.

\smallskip

Combining estimates~\eqref{est:I131}--\eqref{est:I135}, we conclude that 
\begin{equation}\label{est:I13}
\mathcal{I}_{1,3}\le \frac{C_{21}}{B^{30}}\int_{\R^{2}}\left(|\nabla\varepsilon|^{2}+\varepsilon^{2}\right)\left(B\varphi'_{i,B}+\psi_{B}\right)\dd y+C_{22}b^{4}.
\end{equation}
Here, $C_{21}>1$ is a universal constant independent of $B$ and $C_{22}=C_{22}(B)>1$ is a constant dependent only on $B$.

\smallskip
We see that~\eqref{est:I1} follows from~\eqref{est:I11},~\eqref{est:I12} and~\eqref{est:I13}.

\smallskip
\textbf{Step 2.} Estimate on $\mathcal{I}_{2}$. We claim that, there exist some universal constants $B>100$ large enough and $0<\kappa_{1}<B^{-100}$ small enough, such that 
\begin{equation}\label{est:I2}
\mathcal{I}_{2}\le \frac{C_{23}}{B^{30}}\int_{\R^{2}}\left(|\nabla\varepsilon|^2+\varepsilon^2\right)(B\varphi'_{i,B}+\psi_{B})\dd y+C_{24}b^4.
\end{equation}
Here, $C_{23}>1$ is a universal constant independent of $B$ and $C_{24}=C_{24}(B)>1$ is a constant depending only on $B$. 

Indeed, using~Lemma~\ref{le:equvar} and integration by parts,
\begin{equation*}
\begin{aligned}
\mathcal{I}_{2}
&=2\mathcal{J}_{i,j}\int_{\R^{2}}\left({\rm{Mod}}-\left(\frac{\lambda_{s}}{\lambda}+b\right)\Lambda \varepsilon\right)\varphi_{i,B}\varepsilon\dd y+2\mathcal{J}_{i,j}\int_{\R^{2}}\Psi_{b}\varphi_{i,B}\varepsilon\dd y\\
&-\mathcal{J}_{i,j}\int_{\R^{2}}(\varphi'_{i,B}-\varphi'''_{i,B})\varepsilon^{2}\dd y-
\mathcal{J}_{i,j}\int_{\R^{2}}\varphi'_{i,B}\left(3(\partial_{y_{1}}\varepsilon)^{2}+(\partial_{y_{2}}\varepsilon)^{2}\right)\dd y\\
&-2\int_{\R^{2}}(\mathcal{J}_{i,j}\varphi_{i,B}\partial_{y_{1}}Q_{b}+\psi_{B}\partial_{s}Q_{b})\left(3Q_{b}\varepsilon^{2}+\varepsilon^{3}\right)\dd y\\
&+\frac{1}{2}\mathcal{J}_{i,j}\int_{\R^{2}}\varphi'_{i,B}\left(6Q_{b}^{2}\varepsilon^{2}+8Q_{b}\varepsilon^{3}+3\varepsilon^{4}\right).
\end{aligned}
\end{equation*}
On the other hand, from the definition of $\mathcal{J}_{i,j}$ and~\eqref{est:Boot1},
\begin{equation*}
|\mathcal{J}_{i,j}|\lesssim J_{1}\lesssim \left(B^{10}\mathcal{N}_{2}\right)^{\frac{1}{2}}\lesssim (B^{10}\kappa)^{\frac{1}{2}}.
\end{equation*}
Therefore, using a similar argument to the one in Step 1, we complete the proof of~\eqref{est:I2} by taking $\kappa$ small enough and $B$ large enough.

\smallskip
\textbf{Step 3.} Estimate on $\mathcal{I}_{3}$. We claim that, there exist some universal constants $B>100$ large enough and $0<\kappa_{1}<B^{-100}$ small enough, such that 
\begin{equation}\label{est:I3}
\mathcal{I}_{3}\le \frac{C_{25}}{B^{30}}\int_{\R^{2}}\left(|\nabla\varepsilon|^2+\varepsilon^2\right)(B\varphi'_{i,B}+\psi_{B})\dd y+C_{26}b^4.
\end{equation}
Here, $C_{25}>1$ is a universal constant independent of $B$ and $C_{26}=C_{26}(B)>1$ is a constant depending only on $B$. 

Indeed, by an elementary computation, we decompose 
\begin{equation*}
\mathcal{I}_{3}=\mathcal{I}_{3,1}+\mathcal{I}_{3,2}+\mathcal{I}_{3,3}+\mathcal{I}_{3,4},
\end{equation*}
where
\begin{equation*}
\begin{aligned}
\mathcal{I}_{3,1}&=2\frac{\lambda_{s}}{\lambda}\int_{\R^{2}}\psi_{B}\Lambda Q_{b}\left(3Q_{b}^{2}\varepsilon^{2}+\varepsilon^{3}\right)\dd y,\\
\mathcal{I}_{3,2}&=\frac{\lambda_{s}}{\lambda}\int_{\R^{2}}
\left(\left(2-\theta(j-1)\right)\psi_{B}-y_{1}\psi'_{B}\right)|\nabla \varepsilon|^{2}\dd y,\\
\mathcal{I}_{3,3}&=-\frac{1}{2}\frac{\lambda_{s}}{\lambda}\int_{\R^{2}}
\left(\left(2-\theta(j-1)\right)\psi_{B}-y_{1}\psi'_{B}\right)(6Q^{2}_{b}\varepsilon^{2}+4Q_{b}\varepsilon^{3}+\varepsilon^{4})\dd y,\\
\mathcal{I}_{3,4}&=\left(\frac{\dd }{\dd s}\mathcal{J}_{i,j}-\theta(j-1)\frac{\lambda_{s}}{\lambda}(1+\mathcal{J}_{i,j})\right)\int_{\R^{2}}\varphi_{i,B}\varepsilon^{2}\dd y-\frac{\lambda_{s}}{\lambda}(1+\mathcal{J}_{i,j})\int_{\R^{2}}y_{1}\varphi'_{i,B}\varepsilon^{2}\dd y.
\end{aligned}
\end{equation*}

\smallskip
\emph{Estimate on $\mathcal{I}_{3,1}$.} 
Using~\eqref{est:Boot1},~\eqref{est:psiphi3}, Lemma~\ref{le:modu1} and Lemma~\ref{le:weight2D}, we obtain
\begin{equation}\label{est:I31}
\begin{aligned}
\mathcal{I}_{3,1}
&\lesssim 
\left(|b|+\left(\int_{\R^{2}}\varepsilon^{2}e^{-\frac{|y|}{10}}\dd y\right)^{\frac{1}{2}}\right)
\int_{\R^{2}}\psi_{B}\left(\varepsilon^{2}+\varepsilon^{4}\right)\dd y\\
&\le \frac{C_{27}}{B^{30}}\int_{\R^{2}}\left(|\nabla \varepsilon|^{2}+\varepsilon^{2}\right)\psi_{B}\dd y+C_{28}b^{4}.
\end{aligned}
\end{equation}
Here, $C_{27}>1$ is a universal constant independent of $B$ and $C_{28}=C_{28}(B)>1$ is a constant depending only on $B$. 

\smallskip
\emph{Estimate on $\mathcal{I}_{3,2}$.}
Using~\eqref{est:Boot1},~\eqref{est:psiphi3}, (iv) of Lemma~\ref{le:psiphi2} and (ii) of Lemma~\ref{le:modu1},
\begin{equation}\label{est:I32}
\begin{aligned}
\mathcal{I}_{3,2}
&\lesssim
\left(|b|+\left(\int_{\R^{2}}\varepsilon^{2}e^{-\frac{|y|}{10}}\dd y\right)^{\frac{1}{2}}\right)
\int_{\R^{2}}|\nabla \varepsilon|^{2}(B\varphi'_{i,B}+\psi_{B})\dd y\\
&\le \frac{C_{29}}{B^{30}}\int_{\R^{2}}\left(|\nabla \varepsilon|^{2}+\varepsilon^{2}\right)(B\varphi'_{i,B}+\psi_{B})\dd y+C_{30}b^{4}.
\end{aligned}
\end{equation}
Here, $C_{29}>1$ is a universal constant independent of $B$ and $C_{30}=C_{30}(B)>1$ is a constant depending only on $B$. 

\smallskip
\emph{Estimate on $\mathcal{I}_{3,3}$.}
Using a similar argument to the one above, we obtain
\begin{equation}\label{est:I33}
\begin{aligned}
\mathcal{I}_{3,3}
&\lesssim
\left(|b|+\left(\int_{\R^{2}}\varepsilon^{2}e^{-\frac{|y|}{10}}\dd y\right)^{\frac{1}{2}}\right)
\int_{\R^{2}}\left( \varepsilon^{2}+\varepsilon^{4}\right)(\sqrt{\psi_{B}}+\psi_{B})\dd y\\
&\le \frac{C_{31}}{B^{30}}\int_{\R^{2}}\left(|\nabla \varepsilon|^{2}+\varepsilon^{2}\right)(B\varphi'_{i,B}+\psi_{B})\dd y+C_{32}b^{4}.
\end{aligned}
\end{equation}
Here, $C_{31}>1$ is a universal constant independent of $B$ and $C_{32}=C_{32}(B)>1$ is a constant depending only on $B$. 

\smallskip
\emph{Estimate on $\mathcal{I}_{3,4}$.} According to the integration region,  we decompose 
\begin{equation*}
\mathcal{I}_{3,4}=\mathcal{I}_{3,4,1}+\mathcal{I}_{3,4,2}+\mathcal{I}_{3,4,3},
\end{equation*}
where 
\begin{equation*}
\begin{aligned}
\mathcal{I}_{3,4,1}&=-\frac{\lambda_{s}}{\lambda}(1+\mathcal{J}_{i,j})\int_{\R}\int_{-\infty}^{B^{10}}y_{1}\varphi'_{i,B}\varepsilon^{2}\dd y_{1}\dd y_{2},\\
\mathcal{I}_{3,4,2}&=\left(\frac{\dd }{\dd s}\mathcal{J}_{i,j}-\theta(j-1)\frac{\lambda_{s}}{\lambda}(1+\mathcal{J}_{i,j})\right)\int_{\R}\int_{-\infty}^{B^{10}}\varphi_{i,B}\varepsilon^{2}\dd y_{1}\dd y_{2},
\end{aligned}
\end{equation*}
\begin{equation*}
\begin{aligned}
\mathcal{I}_{3,4,3}&=-\frac{\lambda_{s}}{\lambda}(1+\mathcal{J}_{i,j})\int_{\R}\int_{B^{10}}^{\infty}y_{1}\varphi'_{i,B}\varepsilon^{2}\dd y_{1}\dd y_{2}\\
&+\left(\frac{\dd }{\dd s}\mathcal{J}_{i,j}-\theta(j-1)\frac{\lambda_{s}}{\lambda}(1+\mathcal{J}_{i,j})\right)\int_{\R}\int_{B^{10}}^{\infty}\varphi_{i,B}\varepsilon^{2}\dd y_{1}\dd y_{2}.
\end{aligned}
\end{equation*}
Note that, from~\eqref{est:Boot1}, the definition of $\mathcal{J}_{i,j}$ and Lemma~\ref{le:modu2}, 
\begin{equation*}
\left|\frac{\dd}{\dd s}\mathcal{J}_{i,j}\right|+\left|\mathcal{J}_{i,j}\right|\lesssim 
|J_{1s}|+|J_{1}|
\lesssim \kappa^{\frac{1}{2}}+B^{5}\kappa+B^{5}\kappa^{2}\lesssim B^{-50}.
\end{equation*}
Note also that, for any $B$ large enough and $y_{1}<B^{10}$, we have 
\begin{equation*}
|y_{1}\varphi'_{i,B}|\lesssim |B\varphi'_{i,B}|^{\frac{99}{100}}.
\end{equation*}
It follows from~\eqref{est:Boot1}, (iv) of Lemma~\ref{le:psiphi2}, Lemma~\ref{le:modu1} and the H\"older inequality that
\begin{equation*}
\begin{aligned}
\mathcal{I}_{3,4,1}
&\lesssim
B^{\frac{99}{100}}\|\varepsilon\|^{\frac{1}{100}}_{L^{2}}\left( |b|+\left(\int_{\R^{2}}\varepsilon^{2}e^{-\frac{|y|}{10}}\dd y\right)^{\frac{1}{2}}\right)\left(\int_{\R^{2}}\varepsilon^{2}\varphi'_{i,B}\dd y\right)^{\frac{99}{100}}\\
&\le \frac{C_{33}}{B^{30}}\int_{\R^{2}}(|\nabla \varepsilon|^{2}+\varepsilon^{2})(B\varphi'_{i,B}+\psi_{B})\dd y+C_{34}b^{4},\\
\mathcal{I}_{3,4,2}
&\lesssim \frac{1}{B^{30}}\int_{\R^{2}}(|\nabla \varepsilon|^{2}+\varepsilon^{2})(B\varphi'_{i,B}+\psi_{B})\dd y.
\end{aligned}
\end{equation*}
Here, $C_{33}>1$ is a universal constant independent of $B$ and $C_{34}=C_{34}(B)>1$ is a constant depending only on $B$. 

On the other hand, using the definition of $\varphi_{i,B}$, 
\begin{equation*}
(i+6)\varphi_{i,B}(y_{1})-y_{1}\varphi'_{i,B}(y_{1})=\frac{\psi'_{B}(y_{1})}{\sqrt{2\psi_{B}(y_{1})}}\left(\frac{y_{1}}{B^{10}}\right)^{i+6},\quad \mbox{for}\ y_{1}\ge B^{10}.
\end{equation*}
Based on the above identity, we rewrite the term $\mathcal{I}_{3,4,3}$ by 
\begin{equation*}
\mathcal{I}_{3,4,3}=\mathcal{I}_{3,4,3,1}+\mathcal{I}_{3,4,3,2},
\end{equation*}
where
\begin{equation*}
\begin{aligned}
\mathcal{I}_{3,4,3,1}
&=\frac{1}{i+6}\left(\frac{\dd }{\dd s}\mathcal{J}_{i,j}
-\left(\theta(j-1)+i+6\right)\left(1+\mathcal{J}_{i,j}\right)\frac{\lambda_{s}}{\lambda}\right)\int_{\R}\int_{B^{10}}^{\infty}y_{1}\varphi'_{i,B}\varepsilon^{2}\dd y_{1}\dd y_{2},\\
\mathcal{I}_{3,4,3,2}&=\frac{1}{i+6}\left(\frac{\dd }{\dd s}\mathcal{J}_{i,j}
-\theta(j-1)\left(1+\mathcal{J}_{i,j}\right)\frac{\lambda_{s}}{\lambda}\right)\int_{\R}\int_{B^{10}}^{\infty}\frac{\psi'_{B}}{\sqrt{2\psi_{B}}}\left(\frac{y_{1}}{B^{10}}\right)^{i+6}\varepsilon^{2}\dd y_{1}\dd y_{2}.
\end{aligned}
\end{equation*}

Note that, from the definition of $\mathcal{J}_{i,j}$ and an elementary computation,
\begin{equation*}
\begin{aligned}
&\frac{\dd }{\dd s}\mathcal{J}_{i,j}
-\left(\theta(j-1)+i+6\right)\left(1+\mathcal{J}_{i,j}\right)
\frac{\lambda_{s}}{\lambda}\\
&=\left(\theta(j-1)+i+6\right)\left(1-J_{1}\right)^{-2\theta(j-1)-2i-13}\left(2J_{1s}-\frac{\lambda_{s}}{\lambda}+\frac{\lambda_{s}}{\lambda}J_{1}\right).
\end{aligned}
\end{equation*}
It follows from~\eqref{est:Boot1}, Lemma~\ref{le:modu1} and Lemma~\ref{le:modu2} that
\begin{equation*}
\begin{aligned}
&\left|\frac{\dd }{\dd s}\mathcal{J}_{i,j}
-\left(\theta(j-1)+i+6\right)\left(1+\mathcal{J}_{i,j}\right)
\frac{\lambda_{s}}{\lambda}\right|&\lesssim |b|+B^{5}\int_{\R^{2}}\varepsilon^{2}(B^{10}\varphi'_{i,B}+\psi_{B})\dd y.
\end{aligned}
\end{equation*}
Using~\eqref{est:Boot1} and $\theta>\frac{8}{5}>\frac{100}{93}>\frac{50}{47}$, we have 
\begin{equation*}
(|b|+\mathcal{N}_{2})\left(1+\frac{1}{\lambda^{\theta}}\right)\le 2\kappa\Longrightarrow
(|b|+\mathcal{N}_{2})\left(1+\frac{1}{\lambda^{\frac{8}{5}}}+\frac{1}{\lambda^{\frac{50}{47}}}+\frac{1}{\lambda^{\frac{93}{100}}}\right)\le 8\kappa.
\end{equation*}
Moreover, we have
\begin{equation*}
\begin{aligned}
\frac{|b|}{\lambda^{\frac{50}{47}}}&\lesssim|b|^{\frac{63}{188}}
\kappa^{\frac{125}{188}},\quad 
\frac{\mathcal{N}_{2}}{\lambda^{\frac{50}{47}}}\lesssim\mathcal{N}_{2}^{\frac{63}{188}}
\kappa^{\frac{125}{188}},\\
\frac{|b|}{\lambda^{\frac{100}{93}}}&\lesssim|b|^{\frac{61}{186}}
\kappa^{\frac{125}{186}},\quad 
\frac{\mathcal{N}_{2}}{\lambda^{\frac{100}{93}}}\lesssim\mathcal{N}_{2}^{\frac{61}{186}}
\kappa^{\frac{125}{186}}.
\end{aligned}
\end{equation*}
Next, using again the definition of $\varphi_{i,B}$, we find
\begin{equation*}
\left|y_{1}\varphi'_{i,B}\right|\lesssim B^{-10(i+6)}y_{1}^{i+6},\quad \mbox{on}\ y_{1}>B^{10}.
\end{equation*}
Therefore, for $i=1$, from~\eqref{est:y17} and the Cauchy-Schwarz inequality, we deduce that 
\begin{equation*}
\begin{aligned}
\mathcal{I}_{3,4,3}
&\lesssim 
\left(|b|+B^{5}\int_{\R^{2}}\varepsilon^{2}(B^{10}\varphi'_{1,B}+\psi_{B})\dd y\right)
\int_{\R^{2}}\psi_{B}\varepsilon^2\dd y\\
&+ \left(1+\frac{1}{\lambda^{\frac{50}{47}}}\right)
\left(|b|+B^{5}\int_{\R^{2}}\varepsilon^{2}(B^{10}\varphi'_{1,B}+\psi_{B})\dd y\right)
\left(\int_{\R^{2}}\varphi'_{1,B}\varepsilon^2\dd y\right)^{\frac{93}{94}}\\
&\lesssim \frac{1}{B^{30}}\int_{\R^{2}}\left(|\nabla \varepsilon|^{2}+\varepsilon^{2}\right)\left(B\varphi'_{1,B}+\psi_{B}\right)\dd y+b^{4}.
  \end{aligned}
\end{equation*}
Similarly, for $i=2$, from~\eqref{est:y18} and the Cauchy-Schwarz inequality, we deduce that 
\begin{equation*}
\begin{aligned}
\mathcal{I}_{3,4,3}
&\lesssim 
\left(|b|+B^{5}\int_{\R^{2}}\varepsilon^{2}(B^{10}\varphi'_{2,B}+\psi_{B})\dd y\right)
\int_{\R^{2}}\psi_{B}\varepsilon^2\dd y\\
&+ \left(1+\frac{1}{\lambda^{\frac{100}{93}}}\right)
\left(|b|+B^{5}\int_{\R^{2}}\varepsilon^{2}(B^{10}\varphi'_{2,B}+\psi_{B})\dd y\right)
\left(\int_{\R^{2}}\varphi'_{2,B}\varepsilon^2\dd y\right)^{\frac{92}{93}}\\
&\lesssim \frac{1}{B^{30}}\int_{\R^{2}}\left(|\nabla \varepsilon|^{2}+\varepsilon^{2}\right)\left(B\varphi'_{2,B}+\psi_{B}\right)\dd y+b^{4}.
  \end{aligned}
\end{equation*}
On the other hand, using again~\eqref{est:Boot1}, Lemma~\ref{le:modu1} and Lemma~\ref{le:modu2},
\begin{equation*}
\begin{aligned}
&\left|\frac{\dd }{\dd s}\mathcal{J}_{i,j}
-\theta(j-1)\left(1+\mathcal{J}_{i,j}\right)
\frac{\lambda_{s}}{\lambda}\right|\\
&\lesssim |b|+\left(\int_{\R^{2}}\varepsilon^{2}e^{-\frac{|y|}{10}}\dd y\right)^{\frac{1}{2}}+B^{5}\int_{\R^{2}}\varepsilon^{2}(B^{10}\varphi'_{i,B}+\psi_{B})\dd y.
\end{aligned}
\end{equation*}

Moreover, from the definition of $\psi_{B}$, we see that 
\begin{equation*}
\left|\frac{\psi'_{B}}{\sqrt{2\psi_{B}}}\left(\frac{y_{1}}{B^{10}}\right)^{i+6}\right|=\frac{1}{B}\exp\left(-2\left(\frac{y_{1}}{B^{\frac{2}{3}}}+\frac{1}{3}B^{\frac{1}{3}}\right)\right)\left(\frac{y_{1}}{B^{10}}\right)^{i+6}\lesssim B^{-30}\varphi'_{i,B}.
\end{equation*}
It follows that 
\begin{equation*}
\mathcal{I}_{3,4,3,2}\lesssim \frac{1}{B^{30}}\int_{\R^{2}}\varepsilon^{2}\varphi'_{i,B}\dd y.
\end{equation*}
Combining the above estimates, we obtain
\begin{equation}\label{est:I34}
\mathcal{I}_{3,4}
\le \frac{C_{35}}{B^{30}}\int_{\R^{2}}\left(|\nabla \varepsilon|^{2}+\varepsilon^{2}\right)(B\varphi'_{i,B}+\psi_{B})\dd y+C_{36}b^{4}.
\end{equation}
Here, $C_{35}>1$ is a universal constant independent of $B$ and $C_{36}=C_{36}(B)>1$ is a constant depending only on $B$. 

\smallskip
We see that~\eqref{est:I3} follows from~\eqref{est:I31},~\eqref{est:I32},~\eqref{est:I33} and \eqref{est:I34}.

\smallskip
\textbf{Step 4.} Conclusion. Combining the estimates~\eqref{equ:dsFij},~\eqref{est:I1},~\eqref{est:I2} and~\eqref{est:I3}, we complete the proof of~\eqref{est:dsFij} by taking $B$ large enough.
\end{proof}

\subsection{Virial estimate}\label{SS:VIRIAL}
In this subsection, we introduce the virial estimate related to the solution of~\eqref{CP}. As we mentioned in \S\ref{SS:Comm}, since the lack of coercivity of the Schr\"odinger operator appears in the primal virial estimate of $\varepsilon$, we should first introduce a transformed problem, and then based on the special structure of the transformed linearized problem and numerical computation, we could obtain the coercivity and virial estimates for this transformed problem (see more details in Proposition~\ref{Prop:Virial},~Lemma~\ref{le:coerl} and~\cite[Lemma 14.2 and \S16]{FHRY}).

\smallskip
We define the smooth function $\sigma\in [0,1]$ as follows,
\begin{equation*}
\sigma(y_1)=
\begin{cases}
0,&\text{ for } |y_{1}|>2,\\
1,&\text{ for } |y_{1}|\le 1.
\end{cases}
\end{equation*}
Moreover, we define the smooth function $\psi_{0}\in (0,1]$ as follows,
\begin{equation*}
\psi_{0}(y_1)=
\begin{cases}
e^{6y_1},&\text{ for }y_1<-1,\\
\frac{1}{2},&\text{ for }y_1>-\frac{1}{2},
\end{cases}
\quad \mbox{with}\ \  \psi_0'(y_1)\ge 0,\ \ \mbox{on}\ \mathbb{R}.
\end{equation*}
We also define the smooth function $\psi_{1}\in (0,\infty)$ as follows,
\begin{equation*}
\psi_{1}(y_1)=
\begin{cases}
e^{10y_1},&\text{ for }y_1<-1,\\
1+y_{1},&\text{ for }y_1>-\frac{1}{2},
\end{cases}
\quad \mbox{with}\ \  \psi_1'(y_1)> 0,\ \ \mbox{on}\ \mathbb{R}.
\end{equation*}
Let $B>100$ be a large enough universal constant to be chosen later. We set 
\begin{equation*}
\psi_{0,B}(y_1)=\psi_0\left(\frac{y_1}{B}\right)\quad \mbox{and}\quad 
\psi_{1,B}(y_1)=\psi_1\left(\frac{y_1}{B}\right).
\end{equation*}
We also set 
\begin{equation*}
\chi_B(y_1)=
\begin{cases}
\sigma\left(\frac{y_1}{2B}\right)\int_0^{y_1}\frac{2}{B}\psi_{0,B}(\rho)\dd \rho,\quad\quad  \mbox{for}\ y_{1}\le 0,\\
\sigma\left(\frac{y_1}{10B^{10}}\right)\int_0^{y_1}\frac{2}{B}\psi_{0,B}(\rho)\dd \rho,\quad \mbox{for}\ y_{1}> 0.
\end{cases}
\end{equation*}
By the definition of the weight functions, we have the following pointwise estimate.
\begin{lemma}\label{le:pointchi}
The following estimates hold.
\begin{enumerate}
\item \emph{Estimates on $\varphi_{i,B}$.} For $i=1,2$, we have 
\begin{equation*}
\begin{aligned}
|\varphi''_{i,B}|&\lesssim B^{-\frac{2}{3}}\varphi'_{i,B}+B^{-20}\psi_{0,B},\quad \mbox{on}\ \R,\\
|\varphi'''_{i,B}|&\lesssim B^{-\frac{4}{3}}\varphi'_{i,B}+B^{-30}\psi_{0,B},\quad \mbox{on}\ \R.
\end{aligned}
\end{equation*}

\item \emph{Estimates on $\psi_{0,B}$ and $\chi_{B}$.} We have 
\begin{equation*}
\begin{aligned}
B|\chi'_{B}|+B^{2}|\chi''_{B}|+B^{3}\left|\chi'''_{B}\right|&\lesssim \psi_{0,B},\quad \mbox{on}\ \R,\\
B\left|\psi'_{0,B}\right|+B^{2}\left|\psi''_{0,B}\right|+B^{3}|\psi'''_{0,B}|&\lesssim \psi_{0,B},\quad \mbox{on}\ \R.
\end{aligned}
\end{equation*}

\item \emph{First-type estimates on $\chi_{B}$.} We have 
\begin{equation*}
\chi_{B}\lesssim \min\left(B^{9}\psi_{0,B},\psi_{1,B}\sqrt{\psi_{0,B}}\right),\quad \mbox{on}\ \R.
\end{equation*}

\item \emph{Second-type estimates on $\chi_{B}$.} We have 
\begin{equation*}
\left|\chi'_{B}-\frac{2}{B}\psi_{0,B}\right|\lesssim B^{9} \varphi'_{i,B},\quad \mbox{on}\ \R.
\end{equation*}

\item \emph{Third-type estimates on $\chi_{B}$.} We have 
\begin{equation*}
\begin{aligned}
\left|\chi'_{B}-\frac{2}{B}\psi_{0,B}\right|&\lesssim \mathbf{1}_{\left(-\infty,-\frac{B}{2}\right]}(y_{1})+\mathbf{1}_{\left[B^{10},\infty\right)}(y_{1}),\quad \quad \quad \ \ \mbox{on}\ \R,\\
\left|\chi_{B}-\frac{2y_{1}}{B}\psi_{0,B}\right|&\lesssim \left(\mathbf{1}_{\left(-\infty,-\frac{B}{2}\right]}(y_{1})+\mathbf{1}_{\left[B^{10},\infty\right)}(y_{1})\right)|y_{1}|,\quad \mbox{on}\ \R.
\end{aligned}
\end{equation*}
\end{enumerate}
\end{lemma}
\begin{proof}
The proof is directly based on a similar argument to the proof for Lemma~\ref{le:psiphi}--\ref{le:psiphi3} and the definitions of $\chi_{B}$ and $\psi_{0,B}$, and we omit it.
\end{proof}

Let $0<\gamma\ll 1$ be a small enough constant (depending on $B$) to be chosen later. For any $s\in [0,s_{0})$, we set 
\begin{equation*}
\eta=\left(1-\gamma\Delta\right)^{-1}\mathcal{L}\varepsilon\quad \mbox{and}\quad 
\mathcal{P}(s)=\int_{\R^{2}}\eta^2(s,y)\chi_B(y_1)\dd y.
\end{equation*}
Note that, for any $\gamma>0$, we have 
\begin{equation}\label{equ:Q2gamma}
\left[Q^{2},(1-\gamma \Delta)\right]=\gamma\Delta (Q^{2})+2\gamma Q\nabla Q\cdot \nabla.
\end{equation}
We denote 
\begin{equation}\label{equ:modeta}
\begin{aligned}
{\rm{Mod}}_{\eta}
&=\frac{\lambda_{s}}{\lambda}\Lambda \varepsilon+\left(\frac{\lambda_{s}}{\lambda}+b\right)\left(\Lambda Q_{b}-\Lambda Q\right)\\
&+\left(\frac{x_{1s}}{\lambda}-1\right)\left(\partial_{y_{1}}Q_{b}-\partial_{y_{1}}Q+\partial_{y_{1}}\varepsilon\right)\\
&+\frac{x_{2s}}{\lambda}\left(\partial_{y_{2}}Q_{b}-\partial_{y_{2}}Q+\partial_{y_{2}}\varepsilon\right)-b_{s}\frac{\partial Q_{b}}{\partial{b}}.
\end{aligned}
\end{equation}
We now state the equation and the orthogonality conditions of $\eta$.
\begin{lemma}[Equation of $\eta$]\label{le:equeta}
We have
\begin{equation*}
\begin{aligned}
\partial_{s}\eta
&=\mathcal{L}\partial_{y_{1}}\eta-3\gamma(1-\gamma\Delta)^{-1}\left(\Delta(Q^{2})\partial_{y_{1}}\eta+2Q\nabla Q\cdot \nabla \partial_{y_{1}}\eta\right)\\
&+(1-\gamma\Delta)^{-1}{\mathcal{L}}{\rm{Mod}}_{\eta}-2\left(\frac{\lambda_{s}}{\lambda}+b\right)(1-\gamma\Delta)^{-1}Q\\
&+(1-\gamma\Delta)^{-1}\left({\mathcal{L}}\Psi_{b}-{\mathcal{L}}\partial_{y_{1}}R_{b}-{\mathcal{L}}\partial_{y_{1}}R_{NL}\right).
\end{aligned}
\end{equation*}
Moreover, the function $\eta$ satisfies the following orthogonality conditions
\begin{equation*}
\left(\eta,(1-\gamma\Delta)Q\right)=\left(\eta,(1-\gamma\Delta)\partial_{y_{1}}Q\right)=\left(\eta,(1-\gamma\Delta)\partial_{y_{2}}Q\right)=0.
\end{equation*}
\end{lemma}

\begin{proof}
The proof is based on~\eqref{equ:orth},~\eqref{equ:Q2gamma},~\eqref{equ:modeta}, Proposition~\ref{Prop:Spectral}, Lemma~\ref{le:equvar} and an elementary computation.
\end{proof}

In addition, based on the Fourier transform and elementary computations, we have the following identity related to $\varepsilon$ and $\eta$.
\begin{lemma}\label{le:vareeta}
It holds
\begin{equation*}
\begin{aligned}
&(1-\gamma\Delta)^{-1}\mathcal{L}\Lambda\varepsilon-\Lambda\eta\\
&=3(1-\gamma\Delta)^{-1}\left(\varepsilon y\cdot\nabla \left(Q^{2}\right)\right)\\
&+2\gamma(1-\gamma \Delta)^{-2}\Delta\mathcal{L}\varepsilon-2(1-\gamma\Delta)^{-1}\Delta\varepsilon.
\end{aligned}
\end{equation*}
\end{lemma}

\begin{proof}
First, we claim that, for any regular function $f$ on $\R^{2}$, 
\begin{equation}\label{equ:Lambdayf}
\left[(1-\gamma\Delta)^{-1},y\right]\cdot \nabla f=2\gamma\left(1-\gamma\Delta\right)^{-2}f.
\end{equation}
Indeed, using the Fourier transform, we have 
\begin{equation*}
\begin{aligned}
&\mathcal{F}\left(\left[(1-\gamma\Delta)^{-1},y\right]\cdot \nabla f\right)(\xi)\\
&=\nabla _{\xi}\cdot\left(\frac{\xi \widehat{f}(\xi)}{1+\gamma|\xi|^{2}}\right)-\left(\frac{1}{1+\gamma|\xi|^{2}}\right)\nabla_{\xi}\cdot \left(\xi\widehat{f}(\xi)\right)\\
&=\nabla_{\xi}\left(\frac{1}{1+\gamma|\xi|^{2}}\right)\cdot \left(\xi\widehat{f}(\xi)\right)=-\frac{2\gamma|\xi|^{2}\widehat{f}(\xi)}{(1+\gamma|\xi|^{2})^{2}}=2\gamma\mathcal{F}\left((1-\gamma\Delta)^{-2}\Delta f\right),
\end{aligned}
\end{equation*}
which implies~\eqref{equ:Lambdayf}. Moreover, from~\eqref{equ:Lambdayf}, we see that 
\begin{equation}\label{equ:Lambdayf2}
y\cdot \nabla f-(1-\gamma \Delta)y\cdot \left(\nabla \left(1-\gamma \Delta\right)^{-1}f\right)=2\gamma(1-\gamma\Delta)^{-1}\Delta f.
\end{equation}
Note that, from the definition of $\eta$, we have 
\begin{equation*}
\mathcal{L}\Lambda \varepsilon=(1-\gamma\Delta)\Lambda \eta -2\Delta\varepsilon
+3\varepsilon y\cdot\nabla \left(Q^{2}\right)+y\cdot \nabla \mathcal{L}\varepsilon-(1-\gamma\Delta)(y\cdot\nabla \eta).
\end{equation*}
Using~\eqref{equ:Lambdayf2} and the definition of $\eta$, we deduce that
\begin{equation*}
y\cdot \nabla \mathcal{L}\varepsilon-(1-\gamma\Delta)(y\cdot\nabla \eta)=2\gamma (1-\gamma\Delta)^{-1}\Delta \mathcal{L}\varepsilon.
\end{equation*}
Combining the above two identities, we complete the proof of Lemma~\ref{le:vareeta}.
\end{proof}

Then, using the coercivity of $\mathcal{L}$, we obtain the following relations between $\varepsilon$ and $\eta$.
\begin{lemma}\label{le:relationvareta}
Let $B>100$ be a large enough constant and $0<\gamma\ll 1$ be a small enough constant. Then we have 
\begin{equation*}
\begin{aligned}
&\int_{\mathbb{R}^2}(\gamma|\nabla\eta|^2+\eta^{2})\psi_{0,B}\dd y\le C\int_{\mathbb{R}^2}(\gamma^{-1}|\nabla \varepsilon|^2+ \varepsilon^2)\psi_{0,B}\dd y,\\
&\int_{\mathbb{R}^2}(|\nabla \varepsilon|^2+ \varepsilon^2)\psi_{0,B}\dd y\le  C\int_{\mathbb{R}^2}(\gamma^2|\nabla\eta|^2+\eta^{2})\psi_{0,B}\dd y,\\
&\int_{\mathbb{R}^2}(\gamma|\nabla\eta|^2+\eta^{2})\psi_{1,B}\dd y\le C\int_{\mathbb{R}^2}(\gamma^{-1}|\nabla \varepsilon|^2+ \varepsilon^2)\psi_{1,B}\dd y.
\end{aligned}
\end{equation*}
Here, $C>1$ is a large enough universal constant independent of $B$ and $\gamma$.
\end{lemma}
\begin{proof}
First, from integration by parts, we deduce that 
\begin{equation*}
\begin{aligned}
&\int_{\R^{2}}\left((1-\gamma\Delta)\eta\right)\eta \psi_{0,B}\dd y
=\int_{\R^{2}}\left(\gamma|\nabla \eta|^{2}+\eta^{2}\right)\psi_{0,B}\dd y-\frac{\gamma}{2}\int_{\R^{2}}\eta^{2}\psi''_{0,B}\dd y,\\
&\int_{\R^{2}}\left(\mathcal{L}\varepsilon\right)\eta \psi_{0,B}\dd y=\int_{\R^{2}}\left(\nabla \varepsilon\cdot\nabla \eta+(1-3Q^{2})\varepsilon\eta\right)\psi_{0,B}\dd y+\int_{\R^{2}}(\partial_{y_{1}}\varepsilon)\eta \psi'_{0,B}\dd y.
\end{aligned}
\end{equation*}
Combining the above identities with $B|\psi'_{0,B}|+B^{2}|\psi''_{0,B}|\lesssim \psi_{0,B}$  on $\R$ and Cauchy-Schwarz inequality, we complete the proof of the first estimate.
 
Next, using again integration by parts, we see that 
\begin{equation*}
\begin{aligned}
\int_{\R^{2}}\left(\mathcal{L}\varepsilon\right)\varepsilon\psi_{0,B}\dd y
&=\left(\mathcal{L}\left(\varepsilon\sqrt{\psi_{0,B}}\right),\varepsilon\sqrt{\psi_{0,B}}\right)
+\frac{1}{2}\int_{\R^{2}}\varepsilon^{2}\psi''_{0,B}\dd y\\
&+\int_{\R^{2}}(\partial_{y_{1}}\varepsilon)\varepsilon\psi'_{0,B}\dd y-\frac{1}{4}\int_{\R^{2}}\varepsilon^{2}\frac{(\psi'_{0,B})^{2}}{\psi_{0,B}}\dd y.
\end{aligned}
\end{equation*}
It follows from~\eqref{equ:orth}, Proposition~\ref{Prop:Spectral} and $B|\psi'_{0,B}|+B^{2}|\psi''_{0,B}|\lesssim \psi_{0,B}$  on $\R$ that 
\begin{equation*}
\int_{\R^{2}}\left((1-\gamma\Delta)\eta\right)\varepsilon\psi_{0,B}\dd y=\int_{\R^{2}}\left(\mathcal{L}\varepsilon\right)\varepsilon\psi_{0,B}\dd y\ge \frac{\nu}{2}\int_{\R^{2}}(|\nabla \varepsilon|^{2}+\varepsilon^{2})\psi_{0,B}\dd y.
\end{equation*}
Combining the above estimate with the  Cauchy-Schwarz inequality, we complete the proof of the second estimate.

The proof of the third estimate is similar to the case of the first one.
\end{proof}

Based on a similar argument and Lemma~\ref{le:pointchi}, we obtain the following estimate.
\begin{lemma}\label{le:etapsiB}
Let $B>100$ be a large enough constant and $0<\gamma\ll 1$ be a small enough constant. Then for all $i=1,2,$ we have 
\begin{equation*}
\begin{aligned}
\int_{\mathbb{R}^2}(\gamma|\nabla\eta|^2+\eta^{2})\varphi'_{i,B}\dd y
&\le C\int_{\mathbb{R}^2}(\gamma^{-1}|\nabla \varepsilon|^2+ \varepsilon^2)\varphi'_{i,B}\dd y\\
&+\frac{C}{B^{20}}\int_{\mathbb{R}^2}
\left(|\nabla \varepsilon|^2+ \varepsilon^2+\eta^{2}\right)\psi_{0,B}\dd y.
\end{aligned}
\end{equation*}
Here, $C>1$ is a large enough universal constant independent of $B$ and $\gamma$.
\end{lemma}

We now state the virial estimate of $\eta$. Let $B>100$ be a large enough constant to be chosen later and $\gamma=B^{-3}$. Then the following qualitative estimate of the time variation of $\mathcal{P}$ is true.
\begin{proposition}\label{Prop:Virial}
There exist some universal constants $B>100$ large enough, $0<\kappa_{1}< \min\left\{\kappa^{*},B^{-100}\right\}$ small enough  and $0<\nu_{1}<1$ small enough {\rm{(}}independent of $B${\rm{)}} such that the following holds. Assume that for all $s\in[0,s_0]$, the solution $u(t)$ with initial data $u_{0}$ satisfies the bootstrap assumption \eqref{est:Boot1}--\eqref{est:Boot3} with $0<\kappa<\kappa_{1}$. Then for all $(i,j)\in\left\{1,2\right\}^{2}$ and $s\in[0,s_0]$, we have 
\begin{equation}\label{est:dsP}
    	\begin{aligned}
	&\lambda^{\theta(j-1)}\frac{\dd}{\dd s}\left(\frac{\mathcal{P}}{\lambda^{\theta(j-1)}}\right)+\frac{\nu_{1}}{B}\int_{\R^{2}}\left(|\nabla \eta|^2+\eta^2\right)\psi_{0,B}\dd y\\
&\le \frac{C_{37}}{B^{8}}\int_{\R^{2}}\left(|\nabla \varepsilon|^2+\varepsilon^2\right)\left(B^{23}\varphi'_{i,B}+\psi_{0,B}\right)\dd y+C_{38}b^{4}.
	\end{aligned}
	\end{equation}
Here, $C_{37}>1$ is a universal constant independent of $B$ and $C_{38}=C_{38}(B)>1$ is a constant depending only on $B$.

\end{proposition}

To complete the proof of Proposition~\ref{Prop:Virial}, we first recall the following coercivity result from~\cite{FHRY} and the introduce a technical estimate related to the weighted norm.

\smallskip
For any $f\in H^{1}(\R^{2})$, we denote 
\begin{equation*}
\begin{aligned}
\mathcal{A}f=&-\frac{3}{2}\partial_{y_1}^2f-\frac{1}{2}\partial_{y_2}^2f+\frac{1}{2}f-\left(\frac{3}{2}Q^2+3y_1Q\partial_{y_{1}}Q\right)f\\
&+3\frac{(f,y_1Q)}{(Q,Q)}Q^{2}\partial_{y_{1}}Q+3\frac{(f,Q^2\partial_{y_{1}}Q)}{(Q,Q)}y_{1}Q.
\end{aligned}
\end{equation*}

We now recall the following coercivity result of $\mathcal{A}$ from \cite{FHRY}.
\begin{lemma}[\cite{FHRY}]\label{le:coerl}
There exists $\nu_{2}>0$, such that for all $f\in H^1(\mathbb{R}^2)$, 
\begin{equation*}
\left(\mathcal{A}f,f\right)\ge \nu_{2}\|f\|^2_{H^1}-\frac{1}{\nu_{2}}\left(\left(f,Q\right)^2+
\left(f,\partial_{y_{1}}Q\right)^{2}+\left(f,\partial_{y_2}Q\right)^2\right).
\end{equation*}
\end{lemma}
\begin{proof}
We refer to~\cite[\S16]{FHRY} for the numerical checking of the coercivity result.
\end{proof}

Next we introduce the following weighted estimate.
\begin{lemma}\label{le:weight2D2}
Let $\omega:\mathbb{R}^2\rightarrow (0,\infty)$ be a $C^2$ function such that 
\begin{equation}\label{est:omega2}
\left\|\frac{\nabla \omega}{\omega}\right\|_{L^{\infty}(\R^{2})}+
\sum_{|\alpha|=2}\left\|\frac{\partial_{y}^{\alpha}\omega}{\omega}\right\|_{L^{\infty}(\R^{2})}\lesssim 1.
\end{equation}
Then, for all $ f\in H^2(\mathbb{R}^2)$ and $k=0,1,2$, we have
\begin{equation*}
\sum_{|\alpha|=k}\|\omega(1-\gamma\Delta)^{-1}\partial_{y}^{\alpha} f\|_{L^2(\mathbb{R}^2)}\leq C\gamma^{-\frac{k}{2}} \|\omega f\|_{L^2(\mathbb{R}^2)}.
\end{equation*}
Here, $C$ is a universal constant independent of $\gamma$.
\end{lemma}
\begin{proof}
First, from the Fourier transform and the Cauchy-Schwarz inequality, we have
\begin{equation}\label{est:L2L2}
\sum_{|\alpha|=k}\|(1-\gamma\Delta)^{-1}\partial_{y}^\alpha\|_{L^2\rightarrow L^2}\lesssim \gamma^{-\frac{k}{2}},\quad \mbox{for all}\ k=0,1,2.
\end{equation}
Second, for any $\alpha\in \mathbb{N}^{2}$ with $0\le |\alpha|\le 2$, we denote 
\begin{equation*}
F_{\alpha 1}=\omega(1-\gamma\Delta)^{-1}\partial_{y}^{\alpha}f\quad \mbox{and}\quad 
F_{\alpha 2}=(1-\gamma\Delta)^{-1}\omega\partial_{y}^{\alpha}f.
\end{equation*}
By an elementary computation, we find
\begin{equation*}
\frac{1}{\omega}(1-\gamma\Delta)F_{\alpha 2}=\frac{1}{\omega}(1-\gamma\Delta)F_{\alpha 1}-\gamma F_{\alpha 1}\Delta \left(\frac{1}{\omega}\right)-2\gamma \nabla \left(\frac{1}{\omega}\right)\cdot\nabla F_{\alpha 1},
\end{equation*}
which implies
\begin{equation*}
\begin{aligned}
F_{\alpha 1}
&=F_{\alpha 2}-\gamma (1-\gamma\Delta)^{-1}\left(\omega F_{\alpha 1} \Delta \left(\frac{1}{\omega}\right)\right)\\
&+2\gamma (1-\gamma\Delta)^{-1}
\left(\nabla \cdot \left(\omega F_{\alpha 1}\nabla \left(\frac{1}{\omega}\right)\right)\right)\\
&-2\gamma (1-\gamma\Delta)^{-1}\left(F_{\alpha 1}\nabla \omega\cdot \nabla \left(\frac{1}{\omega}\right)\right).
\end{aligned}
\end{equation*}
Next, from~\eqref{est:omega2}, 
\begin{equation*}
\omega\left|\Delta\left(\frac{1}{\omega}\right)\right|+\omega\left|\nabla\left(\frac{1}{\omega}\right)\right|+\left|\nabla\omega \cdot \nabla\left(\frac{1}{\omega}\right)\right|\lesssim 1.
\end{equation*}
It follows from~\eqref{est:L2L2} that
\begin{equation*}
\|F_{\alpha 1}\|_{L^2}\lesssim \|F_{\alpha 2}\|_{L^2}+\gamma\|F_{\alpha 1}\|_{L^2}+\gamma^{\frac{1}{2}}\|F_{\alpha 1}\|_{L^2}\Longrightarrow \|F_{\alpha 1}\|_{L^2}\lesssim \|F_{\alpha 2}\|_{L^2}.
\end{equation*}
On the other hand, for any $\alpha\in \mathbb{N}^{2}$ with $0\le |\alpha|\le 2$, 
\begin{equation*}
\omega\partial_{y}^{\alpha}f\in {\rm{Span}}\left\{\partial_{y}^{\alpha_{1}}(f\partial_{y}^{\alpha_{2}}\omega):\alpha_{1}+\alpha_{2}=\alpha\right\}.
\end{equation*}
It follows from~\eqref{est:omega2} and~\eqref{est:L2L2} that
\begin{equation*}
\left\|F_{\alpha 2}\right\|_{L^{2}}
\lesssim \sum_{\alpha_{1}+\alpha_{2}=\alpha}\left\|(1-\gamma\Delta)^{-1}\partial_{y}^{\alpha_{1}}(f\partial_{y}^{\alpha_{2}}\omega)\right\|_{L^{2}}
\lesssim
 \sum_{|\alpha_{2}|\le \alpha}\left\|f\partial_{y}^{\alpha_{2}}\omega\right\|_{L^{2}}\lesssim \|\omega f\|_{L^{2}}.
\end{equation*}
Combining the above estimates, we complete the proof of Lemma~\ref{le:weight2D2}.
\end{proof}

Note that, the functions $\psi_{0,B}$ and $\psi_{1,B}$ satisfy
\begin{equation}\label{est:psi1B}
\begin{aligned}
\left\|\frac{\nabla \psi_{0,B}}{\psi_{0,B}}\right\|_{L^{\infty}(\R^{2})}+
\sum_{|\alpha|=2}\left\|\frac{\partial_{y}^{\alpha}\psi_{0,B}}{\psi_{0,B}}\right\|_{L^{\infty}(\R^{2})}&\lesssim 1,\\
\left\|\frac{\nabla \psi_{1,B}}{\psi_{1,B}}\right\|_{L^{\infty}(\R^{2})}+
\sum_{|\alpha|=2}\left\|\frac{\partial_{y}^{\alpha}\psi_{1,B}}{\psi_{1,B}}\right\|_{L^{\infty}(\R^{2})}&\lesssim 1.
\end{aligned}
\end{equation}

We now give a complete proof of Proposition~\ref{Prop:Virial}.
\begin{proof}[Proof of Proposition~\ref{Prop:Virial}]
From Lemma~\ref{le:equeta}, we decompose
\begin{equation}\label{equ:dsP}
\frac{\lambda^{\theta(j-1)}}{2}\frac{\dd}{\dd s}\left(\frac{\mathcal{P}}{\lambda^{\theta(j-1)}}\right)
=\frac{1}{2}\frac{\dd \mathcal{P}}{\dd s}-\frac{\theta(j-1)}{2}\frac{\lambda_{s}}{\lambda}\mathcal{P}
=\mathcal{G}_{1}+\mathcal{G}_{2}+\mathcal{G}_{3}+\mathcal{G}_{4},
\end{equation}
where 
\begin{equation*}
\begin{aligned}
\mathcal{G}_{1}&=\frac{\lambda_{s}}{\lambda}\int_{\R^{2}}\left((1-\gamma\Delta)^{-1}\mathcal{L}\Lambda\varepsilon\right)\eta \chi_{B}\dd y-\frac{\theta(j-1)}{2}\frac{\lambda_{s}}{\lambda}\mathcal{P},\\
\mathcal{G}_{2}&=-3\gamma\int_{\R^{2}}\left((1-\gamma\Delta)^{-1}\left(\Delta(Q^{2})\partial_{y_{1}}\eta+2Q\nabla Q\cdot\nabla \partial_{y_{1}}\eta\right)\right)\eta \chi_{B}\dd y,\\
\mathcal{G}_{3}&=\int_{\R^{2}}\left(\mathcal{L}\partial_{y_{1}}\eta\right)\eta \chi_{B}\dd y
-2\left(\frac{\lambda_{s}}{\lambda}+b\right)
\int_{\R^{2}}\left((1-\gamma\Delta)^{-1}Q\right)\eta \chi_{B}\dd y,\\
\mathcal{G}_{4}&=\int_{\R^{2}}\left((1-\gamma\Delta)^{-1}
\mathcal{L}\left({\rm{Mod}_{\eta}}-\frac{\lambda_{s}}{\lambda}\Lambda \varepsilon+\Psi_{b}-\partial_{y_{1}}R_{b}-\partial_{y_{1}}R_{NL}\right)\right)\eta \chi_{B}\dd y.
\end{aligned}
\end{equation*}

\textbf{Step 1.} Estimate on $\mathcal{G}_{1}$. We claim that, there exist some universal constants $B>100$ large enough and $0<\kappa_{1}<B^{-100}$ small enough, such that 
\begin{equation}\label{est:G1}
\begin{aligned}
\mathcal{G}_{1}
&\le \frac{C_{39}}{B^{30}}\int_{\R^{2}}\left(|\nabla\eta|^2+\eta^2\right)\psi_{0,B}\dd y\\
&+\frac{C_{40}}{B^{30}}\int_{\R^{2}}\left(|\nabla\varepsilon|^2+\varepsilon^2\right)(B\varphi'_{i,B}+\psi_{0,B})\dd y.
\end{aligned}
\end{equation}
Here, $C_{39}>1$ and $C_{40}>1$ are some universal constants independent of $B$.

Indeed, from Lemma~\ref{le:vareeta} and integration by parts, we see that 
\begin{equation*}
\begin{aligned}
&\int_{\R^{2}}\left((1-\gamma\Delta)^{-1}\mathcal{L}\Lambda\varepsilon\right)\eta \chi_{B}\dd y\\
&=-\frac{1}{2}\int_{\R^{2}}\eta ^{2}\chi'_{B}\dd y+3\int_{\R^{2}}\left((1-\gamma\Delta)^{-1}\left(\varepsilon y\cdot\nabla \left(Q^{2}\right)\right)\right)\eta \chi_{B}\dd y\\
&+2\gamma\int_{\R^{2}}\left((1-\gamma\Delta)^{-2}\Delta\mathcal{L}\varepsilon\right)\eta \chi_{B}\dd y-2\int_{\R^{2}}\left((1-\gamma\Delta)^{-1}\Delta \varepsilon\right)\eta \chi_{B}\dd y.
\end{aligned}
\end{equation*}
From~Lemma~\ref{le:pointchi}, we directly have
\begin{equation*}
\left|\int_{\R^{2}}\eta^{2}\chi'_{B}\dd y\right|\lesssim \frac{1}{B}\int_{\R^{2}}\eta^{2}\psi_{0,B}\dd y.
\end{equation*}
Then, using again Lemma~\ref{le:pointchi}, Lemma~\ref{le:weight2D2} and the Cauchy-Schwarz inequality, 
\begin{equation*}
\begin{aligned}
&\left|\int_{\R^{2}}\left((1-\gamma\Delta)^{-1}\left(\varepsilon y\cdot\nabla \left(Q^{2}\right)\right)\right)\eta \chi_{B}\dd y\right|\\
&\lesssim B^{9}\left( \int_{\R^{2}}\varepsilon^{2}\psi_{0,B}\dd y\right)^{\frac{1}{2}}\left(\int_{\R^{2}}\eta^{2}\psi_{0,B}\dd y\right)^{\frac{1}{2}}\\
&\lesssim B^{9}\int_{\R^{2}}\eta^{2}\psi_{0,B}\dd y+B^{9}\int_{\R^{2}}
\left(|\nabla \varepsilon|^{2}+\varepsilon^{2}\right)\psi_{0,B}\dd y.
\end{aligned}
\end{equation*}
Based on a similar argument, we also obtain 
\begin{equation*}
\begin{aligned}
&\gamma\left|\int_{\R^{2}}\left((1-\gamma\Delta)^{-2}\Delta\mathcal{L}\varepsilon\right)\eta \chi_{B}\dd y\right|+\left|\int_{\R^{2}}\left((1-\gamma\Delta)^{-1}\Delta \varepsilon\right)\eta \chi_{B}\dd y\right|\\
&\lesssim B^{12}\int_{\R^{2}}\left(|\nabla \eta|^{2}+\eta^{2}\right)\psi_{0,B}\dd y+B^{12}\int_{\R^{2}}\left(|\nabla \varepsilon|^{2}+\varepsilon^{2}\right)\psi_{0,B}\dd y.
\end{aligned}
\end{equation*}
On the other hand, using again Lemma~\ref{le:pointchi} and Lemma~\ref{le:relationvareta},
\begin{equation*}
\left|\mathcal{P}\right|
\lesssim B^{9} \int_{\R^{2}}\eta^{2}\psi_{0,B}\dd y\\
\lesssim B^{12}\int_{\R^{2}}\left(|\nabla \varepsilon|^{2}+\varepsilon^{2}\right)\psi_{0,B}\dd y.
\end{equation*}
We see that~\eqref{est:G1} follows from the above estimates,~\eqref{est:Boot1} and (ii) of Lemma~\ref{le:modu1}.

\smallskip
\textbf{Step 2.} Estimate on $\mathcal{G}_{2}$.
We claim that, there exist some universal constants $B>100$ large enough and $0<\kappa_{1}<B^{-100}$ small enough, such that 
\begin{equation}\label{est:G2}
\begin{aligned}
\mathcal{G}_{2}
&\le \frac{C_{41}}{B^{\frac{3}{2}}}\int_{\R^{2}}\left(|\nabla\eta|^2+\eta^2\right)\psi_{0,B}\dd y.
\end{aligned}
\end{equation}
Here, $C_{41}>1$ is a universal constants independent of $B$.

Indeed, using~\eqref{est:psi1B}, Lemma~\ref{le:pointchi} and Lemma~\ref{le:weight2D2}, we find
\begin{equation*}
\begin{aligned}
&\left|\int_{\R^{2}}\left((1-\gamma\Delta)^{-1}\left(\Delta(Q^{2})\partial_{y_{1}}\eta\right)\right)\eta \chi_{B}\dd y\right|\\
&\lesssim
\left\||\nabla \eta|\Delta(Q^{2})\psi_{1,B}\right\|_{L^{2}}
\left\|\eta\sqrt{\psi_{0,B}}\right\|_{L^{2}}\lesssim 
 \int_{\R^{2}}\left(|\nabla \eta|^{2}+\eta^{2}\right)\psi_{0,B}\dd y.
\end{aligned}
\end{equation*}
Next, by an elementary computation,
\begin{equation*}
2Q\nabla Q\cdot \nabla \partial_{y_{1}}\eta
=2\partial_{y_{1}}\left(Q\nabla Q\cdot \nabla \eta\right)-2\nabla \eta \cdot \partial_{y_{1}}\left(Q\nabla Q\right).
\end{equation*}
It follows from~\eqref{est:psi1B}, Lemma~\ref{le:pointchi} and Lemma~\ref{le:weight2D2} that 
\begin{equation*}
\begin{aligned}
&\left|\int_{\R^{2}}\left((1-\gamma\Delta)^{-1}\left(2Q\nabla Q\cdot\nabla\partial_{y_{1}}\eta\right)\right)\eta \chi_{B}\dd y\right|\\
&\lesssim \left(\gamma^{-\frac{1}{2}}\left\|\left(Q\nabla Q\cdot\nabla \eta\right)\psi_{1,B}\right\|_{L^{2}}+\left\|\left(\nabla \eta\cdot\partial_{y_{1}}(Q\nabla Q)\right)\psi_{1,B}\right\|_{L^{2}}\right)\left\|\eta \sqrt{\psi_{0,B}}\right\|_{L^{2}}\\
&\lesssim \gamma^{-\frac{1}{2}}
\left(\int_{\R^{2}}|\nabla \eta|^{2}\psi_{0,B}\dd y\right)^{\frac{1}{2}}
\left(\int_{\R^{2}}\eta^{2}\psi_{0,B}\dd y\right)^{\frac{1}{2}}\lesssim 
B^{\frac{3}{2}} \int_{\R^{2}}\left(|\nabla \eta|^{2}+\eta^{2}\right)\psi_{0,B}\dd y.
\end{aligned}
\end{equation*}
We see that~\eqref{est:G2} follows from the above estimates and $\gamma=B^{-3}$.

\smallskip
\textbf{Step 3.} Estimate on $\mathcal{G}_{3}$. We claim that, there exist some universal constants $B>100$ large enough and $0<\kappa_{1}<B^{-100}$ small enough, such that 
\begin{equation}\label{est:G3}
\begin{aligned}
\mathcal{G}_{3}
&\le -\frac{\nu_{3}}{B}\int_{\R^{2}}\left(|\nabla\eta|^2+\eta^2\right)\psi_{0,B}\dd y\\
&+\frac{C_{42}}{B^{8}}\int_{\R^{2}}\left(|\nabla\varepsilon|^2+\varepsilon^2\right)
(B^{23}\varphi'_{i,B}+\psi_{0,B})\dd y+C_{43}b^{4}.
\end{aligned}
\end{equation}
Here, $\nu_{3}>0$ and $C_{42}>1$ are some universal constants independent of $B$ and $C_{43}=C_{43}(B)$ is a constant depending only on $B$. 

\smallskip
Indeed, by an elementary computation,
\begin{equation*}
\begin{aligned}
&\int_{\R^{2}}\left(\mathcal{L}\partial_{y_{1}}\eta\right)\eta \chi_{B}\dd y\\
&=-\frac{3}{2}\int_{\R^{2}}(\partial_{y_{1}}\eta)^{2}\chi'_{B}\dd y
-\frac{1}{2}\int_{\R^{2}}\left(\partial_{y_{2}}\eta\right)^{2}\chi'_{B}\dd y
-\frac{1}{2}\int_{\R^{2}}\eta^{2}\chi'_{B}\dd y\\
&+3\int_{\R^{2}}\left(Q\partial_{y_{1}}Q\right)\eta^{2}\chi_{B}\dd y
+\frac{3}{2}\int_{\R^{2}}Q^{2}\eta^{2}\chi'_{B}\dd y+\frac{1}{2}\int_{\R^{2}}\eta^{2}\chi'''_{B}\dd y.
\end{aligned}
\end{equation*}

Based on the above identity, we rewrite the term $\mathcal{G}_{3}$ by 
\begin{equation*}
\mathcal{G}_{3}=\mathcal{G}_{3,1}+\mathcal{G}_{3,2}+\mathcal{G}_{3,3},
\end{equation*}
where
\begin{equation*}
\begin{aligned}
\mathcal{G}_{3,1}=\frac{6}{B}\frac{(\eta,Q^2\partial_{y_{1}}Q)}{(Q,Q)}(\eta,y_1Q)-2\left(\frac{\lambda_s}{\lambda}+b\right) \left((1-\gamma\Delta)^{-1}Q,\eta\chi_{B}\right),\qquad \qquad \quad 
\end{aligned}
\end{equation*}
\begin{equation*}
\begin{aligned}
\mathcal{G}_{3,2}&=
-\frac{1}{B}\int_{\R^{2}}\left(3(\partial_{y_{1}}\eta)^{2}+\left(\partial_{y_{2}}\eta\right)^{2}\right)\psi_{0,B}\dd y
-\frac{1}{B}\int_{\R^{2}}\eta^{2}\psi_{0,B}\dd y\\
&+\frac{3}{B}\int_{\R^{2}}\left(Q^{2}+2y_{1}Q\partial_{y_{1}}Q\right)\eta^{2}\psi_{0,B}\dd y
-\frac{6}{B}\frac{(\eta,Q^2\partial_{y_{1}}Q)}{(Q,Q)}(\eta,y_1Q),\qquad \qquad \quad 
\end{aligned}
\end{equation*}
\begin{equation*}
\begin{aligned}
\mathcal{G}_{3,3}&=
-\frac{1}{2}\int_{\R^{2}}\left(3(\partial_{y_{1}}\eta)^{2}+\left(\partial_{y_{2}}\eta\right)^{2}+\eta^{2}\right)\left(\chi'_{B}-\frac{2}{B}\psi_{0,B}\right)\dd y+\frac{1}{2}\int_{\R^{2}}\eta^{2}\chi'''_{B}\dd y
\\
&+\frac{3}{2}\int_{\R^{2}}\eta^{2}Q^{2}\left(\chi'_{B}-\frac{2}{B}\psi_{0,B}\right)\dd y
+3\int_{\R^{2}}\eta^{2}\left(Q\partial_{y_{1}}Q\right)\left(\chi_{B}-\frac{2y_{1}}{B}\psi_{0,B}\right)\dd y.
\end{aligned}
\end{equation*}

\smallskip
\emph{Estimate on $\mathcal{G}_{3,1}$.} Using again Lemma~\ref{le:equeta}, we have
\begin{equation*}
\begin{aligned}
&\left(\mathcal{L}\partial_{y_{1}}\eta,(1-\gamma\Delta)Q\right)\\
&=2\left(\frac{\lambda_{s}}{\lambda}+b\right)(Q,Q)+\left(\mathcal{L}\partial_{y_{1}}R_{b}+\mathcal{L}\partial_{y_{1}}R_{NL},Q\right)\\
&+3\gamma\left(\Delta (Q^{2})\partial_{y_{1}}\eta+2Q\nabla Q\cdot\nabla \partial_{y_{1}}\eta,Q\right)
-\left(\mathcal{L}{\rm{Mod}}_{\eta}+\mathcal{L}\Psi_{b},Q\right).
\end{aligned}
\end{equation*}
Note that 
\begin{equation*}
\left(\mathcal{L}\partial_{y_{1}}\eta,(1-\gamma\Delta)Q\right)
=6\left(\eta,Q^{2}\partial_{y_{1}}Q\right)+\gamma\left(\eta,\partial_{y_{1}}\mathcal{L}\Delta Q\right).
\end{equation*}
Therefore, from Lemma~\ref{le:modu1} and the Cauchy-Schwarz inequality,
\begin{equation*}
\left|\frac{\lambda_s}{\lambda}+b-3\frac{(\eta,Q^{2}\partial_{y_{1}}Q)}{(Q,Q)}\right|
\lesssim \gamma\left(\int_{\R^{2}}|\eta|^2\psi_{0,B}\dd y\right)^{\frac{1}{2}}+b^2+\int_{\R^{2}}\varepsilon^{2}\psi_{0,B}\dd y.
\end{equation*}
It follows from~\eqref{est:Boot1} and Lemma~\ref{le:relationvareta} that 
\begin{equation*}
\begin{aligned}
&\left(\gamma\left(\int_{\R^{2}}\eta^{2}\psi_{0,B}\dd y\right)^{\frac{1}{2}}+b^2+\int_{\R^{2}}\varepsilon^2\psi_{0,B}\dd y\right)
\left(\int_{\R^{2}}\eta^2\psi_{0,B}\dd y\right)^{\frac{1}{2}}\\
&\le \left((\gamma+\|\varepsilon\|_{L^{2}})\left(\int_{\R^{2}}\left(|\nabla \eta|^{2}+\eta^{2}\right)\psi_{0,B}\dd y\right)^{\frac{1}{2}}+b^{2}\right)\left(\int_{\R^{2}}\eta^2\psi_{0,B}\dd y\right)^{\frac{1}{2}}\\
&\le \frac{C_{44}}{B^3}\int_{\R^{2}}(|\nabla \eta|^2+|\eta|^2)\psi_{0,B}\dd y+C_{45}b^4.
\end{aligned}
\end{equation*}
Here, $C_{44}>1$ is a universal constant independent of $B$ and $C_{45}=C_{45}(B)>1$ is a constant depending only on $B$. In the above estimate, we use $\gamma=B^{-3}$.

\smallskip
Note also that 
\begin{equation*}
(1-\gamma\Delta)^{-1}Q=Q+\gamma(1-\gamma\Delta)^{-1}\Delta Q.
\end{equation*}
Using~\eqref{est:psi1B},~Lemma~\ref{le:pointchi} and the Cauchy-Schwarz inequality, we find
\begin{equation*}
\begin{aligned}
&\left|\int_{\R^{2}}\left((1-\gamma\Delta)^{-1} \Delta Q\right)\eta \chi_{B}\dd y\right|\\
&\lesssim \int_{\R^{2}}\left|\psi_{1,B}\left((1-\gamma\Delta)^{-1} \Delta Q\right)\right|\eta \sqrt{\psi_{0,B}}\dd y\lesssim \left(\int_{\R^{2}}\eta^{2} \psi_{0,B}\dd y\right)^{\frac{1}{2}}.
\end{aligned}
\end{equation*}
Therefore, from $\chi_{B}(y_1)=\frac{y_{1}}{B}$ for $|y_{1}|\le \frac{B}{2}$ and $\gamma=B^{-3}$, we conclude that 
\begin{equation}\label{est:G31}
\begin{aligned}
\mathcal{G}_{3,1}
&\lesssim \left| (\eta,Q^{2}\partial_{y_{1}}Q)\left(\eta Q,\chi_{B}-\frac{y_{1}}{B}\right)\right|\\
&+\left|\left(\frac{\lambda_{s}}{\lambda}+b-3\frac{(\eta,Q^{2}\partial_{y_{1}}Q)}{(Q,Q)}
\right)\left(\eta Q,\chi_{B}\right)\right|\\
&+\gamma\left|\left(\frac{\lambda_{s}}{\lambda}+b\right)
\left((1-\gamma\Delta)^{-1}\Delta Q,\eta \chi_{B}\right)\right|\\
&\le \frac{C_{46}}{B^3}\int_{\R^{2}}(|\nabla \eta|^2+\eta^2)\psi_{0,B}\dd y+C_{47}b^4.
\end{aligned}
\end{equation}
Here, $C_{46}>1$ is a universal constant independent of $B$ and $C_{47}=C_{47}(B)>1$ is a constant depending only on $B$.

\smallskip
\emph{Estimate on $\mathcal{G}_{3,2}$.} By an elementary computation, we rewrite
\begin{equation*}
\begin{aligned}
\mathcal{G}_{3,2}
&=-\frac{2}{B}\left(\mathcal{A}\eta \sqrt{\psi_{0,B}},\eta \sqrt{\psi_{0,B}}\right)
+\frac{3}{4B}\int_{\R^{2}}\eta^{2}\left(\frac{(\psi'_{0,B})^{2}}{\psi_{0,B}}-2\psi''_{0,B}\right)\dd y\\
&+\frac{6}{B(Q,Q)}\left(\left(\eta\sqrt{2\psi_{0,B}},y_{1}Q\right)\left(\eta\sqrt{2\psi_{0,B}},Q^{2}\partial_{y_{1}}Q\right)-\left(\eta,y_{1}Q\right)\left(\eta,Q^{2}\partial_{y_{1}}Q\right)\right).
\end{aligned}
\end{equation*}
First, using Lemma~\ref{le:coerl}, we deduce that 
\begin{equation*}
\begin{aligned}
\left(\mathcal{A}\eta \sqrt{\psi_{0,B}},\eta \sqrt{\psi_{0,B}}\right)
&\ge \nu_{2}\left\|\eta \sqrt{\psi_{0,B}}\right\|_{H^{1}}^{2}-\frac{1}{\nu_{2}}\left(\eta\sqrt{\psi_{0,B}},Q\right)^{2}\\
&-\frac{1}{\nu_{2}}\left(\eta\sqrt{\psi_{0,B}},\partial_{y_{1}}Q\right)^{2}-\frac{1}{\nu_{2}}\left(\eta\sqrt{\psi_{0,B}},\partial_{y_{2}}Q\right)^{2}.
\end{aligned}
\end{equation*}
On the one hand side, from Lemma~\ref{le:pointchi}, we see that 
\begin{equation*}
\left\|\eta \sqrt{\psi_{0,B}}\right\|_{H^{1}}^{2}=\left(1+O\left(\frac{1}{B}\right)\right)
\int_{\R^{2}}\left(\left|\nabla \eta\right|^{2}+\eta^{2}\right)\psi_{0,B}\dd y.
\end{equation*}
On the other hand, from Lemma~\ref{le:equeta} and the definition of $\psi_{0,B}$, 
\begin{equation*}
\left(\eta\sqrt{\psi_{0,B}},Q\right)^{2}+
\left(\eta\sqrt{\psi_{0,B}},\partial_{y_{1}}Q\right)^{2}+
\left(\eta\sqrt{\psi_{0,B}},\partial_{y_{2}}Q\right)^{2}
\lesssim \frac{1}{B^{6}}\int_{\R^{2}}\eta^{2}\psi_{0,B}\dd y.
\end{equation*}
Second, using again Lemma~\ref{le:pointchi}, we have
\begin{equation*}
\left|\frac{1}{B}\int_{\R^{2}}\eta^{2}\left(\frac{(\psi'_{0,B})^{2}}{\psi_{0,B}}-2\psi''_{0,B}\right)\dd y\right|
\lesssim \frac{1}{B^{2}}\int_{\R^{2}}\eta^{2}\psi_{0,B}\dd y.
\end{equation*}

Last, using again the definition of $\psi_{0,B}$ and the exponential decay of $Q$, 
\begin{equation*}
\begin{aligned}
\left|\left(\eta\sqrt{2\psi_{0,B}},y_{1}Q\right)-\left(\eta,y_{1}Q\right)\right|
&\lesssim  \frac{1}{B^{2}}\left(\int_{\R^{2}}\eta^{2}\psi_{0,B}\dd y\right)^{\frac{1}{2}},\\
\left|\left(\eta\sqrt{2\psi_{0,B}},Q^{2}\partial_{y_{1}}Q\right)-\left(\eta,Q^{2}\partial_{y_{1}}Q\right)\right|
&\lesssim \frac{1}{B^{2}}\left(\int_{\R^{2}}\eta^{2}\psi_{0,B}\dd y\right)^{\frac{1}{2}}.
\end{aligned}
\end{equation*}
Based on the above estimates and the exponential decay of $Q$, we obtain
\begin{equation*}
\begin{aligned}
&\left|\left(\eta\sqrt{2\psi_{0,B}},y_{1}Q\right)\left(\eta\sqrt{2\psi_{0,B}},Q^{2}\partial_{y_{1}}Q\right)-\left(\eta,y_{1}Q\right)\left(\eta,Q^{2}\partial_{y_{1}}Q\right)\right|\\
&\lesssim \frac{1}{B^{2}}\left(\int_{\R^{2}}\eta^{2}\psi_{0,B}\dd y\right)^{\frac{1}{2}}\left(\int_{\R^{2}}\eta^{2}e^{-\frac{|y|}{10}}\dd y\right)\lesssim \frac{1}{B^{2}}\int_{\R^{2}}\eta^{2}\psi_{0,B}\dd y.
\end{aligned}
\end{equation*}
Combining the above estimates, for $B$ large enough, we conclude that 
\begin{equation}\label{est:G32}
\mathcal{G}_{3,2}\le -\frac{\nu_{2}}{2B}\int_{\R^{2}}\left(\left|\nabla \eta\right|^{2}+\eta^{2}\right)\psi_{0,B}\dd y.
\end{equation}

\emph{Estimate on $\mathcal{G}_{3,3}$.} Recall that, from Lemma~\ref{le:pointchi}, we have
\begin{equation*}
\left|\chi'_{B}-\frac{2}{B}\psi_{0,B}\right|\lesssim B^{9}\varphi'_{i,B},\quad \mbox{on}\ \R.
\end{equation*}
It follows from Lemma~\ref{le:etapsiB} that
\begin{equation*}
\begin{aligned}
&\left|\int_{\R^{2}}\left(3(\partial_{y_{1}}\eta)^{2}+\left(\partial_{y_{2}}\eta\right)^{2}+\eta^{2}\right)\left(\chi'_{B}-\frac{2}{B}\psi_{0,B}\right)\dd y\right|\\
&\lesssim 
B^{15}\int_{\R^{2}}\left(|\nabla \varepsilon|^{2}+\varepsilon^{2}\right)\varphi'_{i,B}\dd y
+\frac{1}{B^{8}}\int_{\R^{2}}\left(|\nabla \varepsilon|^{2}+\varepsilon^{2}+\eta^{2}\right)\psi_{0,B}\dd y.
\end{aligned}
\end{equation*}
Next, using again Lemma~\ref{le:pointchi} and the exponential decay of $Q$, we deduce that 
\begin{equation*}
\begin{aligned}
&\left|\int_{\R^{2}}\eta^{2}\chi'''_{B}\dd y\right|
+\left|\int_{\R^{2}}\eta^{2}Q^{2}\left(\chi'_{B}-\frac{2}{B}\psi_{0,B}\right)\dd y\right|\\
&+\left|\int_{\R^{2}}\eta^{2}\left(Q\partial_{y_{1}}Q\right)\left(\chi_{B}-\frac{2y_{1}}{B}\psi_{0,B}\right)\dd y\right|\lesssim \frac{1}{B^{3}}\int_{\R^{2}}\eta^{2}\psi_{0,B}\dd y.
\end{aligned}
\end{equation*}
Combining the above estimates, we obtain
\begin{equation}\label{est:G33}
\begin{aligned}
\mathcal{G}_{3,3}
&\le \frac{C_{48}}{B^{3}}\int_{\R^{2}}\left(\left|\nabla \eta\right|^{2}+\eta^{2}\right)\psi_{0,B}\dd y\\
&+\frac{C_{49}}{B^{8}}\int_{\R^{2}}\left(\left|\nabla \varepsilon\right|^{2}+\varepsilon^{2}\right)
\left(B^{23}\varphi'_{i,B}+\psi_{0,B}\right)\dd y.
\end{aligned}
\end{equation}
Here, $C_{48}>1$ and $C_{49}>1$ are some universal constants independent of $B$.

We see that~\eqref{est:G3} follows from~\eqref{est:G31},~\eqref{est:G32},~\eqref{est:G33} and $B$ large enough.

\smallskip
\textbf{Step 4.} Estimate on $\mathcal{G}_{4}$. We claim that, there exist some universal constants $B>100$ large enough and $0<\kappa_{1}<B^{-100}$ small enough, such that 
\begin{equation}\label{est:G4}
\begin{aligned}
\mathcal{G}_{4}
&\le \frac{C_{50}}{B^{10}}\int_{\R^{2}}\left(|\nabla\eta|^2+\eta^2\right)\psi_{0,B}\dd y\\
&+\frac{C_{51}}{B^{10}}\int_{\R^{2}}\left(|\nabla\varepsilon|^2+\varepsilon^2\right)
(B\varphi'_{i,B}+\psi_{0,B})\dd y+C_{52}b^{4}.
\end{aligned}
\end{equation}
Here, $C_{50}>1$ and $C_{51}>1$ are some universal constants independent of $B$ and $C_{52}=C_{52}(B)$ is a constant depending only on $B$. 

\smallskip
Note that, from the definition of ${\rm{Mod}}_{\eta}$ in~\eqref{equ:modeta}, we have 
\begin{equation*}
\begin{aligned}
&\left|{\rm{Mod}_{\eta}}-\frac{\lambda_{s}}{\lambda}\Lambda \varepsilon\right|\\
&\lesssim
 \left(b^{2}+\left(\int_{\R^{2}}\varepsilon^{2}e^{-\frac{|y|}{10}}\dd y\right)^{\frac{1}{2}}\right)|\nabla \varepsilon|\\
 &+\left(b^{2}+\int_{\R^{2}}\varepsilon^{2}e^{-\frac{|y|}{10}}\dd y\right)\left(e^{-\frac{|y_{2}|}{3}}\mathbf{1}_{[-2,0]}(|b|^{\frac{3}{4}}y_{1})+e^{-\frac{|y|}{3}}\right)\\
&+ |b|(1+|y|)\left(b^{2}+\left(\int_{\R^{2}}\varepsilon^{2}e^{-\frac{|y|}{10}}\dd y\right)^{\frac{1}{2}}\right)\left(e^{-\frac{|y_{2}|}{3}}\mathbf{1}_{[-2,0]}(|b|^{\frac{3}{4}}y_{1})+e^{-\frac{|y|}{3}}\right).
\end{aligned}
\end{equation*}
Based on the above estimate, we deduce that
\begin{equation*}
\begin{aligned}
&\left\|\left({\rm{Mod}}_{\eta}-\frac{\lambda_{s}}{\lambda}\Lambda \varepsilon\right)\sqrt{\psi_{0,B}}\right\|^{2}_{L^{2}}\\
&\lesssim \left(b^{4}+\int_{\R^{2}}\varepsilon^{2}e^{-\frac{|y|}{10}}\dd y\right)\int_{\R^{2}}|\nabla \varepsilon|^{2}\psi_{0,B}\dd y\\
 &+\left(b^{2}+\int_{\R^{2}}\varepsilon^{2}e^{-\frac{|y|}{10}}\dd y\right)^{2}\int_{\R^{2}}\left(e^{-\frac{2|y_{2}|}{3}}\mathbf{1}_{[-2,0]}(|b|^{\frac{3}{4}}y_{1})+e^{-\frac{2|y|}{3}}\right)\psi_{0,B}\dd y\\
&+ b^{2}\left(b^{4}+\int_{\R^{2}}\varepsilon^{2}e^{-\frac{|y|}{10}}\dd y\right)\int_{\R^{2}}(1+|y|^{2})\left(e^{-\frac{2|y_{2}|}{3}}\mathbf{1}_{[-2,0]}(|b|^{\frac{3}{4}}y_{1})+e^{-\frac{2|y|}{3}}\right)\psi_{0,B}\dd y.
\end{aligned}
\end{equation*}
It follows from~\eqref{est:Boot1} and $\psi_{0,B}\lesssim \psi_{B}$ that 
\begin{equation*}
\left\|\left({\rm{Mod}}_{\eta}-\frac{\lambda_{s}}{\lambda}\Lambda \varepsilon\right)\sqrt{\psi_{0,B}}\right\|_{L^{2}}\lesssim B^{\frac{1}{2}}b^{2}+\frac{1}{B^{30}}\left\|\varepsilon\sqrt{\psi_{0,B}}\right\|_{L^{2}}.
\end{equation*}
Based on a similar argument and (ii) of Lemma~\ref{le:Ketest}, we see that 
\begin{equation*}
\left\|\Psi_{b} \sqrt{\psi_{0,B}}\right\|_{L^{2}}
\lesssim\left( \int_{\R^{2}}\Psi^{2}_{b}\psi_{0,B}\dd y\right)^{\frac{1}{2}}
\lesssim B^{\frac{1}{2}}b^{2}.
\end{equation*}
On the other hand, from the definitions of $R_{b}$ and $R_{NL}$ in Lemma~\ref{le:equvar},
\begin{equation*}
\left|R_{b}+R_{NL}\right|\lesssim |b||\varepsilon|+|b|\varepsilon^{2}+|\varepsilon|^{3}.
\end{equation*}
Therefore, from the 2D Sobolev embedding estimate, we have 
\begin{equation*}
\begin{aligned}
\left\|\left(R_{b}+R_{NL}\right)\sqrt{\psi_{0,B}}\right\|_{L^{2}}
&\lesssim |b|\left\|\varepsilon\sqrt{\psi_{0,B}}\right\|_{L^{2}}+\left\|\varepsilon\psi^{\frac{1}{6}}_{0,B}\right\|^{3}_{H^{1}}\\
&\lesssim \frac{1}{B^{30}}\left(\int_{\R^{2}}\left(|\nabla\varepsilon|^{2}+\varepsilon^{2}\right)\left(B\varphi'_{i,B}+\psi_{0,B}\right)\dd y\right)^{\frac{1}{2}}.
\end{aligned}
\end{equation*}
Here, we use the fact that 
\begin{equation*}
\psi_{0,B}^{\frac{1}{3}}\lesssim \psi_{B}\lesssim \varphi_{i,B}\quad \mbox{and}\quad 
\psi_{0,B}^{\frac{1}{3}}\lesssim B\varphi'_{i,B}+\psi_{0,B}.
\end{equation*}

We see that~\eqref{est:G4} follows from the above estimates, Lemma~\ref{le:pointchi} and Lemma~\ref{le:weight2D2}.

\smallskip
\textbf{Step 5.} Conclusion. Combining the estimates~\eqref{est:G1},\eqref{est:G2},~\eqref{est:G3} with~\eqref{est:G4}, we complete the proof of Proposition~\ref{Prop:Virial}.

\end{proof}

\subsection{Energy-Virial Lyapunov functional}\label{SS:Mon}
In this subsection, we introduce the energy-virial Lyapunov functional $\mathcal{M}_{i,j}$ which will play a crucial role in closing the energy estimate in the bootstrap setting~\eqref{est:Boot1}--\eqref{est:Boot3}. We mention here that, the construction of this functional is based on the combination of the energy and virial quantity that we defined in \S\ref{SS:ENERGY} and \S\ref{SS:VIRIAL}, respectively.

\smallskip
For $\left(i,j\right)\in \left\{1,2\right\}^{2}$, we define
\begin{equation*}
\mathcal{M}_{i,j}=\mathcal{F}_{i,j}+\frac{1}{B^{20}}\mathcal{P}.
\end{equation*}
\begin{proposition}\label{Prop:Mon}
There exist some universal constants $B>100$ large enough and $0<\kappa_{1}<B^{-100}$ small enough such that the following holds. Assume that for all $s\in[0,s_0]$, the solution $u(t)$ with initial data $u_{0}$ satisfies the bootstrap assumption \eqref{est:Boot1}--\eqref{est:Boot3} with $0<\kappa<\kappa_{1}$. Then for all $(i,j)\in\left\{1,2\right\}^{2}$ and $s\in[0,s_0]$, the following estimates are true.
\begin{enumerate}
	\item \emph{Coercivity.} It holds
	\begin{equation*}
	\mathcal{N}_i\lesssim\mathcal{M}_{i,j}\lesssim \mathcal{N}_i.
	\end{equation*}
\item {\emph{Monotonicity formula}.} There exists a universal constant $D>0$ (independent of $B$), such that
\begin{equation*}
\lambda^{\theta(j-1)}\frac{\dd}{\dd s}
\left(\frac{\mathcal{M}_{i,j}}{\lambda^{\theta(j-1)}}\right)+\frac{D}{B^{27}}\mathcal{N}_{i-1}\le C_{53}b^{4}.
\end{equation*}
Here, $C_{53}=C_{53}(B)>1$ is a constant dependent only on $B$.
\end{enumerate}
\end{proposition}

\begin{proof}

Proof of (i). On the one hand side, from (iii) of Lemma~\ref{le:pointchi} and Lemma~\ref{le:relationvareta},
\begin{equation}\label{est:PL21}
\left|\mathcal{P}\right|\lesssim B^{9}\int_{\R^{2}}\eta^{2}\psi_{0,B}\dd y\lesssim B^{16}\int_{\R^{2}}\left(|\nabla \varepsilon|^{2}+\varepsilon^{2}\right)\psi_{0,B}\dd y\lesssim B^{16} \mathcal{N}_{i}.
\end{equation}
Here, we use the fact that $\psi_{0,B}\lesssim \min(\varphi_{i,B},\psi_{B})$ and the definition of $\mathcal{N}_{i}$.

\smallskip
On the other hand side, we decompose 
\begin{equation*}
\mathcal{F}_{i,j}=\mathcal{N}_{i}
-3\int_{\R^{2}}Q_{b}^{2}\varepsilon^{2}\psi_{B}\dd y
+\mathcal{J}_{i,j}\int_{\R^{2}}\varepsilon^{2}\varphi_{i,B}\dd y
-2\int_{\R^{2}}Q_{b}\varepsilon^{3}\psi_{B}\dd y-\frac{1}{2}\int_{\R^{2}}\varepsilon^{4}\psi_{B}\dd y.
\end{equation*}
First, by an elementary computation,
\begin{equation*}
\begin{aligned}
\mathcal{N}_{i}-3\int_{\R^{2}}Q_{b}^{2}\varepsilon^{2}\psi_{B}\dd y
&=\left(\mathcal{L}(\varepsilon\sqrt{\psi_{B}}),\varepsilon\sqrt{\psi_{B}}\right)+\frac{1}{4}\int_{\R^{2}}\varepsilon^{2}\left(\frac{(\psi'_{B})^{2}}{\psi_{B}}-2\psi''_{B}\right)\dd y\\
&+\int_{\R^{2}}\varepsilon^{2}(\varphi_{i,B}-\psi_{B})\dd y-3\int_{\R^{2}}\left(Q^{2}_{b}-Q^{2}\right)\varepsilon^{2}\psi_{B}\dd y.
\end{aligned}
\end{equation*}
Therefore, from (iv) of Proposition~\ref{Prop:Spectral}, (i) of Lemma~\ref{le:psiphi2} and~\eqref{equ:orth}, we deduce that 
\begin{equation}\label{est:e21}
\begin{aligned}
\mathcal{N}_{i}-3\int_{\R^{2}}Q_{b}^{2}\varepsilon^{2}\psi_{B}\dd y
&\ge \frac{\nu}{2}\int_{\R^{2}}\left(|\nabla \varepsilon|^{2}+\varepsilon^{2}\right)\psi_{B}\dd y\\
&+\int_{\R^{2}}\varepsilon^{2}(\varphi_{i,B}-\psi_{B})\dd y\ge \frac{1}{2}\min(1,\nu)\mathcal{N}_{i}.
\end{aligned}
\end{equation}
Second, from~\eqref{est:Boot1}, the definition of $\mathcal{J}_{i,j}$ in~\eqref{equ:defJij} and (iii) of Lemma~\ref{le:modu1}, 
\begin{equation}\label{est:e22}
\left|\mathcal{J}_{i,j}\right|\lesssim \left|J_{1}\right|\lesssim \frac{1}{B^{10}}\Longrightarrow
\left|\mathcal{J}_{i,j}\int_{\R^{2}}\varepsilon^{2}\varphi_{i,B}\dd y\right|\lesssim \frac{1}{B^{10}}\mathcal{N}_{i}.
\end{equation}
Next, using again the 2D Sobolev embedding and (i) of Lemma~\ref{le:psiphi2}, 
\begin{equation}\label{est:e3}
\begin{aligned}
\left|\int_{\R^{2}}Q_{b}\varepsilon^{3}\psi_{B}\dd y\right|
&\lesssim \|\varepsilon\|_{L^{2}}\left\|\varepsilon\sqrt{\psi_{B}}\right\|_{H^{1}}^{2}\\
&\lesssim \|\varepsilon\|_{L^{2}}\int_{\R^{2}}\left(|\nabla \varepsilon|^{2}+\varepsilon^{2}\right)\psi_{B}\dd y\lesssim \frac{1}{B^{10}}\mathcal{N}_{i}.
\end{aligned}
\end{equation}
Last, from~\eqref{est:weight2}, we obtain
\begin{equation}\label{est:e4}
\left|\int_{\R^{2}}\varepsilon^{4}\psi_{B}\dd y\right|\lesssim \|\varepsilon\|_{L^{2}}^{2}
\int_{\R^{2}}\left(|\nabla \varepsilon|^{2}+\varepsilon^{2}\right)\psi_{B}\dd y\lesssim \frac{1}{B^{10}}\mathcal{N}_{i}.
\end{equation}
Combining~\eqref{est:PL21},~\eqref{est:e21},~\eqref{est:e22},~\eqref{est:e3} with~\eqref{est:e4}, we complete the proof of (i).

\smallskip
Proof of (ii). From Proposition~\ref{Prop:dsFij} and Proposition~\ref{Prop:Virial}, we see that 
\begin{equation*}
\begin{aligned}
&\lambda^{\theta(j-1)}\frac{\dd }{\dd s}\left(\frac{\mathcal{M}_{i,j}}{\lambda^{\theta(j-1)}}\right)
+\frac{1}{4}\int_{\R^{2}}\left(|\nabla \varepsilon|^{2}+\varepsilon^{2}\right)\varphi'_{i,B}\dd y
+\frac{\nu_{1}}{B^{21}}\int_{\R^{2}}\left(|\nabla \eta|^{2}+\eta^{2}\right)\psi_{0,B}\dd y\\
&\le \frac{C_{54}}{B^{5}}
\int_{\R^{2}}\left(|\nabla \varepsilon|^{2}+\varepsilon^{2}\right)\varphi'_{i,B}\dd y
+\frac{C_{55}}{B^{28}}\int_{\R^{2}}\left(|\nabla \varepsilon|^{2}+\varepsilon^{2}\right)
\psi_{B}\dd y+C_{56}b^{4}.
\end{aligned}
\end{equation*}
Here, $C_{54}>1$ and $C_{55}>1$ are some universal constants independent of $B$ and $C_{56}=C_{56}(B)>1$ is a constant dependent only on $B$. Therefore, from Lemma~\ref{le:relationvareta}, there exists a universal constant $\nu_{4}>0$ (independent of $B$) such that 
\begin{equation*}
\begin{aligned}
&\lambda^{\theta(j-1)}\frac{\dd }{\dd s}\left(\frac{\mathcal{M}_{i,j}}{\lambda^{\theta(j-1)}}\right)
+\frac{1}{5}\int_{\R^{2}}\left(|\nabla \varepsilon|^{2}+\varepsilon^{2}\right)\varphi'_{i,B}\dd y\\
&+\frac{\nu_{4}}{B^{27}}\int_{\R^{2}}\left(|\nabla \varepsilon|^{2}+\varepsilon^{2}\right)\psi_{0,B}\dd y
\le
\frac{C_{55}}{B^{28}}\int_{\R^{2}}\left(|\nabla \varepsilon|^{2}+\varepsilon^{2}\right)
\psi_{B}\dd y+C_{56}b^{4}.
\end{aligned}
\end{equation*}
Based on the above estimate and $\psi_{B}\lesssim B\varphi'_{i,B}+\psi_{0,B}$ on $\R$, we deduce that 
\begin{equation*}
\lambda^{\theta(j-1)}\frac{\dd }{\dd s}\left(\frac{\mathcal{M}_{i,j}}{\lambda^{\theta(j-1)}}\right)
+
\int_{\R^{2}}\left(|\nabla \varepsilon|^{2}+\varepsilon^{2}\right)\left(\frac{1}{6}\varphi'_{i,B}+\frac{\nu_{4}}{2B^{27}}\psi_{0,B}\right)\dd y\le C_{56}b^{4}.
\end{equation*}
It follows from $\psi_{B}\lesssim B\varphi'_{i,B}+\psi_{0,B}$ and (v) of Lemma~\ref{le:psiphi2} that
\begin{equation*}
\frac{\psi_{B}}{B^{27}}+\frac{\varphi_{i-1,B}}{B^{27}}\lesssim \frac{\varphi'_{i,B}}{B^{17}}+\frac{\psi_{B}}{B^{27}}\lesssim \frac{1}{6}\varphi'_{i,B}+\frac{\nu_{4}}{2B^{27}}\psi_{0,B}.
\end{equation*}
Combining the above two estimates, we complete the proof of (ii).
\end{proof}

\subsection{Decay property on the $y_1$-variable}
In this subsection, we will introduce an elementary estimate of the decay for the remainder term $\varepsilon$ on the right-hand side of $y_{1}$-variable. We mention here that, this estimate will be used to close the bootstrap estimate~\eqref{est:Boot3} at the end of proof of Theorem~\ref{MT} in \S\ref{S:Endproof}.

We define the smooth function $\Phi\in C^{\infty}(\R)$ as follows,
\begin{equation*}
\Phi(y_1)=
\begin{cases}
y_1^{100},&\text{ for }y_1>1,\\
0,&\text{ for }y_1<0,
\end{cases}\quad\mbox{with}\quad \Phi'\geq 0,\quad \mbox{on}\ \R.
\end{equation*}
\begin{proposition}\label{Prop:decay}
Under the assumption of Proposition \ref{Prop:Mon}, we have
\begin{equation*}
\frac{1}{\lambda^{100}}\frac{\dd}{\dd s}
\left(\lambda^{100}\int_{\R^{2}}\varepsilon^2\Phi\dd y\right)\lesssim b^2+\int_{\R^{2}}\varepsilon^2\psi_B\dd y.
\end{equation*}
\end{proposition}
\begin{proof}
Using~\eqref{le:equvar}, we obtain
\begin{equation*}
\begin{aligned}
\frac{1}{2}\frac{\dd }{\dd s}\int_{\R^{2}}\varepsilon^{2}\Phi\dd y&=
\int_{\R^{2}}\varepsilon\Phi\left(\frac{\lambda_{s}}{\lambda}\Lambda \varepsilon+\partial_{y_{1}}\left(-\Delta \varepsilon+\varepsilon\right)\right)\dd y\\
&+\int_{\R^{2}}\varepsilon\Phi\left({\rm{Mod}}-\left(\frac{\lambda_{s}}{\lambda}+b\right)\Lambda\varepsilon+\Psi_{b}\right)\dd y\\
&-\int_{\R^{2}}\varepsilon\Phi\partial_{y_{1}}\left(\left(Q_{b}+\varepsilon\right)^{3}-Q^{3}_{b}\right)\dd y.
\end{aligned}
\end{equation*}
From integration by parts and the fact that $y_1\Phi'=100\Phi$ for $y_1\geq1$ and $\Phi'''\ll\Phi'$ for $y_1$ large, we obtain
\begin{equation*}
\begin{aligned}
&\int_{\R^{2}}\varepsilon\Phi\left(\frac{\lambda_{s}}{\lambda}\Lambda \varepsilon+\partial_{y_{1}}\left(-\Delta \varepsilon+\varepsilon\right)\right)\dd y\\
&=-\frac{1}{2}\frac{\lambda_s}{\lambda}\int_{\R^{2}} \varepsilon^2y_1\Phi'\dd y
-\frac{1}{2}\int_{\R^{2}}\left(3(\partial_{y_{1}}\varepsilon)^{2}
+(\partial_{y_{2}}\varepsilon)^{2}+\varepsilon^{2}\right)\Phi'\dd y+\frac{1}{2}\int_{\R^{2}}\varepsilon^{2}\Phi'''\dd y\\
&=-50\frac{\lambda_s}{\lambda}\int_{\R^{2}} \varepsilon^2\Phi\dd y
-\frac{1}{2}\int_{\R^{2}}\left(3(\partial_{y_{1}}\varepsilon)^{2}
+(\partial_{y_{2}}\varepsilon)^{2}+\varepsilon^{2}\right)\Phi'\dd y+O\left(\int_{\R^{2}}\varepsilon^{2}\psi_{B}\dd y\right).
\end{aligned}
\end{equation*}
Recall that, from the definition of ${\rm{Mod}}$ in \S\ref{SS:Geo}, 
\begin{equation*}
\begin{aligned}
\left({\rm{Mod}}-\left(\frac{\lambda_{s}}{\lambda}+b\right)\Lambda\varepsilon+\Psi_{b}\right)
&=\left(\frac{\lambda_{s}}{\lambda}+b\right)\Lambda Q_{b}
+\left(\frac{x_{1s}}{\lambda}-1\right)\left(\partial_{y_{1}}Q_{b}+\partial_{y_{1}}\varepsilon\right)\\
&+\frac{x_{2s}}{\lambda}\left(\partial_{y_{2}}Q_{b}+\partial_{y_{2}}\varepsilon\right)+\Psi_{b}-b_{s}\frac{\partial Q_{b}}{\partial b}.
\end{aligned}
\end{equation*}
Therefore, from (i) and (ii) of Lemma~\ref{le:Ketest} and (ii) of Lemma~\ref{le:modu1}, we obtain 
\begin{equation*}
\begin{aligned}
&\left|\int_{\R^{2}}\varepsilon\Phi\left({\rm{Mod}}-\left(\frac{\lambda_{s}}{\lambda}+b\right)\Lambda\varepsilon+\Psi_{b}\right)\dd y\right|\\
&\lesssim 
b^{2}+\int_{\R^{2}}\varepsilon^{2}\psi_{B}\dd y+\frac{1}{B^{10}}\int_{\R^{2}}\varepsilon^{2}\Phi'\dd y.
\end{aligned}
\end{equation*}
Using integration by parts, the non-linear term can be rewritten by 
\begin{equation*}
\begin{aligned}
&\int_{\R^{2}}\varepsilon\Phi\partial_{y_{1}}\left(\left(Q_{b}+\varepsilon\right)^{3}-Q^{3}_{b}\right)\dd y\\
&=\frac{3}{2}\int_{\R^{2}}\varepsilon^{2}Q_{b}\left(2\Phi\partial_{y_{1}}Q_{b}-\Phi'Q_{b}\right)\dd y\\
&+\int_{\R^{2}}\varepsilon^{3}\left(\Phi\partial_{y_{1}}Q_{b}-\Phi'Q_{b}\right)\dd y-\frac{3}{4}\int_{\R^{2}}\varepsilon^{4}\Phi'\dd y.
\end{aligned}
\end{equation*}
Note that
\begin{equation*}
\left\|\frac{\nabla \left(\Phi'+\psi_{B}\right)}{\Phi'+\psi_{B}}\right\|_{L^{\infty}(\R^{2})}+
\sum_{|\alpha|=2}\left\|\frac{\partial_{y}^{\alpha}\left(\Phi'+\psi_{B}\right)}{\Phi'+\psi_{B}}\right\|_{L^{\infty}(\R^{2})}\lesssim 1.
\end{equation*}
Therefore, using again~\eqref{est:weight2} and (i) of Lemma~\ref{le:Ketest}, we obtain
\begin{equation*}
\left|\int_{\R^{2}}\varepsilon\Phi\partial_{y_{1}}
\left(\left(Q_{b}+\varepsilon\right)^{3}-Q^{3}_{b}\right)\dd y\right|
\lesssim \int_{\R^{2}}\varepsilon^{2}\psi_{B}\dd y+\frac{1}{B^{10}}\int_{\R^{2}}\left(|\nabla \varepsilon|^{2}+\varepsilon^{2}\right)\Phi'\dd y.
\end{equation*}
Combining the above estimates, for $B$ large enough, we obtain
\begin{equation*}
\frac{\dd}{\dd s}\int_{\R^{2}}\varepsilon^2\Phi\dd y+100\frac{\lambda_s}{\lambda}\int_{\R^{2}}\varepsilon^2\Phi\dd y\lesssim b^2+ \int_{\R^{2}} \varepsilon^2\psi_B\dd y,
\end{equation*}
which completes the proof of Proposition~\ref{Prop:decay} immediately.
\end{proof}

\section{End of the proof of Theorem~\ref{MT}}\label{S:Endproof}
Let $0<\alpha\ll \alpha^{*}\ll \kappa \ll 1$ to be chosen later. Recall that, in Definition~\ref{def:Tube}, we define a $L^{2}$-moduled tube $\mathcal{T}_{\alpha^{*}}$ and a set of initial data $\mathcal{A}_{\alpha}$.
In this section, we will classify the asymptotic behavior of any solution with initial data in $\mathcal{A}_{\alpha}$ which directly implies Theorem \ref{MT}. 
We start with the following definition,
\begin{equation}\label{equ:deft1}
t^{*}=\sup\{0<t<+\infty:u(t_{1})\in \mathcal{T}_{\alpha^{*}},\ \forall t_{1}\in [0,t]\}.
\end{equation} 
Since $0<\alpha\ll\alpha^*\ll \kappa \ll1$, then for any initial data $u_{0}\in \mathcal{A}_{\alpha}$, we have $t^*>0$.

Next, by Proposition \ref{Prop:decomposition}, we know that $u(t)$ admits the following geometrical decomposition on $[0,t^*]$:
\begin{equation*}
u(t,x)=\frac{1}{{\lambda}(t)}\left[{Q}_{b(t)}+{\varepsilon}(t)\right]\left(\frac{x-{x}(t)}{{\lambda}(t)}\right).
\end{equation*}

Using the fact that $u_{0}\in \mathcal{A}_{\alpha}$, we obtain
\begin{equation}\label{est:smallnessinital}
\begin{aligned}
\|\varepsilon(0)\|_{H^1}+|b(0)|+&|1-\lambda(0)|+\mathcal{N}_2(0)\lesssim \delta(\alpha),\\
|E(u_0)|+\bigg|\int u_0^2-\int Q^2\bigg|\lesssim\delta(\alpha)\quad &\mbox{and}\quad 
\int_{\R}\int_{0}^{\infty}y_1^{100}\varepsilon^2(0,y)\dd y_{1}\dd y_{2}\leq2.
\end{aligned}
\end{equation}
We fix constants $B>100$ large enough and $0<\kappa_{1}< \min\left\{\kappa_{*},B^{-100}\right\}$ small enough such that Proposition~\ref{Prop:dsFij} and Proposition~\ref{Prop:Virial} hold. Define
\begin{equation}\label{equ:deft2}
t^{**}=\sup\{0<t<t^*:\eqref{est:Boot1}-\eqref{est:Boot3}\ \mbox{hold for all }t_{1}\in[0,t]\}.
\end{equation}
Note that from~\eqref{est:smallnessinital} and a straightforward continuity argument, we have $t^{**}>0$ is well-defined.  The key point in our analysis is to deduce $t^{*}=t^{**}$ by improving the bootstrap assumptions~\eqref{est:Boot1}--\eqref{est:Boot3}.
From now on, we denote $s^*=s(t^*)$, $s^{**}=s(t^{**})$. 
In the remainder of the proof, the implied constants in $\lesssim$ and $O$ do not depend on the small constant $\kappa$ appearing in bootstrap assumptions~\eqref{est:Boot1}--\eqref{est:Boot3} but can depend on the large constant $B$.

\subsection{Consequence of the monotonicity formula} We derive some crucial estimates from the monotonicity formula introduced in Proposition~\ref{Prop:Mon}. The proof of the following Lemma is similar to~\cite[Lemma 4.3]{MMR}, but it is given for the sake of completeness and the readers' convenience.
\begin{lemma}\label{le:Conmon}
The following estimates hold.
\begin{enumerate}

\item Control of $b$. For all $0\leq s_1<s_2\leq s^{**}$ and $m=2,3,4$, we have
	\begin{equation*}
    \int_{s_1}^{s_2}|b(s)|^m\,\dd s\lesssim\mathcal{N}_1(s_1)+|b(s_1)|^{m-1}+|b(s_2)|^{m-1}.
	\end{equation*}

	\item \emph{Control of $\mathcal{N}_{i}$.} For all $0\le s_1<s_2\le s^{**}$ and $i=1,2$, we have
	\begin{equation*}
	\begin{aligned}
	\mathcal{N}_{i}(s_2)+\int_{s_1}^{s_2}\mathcal{N}_{i-1}(s)\dd s&\lesssim\mathcal{N}_i(s_1)+|b(s_1)|^3+|b(s_2)|^3,\\
	\frac{\mathcal{N}_i(s_2)}{\lambda^{\theta}(s_2)}+\int_{s_1}^{s_2}\frac{1}{\lambda^{\theta}(s)}\left(\mathcal{N}_{i-1}(s)+|b(s)|^3\right)\dd s&\lesssim\frac{\mathcal{N}_i(s_1)}{\lambda^{\theta}(s_1)}+\frac{b^2(s_1)}{\lambda^{\theta}(s_1)}+\frac{b^2(s_2)}{\lambda^{\theta}(s_2)}.
	\end{aligned}
	\end{equation*}
	\item \emph{Control of $\frac{b}{\lambda^{\theta}}$}. For all $0\leq s_1<s_2\leq s^{**}$, we have 
	\begin{equation*}
	\left|\frac{b(s_2)}{\lambda^{\theta}(s_2)}-\frac{b(s_1)}{\lambda^{\theta}(s_1)}\right|\le K\left(\frac{\mathcal{N}_1(s_1)}{\lambda^{\theta}(s_1)}+\frac{b^2(s_1)}{\lambda^{\theta}(s_1)}+\frac{b^2(s_2)}{\lambda^{\theta}(s_2)}\right).
	\end{equation*}
	Here, $K>1$ is a universal constant.
	\item \emph{Refined control of $\lambda$}. Let $\lambda_0(s)=\lambda(s)(1-J_1(s))^2$. Then for all $s\in[0,s^{**}]$,
	\begin{equation*}
	\bigg|\frac{\lambda_{0s}}{\lambda_0}+b\bigg|\lesssim\mathcal{N}_0+b^2.
	\end{equation*}
\end{enumerate}
\end{lemma}
\begin{proof}
Proof of (i)--(ii).
From~\eqref{est:bsb2}, we have
\begin{equation}\label{est:endbs}
b^2= -\frac{b_{s}}{\theta}+O\left(\mathcal{N}_0+|b|\mathcal{N}_0^{\frac{1}{2}}+|b|^3\right)\Longrightarrow \frac{3}{4}b^{2}+\frac{b_{s}}{\theta}\lesssim \mathcal{N}_{0}.
\end{equation}
Note that, for $m=2,3,4$, we also have 
\begin{equation*}
b_{s}|b|^{m-2}=\frac{1}{m-1}\frac{\dd }{\dd s}\left(b|b|^{m-2}\right).
\end{equation*}
Based on the above identities,~\eqref{est:Boot1} and~\eqref{est:endbs}, for all $0\leq s_1<s_2\leq s^{**}$, we have
\begin{equation*}
\int_{s_1}^{s_2}|b(s)|^m\dd s\lesssim
\kappa^{m-2}\int_{s_1}^{s_2}\mathcal{N}_0(s)\dd s+ |b(s_1)|^{m-1}+|b(s_2)|^{m-1}.
\end{equation*}
Next, from (i) and (ii) of Proposition~\ref{Prop:Mon}, for $i=1,2$, we have
\begin{align*}
\mathcal{N}_{i}(s_2)+\int_{s_1}^{s_2}\mathcal{N}_{i-1}(s)\,\dd s\lesssim\mathcal{N}_i(s_1)+\int_{s_1}^{s_2}|b(s)|^4\,\dd s.
\end{align*}
The above two estimates imply (i) and the first estimate in (ii) immediately.

Then, using again~\eqref{est:endbs}, we have
\begin{equation*}
\int_{s_{1}}^{s_{2}}\frac{|b(s)|^{3}}{\lambda^{\theta}(s)}\dd s+\int_{s_{1}}^{s_{2}}\frac{b_{s}(s)|b(s)|}{\lambda^{\theta}(s)}\dd s\lesssim \int_{s_{1}}^{s_{2}}\frac{\mathcal{N}_{0}(s)}{\lambda^{\theta}(s)}\dd s,
\end{equation*}
which implies
\begin{equation*}
\int_{s_{1}}^{s_{2}}\frac{|b(s)|^{3}}{\lambda^{\theta}(s)}\dd s\lesssim 
 \frac{b^2(s_1)}{\lambda^{\theta}(s_1)}+\frac{b^2(s_2)}{\lambda^{\theta}(s_2)}+
 \int_{s_1}^{s_2}\frac{\mathcal{N}_0(s)}{\lambda^{\theta}(s)}\dd s.
\end{equation*}
Using again (i) and (ii) of Proposition~\ref{Prop:Mon}, we deduce that 
\begin{equation*}
\frac{\mathcal{N}_{i}(s_2)}{\lambda^{\theta}(s_2)}+\int_{s_1}^{s_2}\frac{\mathcal{N}_{i-1}(s)}{\lambda^{\theta}(s)}\dd s
\lesssim\frac{\mathcal{N}_{i}(s_1)}{\lambda^{\theta}(s_1)}+\int_{s_1}^{s_2}\frac{b^{4}(s)}{\lambda^{\theta}(s)}.
\end{equation*}
The above two estimates imply the second estimate in (ii) immediately.

Proof of (iii). First, from (iii) of Lemma~\ref{le:modu1}, we have 
\begin{equation*}
|e^{J}-1|\lesssim |J|\lesssim \mathcal{N}_0^{\frac{1}{2}}\lesssim \delta(\kappa)\ll1.
\end{equation*}
Based on (iii) of Lemma~\ref{le:modu2} and (ii) of Lemma~\ref{le:Conmon}, for 
$0\le s_1<s_2\le s^{**}$, we have
\begin{equation*}
\begin{aligned}
&\left|
\frac{b(s_2)}{\lambda^{\theta}(s_2)}e^{J(s_2)}-\frac{b(s_1)}{\lambda^{\theta}(s_1)}e^{J(s_1)}
\right|\\
&\lesssim \int_{s_1}^{s_2}\left|\frac{\dd}{\dd s}\left(\frac{b}{\lambda^{\theta}}e^J\right)(s)\dd s\right|
\lesssim \frac{\mathcal{N}_1(s_1)}{\lambda^{\theta}(s_1)}+\frac{b^2(s_1)}{\lambda^{\theta}(s_1)}+\frac{b^2(s_2)}{\lambda^{\theta}(s_2)}.
\end{aligned}
\end{equation*}
On the other hand, we have 
\begin{equation*}
\left|\frac{b}{\lambda^{\theta}}e^J-\frac{b}{\lambda^{\theta}}\right|\lesssim \frac{|b|\mathcal{N}_0^{\frac{1}{2}}}{\lambda^{\theta}}\lesssim \frac{b^2+\mathcal{N}_1}{\lambda^{\theta}}.
\end{equation*}
Combining the above two estimates, we complete the proof of (iii).

Proof of (iv). By an elementary computation, 
\begin{equation*}
\left|\frac{\lambda_{0s}}{\lambda_{0}}+b-\left(\frac{\lambda_{s}}{\lambda}+b-2J_{1s}\right)\right|
\lesssim\frac{|J_{1}||J_{1s}|}{1-J_{1}}.
\end{equation*}
Note that, from (ii) and (iii) of Lemma~\ref{le:modu1} and (i) of Lemma~\ref{le:modu2}, we have 
\begin{equation*}
\left|J_{1}\right|+\left|J_{1s}\right|\lesssim b^{2}+\mathcal{N}_{0}^{\frac{1}{2}}.
\end{equation*}
Combining the above estimates with Lemma~\ref{le:modu2}, we complete the proof of (iv).
\end{proof}

\subsection{Rigidity dynamics in $\mathcal{A}_{\alpha}$} In this subsection, we will give a specific classification for the asymptotic behavior of solutions with initial data in $\mathcal{A}_{\alpha}$. We denote 
\begin{equation*}
\widetilde{\mathcal{N}}_{1}(t)=\mathcal{N}_{1}(t)+b^{2}(t),\quad \mbox{on}\ [0,t^{*}].
\end{equation*}

Denote $t_{1}^{*}$ by the following separation time,
\begin{equation*}
\begin{aligned}
t_1^*=\begin{cases}
0, &\mbox{if} \ |b(0)|\ge C^{*}\widetilde{N}_{1}(0),\\
\sup\left\{0<t<t^*:|b(t_{1})|\le C^{*}\widetilde{N}_{1}(t_{1}),\forall t_{1}\in[0,t]\right\},\quad &\mbox{otherwise}.
\end{cases}
\end{aligned}
\end{equation*}
Here $C^{*}=100K$ and $K$ is the constant introduced in Lemma~\ref{le:Conmon}.
Then the following rigidity dynamics of solution flows near soliton manifold hold.
\begin{proposition}[Rigidity Dynamics]\label{Prop:rigidity}
There exist universal constants $0<\alpha\ll\alpha^*\ll\kappa$ such that the following holds. Let $u_0\in\mathcal{A}_{\alpha}$ and $u(t)$ be the corresponding solution to \eqref{CP} on $[0,T)$. Then the following trichotomy holds:
\begin{description}
\item [Soliton]  If $t^*=t_{1}^{*}$, then $t^*=t_{1}^{*}=T=\infty$ with
\begin{equation}\label{est:Solitoncase}
\begin{aligned}
&\lambda(t)=\lambda_{\infty}\left(1+o(1)\right),\quad |b(t)|+\mathcal{N}_2(t)\rightarrow0,\ \mbox{as}\ t\to \infty,\\
 &x_1(t)=\frac{t}{\lambda_{\infty}^2}\left(1+o(1)\right),\quad \quad x_2(t)\rightarrow x_{2,\infty},\ \ \mbox{as}\ t\to \infty,
 \end{aligned}
\end{equation}
for some $(\lambda_{\infty},x_{2,\infty})\in \R^{2}$. In addition, we have $|\lambda_{\infty}-1|\lesssim \delta(\alpha_0)$.

\smallskip
\item [Exit] If $t^{*}>t_{1}^{*}$ with $b(t^{*}_{1})\le -C^{*}\widetilde{\mathcal{N}}_{1}(t^{*}_{1})$, then $t^*<T$. In particular,
\begin{equation}\label{est:Exitcase1}
\inf_{\substack{\lambda_{0}>0\\x_{0}\in \R^{2}}}
\left\|u(t^*)-\frac{1}{\lambda_0}Q\left(\frac{x-x_0}{\lambda_0}\right)\right\|_{L^2}=\alpha^{*}.
\end{equation}
Moreover, the following estimates hold
\begin{equation}\label{est:Exitcase2}
b(t^*)\le-C(\alpha^*)<0\quad \mbox{and}\quad\lambda(t^*)\geq\frac{C(\alpha^*)}{\delta(\alpha)}\gg1.
\end{equation}

\smallskip
\item [Blow-up] If $t^{*}>t_{1}^{*}$ with $b(t^{*}_{1})\ge C^{*}\widetilde{\mathcal{N}}_{1}(t_{1}^{*})$, then $t^*=T<\infty$. In particular
\begin{equation}\label{est:Blowupcase}
\begin{aligned}
\lim_{t\uparrow T}\|\nabla\varepsilon(t)\|_{L^2}=0,\quad \lim_{t\uparrow T}\frac{\lambda(t)}{(T-t)^{\frac{1}{3-\theta}}}=\ell_1(u_0),\\
\lim_{t\uparrow T}\frac{b(t)}{(T-t)^{\frac{\theta}{3-\theta}}}=\ell_2(u_0),\quad  x_1(t)\sim (T-t)^{-\frac{\theta-1}{3-\theta}}.
\end{aligned}
\end{equation}
Here, $\ell_1$ and $\ell_2$ are positive constants depending only on the initial data $u_{0}$.
\end{description}
\end{proposition}

Note that, Proposition~\ref{Prop:rigidity} classifies the behavior of solution flows near the soliton manifold which implies Theorem~\ref{MT}. In the rest of this section, we are devoted to the proof of Proposition~\ref{Prop:rigidity} and split the proof into the following three parts.

\subsubsection{The Soliton case.} \label{SSS:Soliton}

Assume that $t^*=t_{1}^{*}$, \emph{i.e.} for all $t\in[0,t^*]$,
\begin{equation}\label{est:Solitonb}
|b(t)|\leq C^*\big[\mathcal{N}_1(t)+b^2(t)\big]\Longrightarrow |b(t)|\leq 2C^*\mathcal{N}_1(t).
\end{equation}

Step 1. Closing the bootstrap. We claim that, for all $s\in [0,s^{**}]$, 
\begin{equation}\label{est:Bootsoliton}
\begin{aligned}
|b(s)|+\|\varepsilon(s)\|_{L^2}+\mathcal{N}_2(s)+\int_0^{s}\mathcal{N}_1(s_{1})\dd s_{1}\lesssim\delta(\alpha),\\
 |\lambda(s)-1|\lesssim \delta(\alpha)\quad \mbox{and}\quad 
\int_{\R}\int_{0}^{\infty}y_1^{100}\varepsilon^2(s,y)\dd y_{1} \dd y_{2}\le 5.
\end{aligned}
\end{equation}

Note that, from $0<\alpha\ll\alpha^*\ll\kappa$, the estimates in ~\eqref{est:Bootsoliton} strictly improve the boostrap estimates~\eqref{est:Boot1}--\eqref{est:Boot3} and so we obtain $t^{*}=t^{**}$.

\smallskip
Indeed, from~\eqref{est:endbs},~\eqref{est:Solitonb} and (ii) of Lemma~\ref{le:Conmon}, for all $s\in [0,s^{**}]$, we have 
\begin{equation*}
\begin{aligned}
\left|b(s)\right|
&\lesssim \left|b(0)\right|+\int_{0}^{s}\left(b^{2}(s_{1})+\mathcal{N}_{0}(s_{1})\right)\dd s_{1}\\
&\lesssim  \left|b(0)\right|+\int_{0}^{s}\left(4\left(C^{*}\right)^{2}\mathcal{N}_{1}^{2}(s_{1})+\mathcal{N}_{0}(s_{1})\right)\dd s_{1}\\
&\lesssim |b(0)|+|b(0)|^{3}+\mathcal{N}_{2}(0)+|b(s)|^{3}.
\end{aligned}
\end{equation*}
It follows from~\eqref{est:smallnessinital}, (i) of Lemma~\ref{le:modu1} and (i) of Lemma~\ref{le:Conmon} that 
\begin{equation}\label{est:Solitonbvn1}
\left|b(s)\right|+\|\varepsilon(s)\|_{L^2}+\mathcal{N}_2(s)+\int_{0}^{s}\mathcal{N}_{1}(s_{1})\dd s_{1}\lesssim \delta(\alpha),\quad \mbox{for all}\ s\in [0,s^{**}].
\end{equation}
Then we use~\eqref{est:Solitonb} and (iv) of Lemma~\ref{le:Conmon} to obtain
\begin{align*}
\left|\frac{\lambda_{0s}}{\lambda_0}\right|\lesssim b+\mathcal{N}_0+b^2\lesssim 4C^*\mathcal{N}_1+\mathcal{N}_0\lesssim\mathcal{N}_{1}.
\end{align*}
Integrating the above estimate over $[s_{1},s_{2}]$ for any $0\le s_{1}<s_{2}\le s^{**}$ and then using~\eqref{est:smallnessinital} and \eqref{est:Solitonbvn1}, we deduce that 
\begin{equation}\label{est:Solitonlambda}
\left|\frac{\lambda_{0}(s_{2})}{\lambda_{0}(s_{1})}-1\right|\lesssim \delta(\alpha)
\Longrightarrow | \lambda(s)-1|\lesssim \delta(\alpha)\ \mbox{on}\ [0,s^{**}].
\end{equation}
Last, integrating the estimate in Proposition~\ref{Prop:decay} over $[0,s]$ for any $s\in [0,s^{**}]$ and then using~\eqref{est:smallnessinital},~\eqref{est:Solitonlambda} and (i) and (ii) of Lemma~\ref{le:Conmon}, we obtain
\begin{equation}\label{est:Solitone2}
\begin{aligned}
\int_{\R^{2}}\varepsilon^{2}(s)\Phi\dd y
&\le C \int_{0}^{s}\frac{\lambda^{100}(s_{1})}{\lambda^{100}(s)}\left(b^{2}(s_{1})+\mathcal{N}_{1}(s_{1})\right)\dd s_{1}\\
&+\frac{\lambda^{100}(0)}{\lambda^{100}(s)}\int_{\R^{2}}\varepsilon^{2}(0)\Phi\dd y\le 2+\delta(\alpha)\le 5.
\end{aligned}
\end{equation}
Here, $C>1$ is a universal constant independent with $\alpha$. Combining~\eqref{est:Solitonbvn1} and~\eqref{est:Solitonlambda} with~\eqref{est:Solitone2}, we complete the proof of~\eqref{est:Bootsoliton}, and thus, from a standard continuity argument, we obtain $t^{*}=t^{**}=T=\infty$.

\smallskip Step 2. Proof of~\eqref{est:Solitoncase}.
From~\eqref{est:Bootsoliton} and (i) of Lemma~\ref{le:Conmon}, we know that $\mathcal{N}_1\in L^1((0,+\infty))$, and thus, there exists an increasing sequence $\{s_n\}_{n=1}^{\infty}\subset(0,+\infty)$ such that $s_n\rightarrow+\infty$ and $\mathcal{N}_1(s_n)\rightarrow0$ as $n\to \infty$. Then, from~\eqref{est:Solitonb} and (i) of Lemma~\ref{le:Conmon}, for all $n\in\mathbb{N}^+$ and $s\geq s_n$, we have 
\begin{equation*}
\mathcal{N}_1(s)\lesssim \mathcal{N}_1(s_n)+\mathcal{N}_1^3(s)+\mathcal{N}_1^3(s_n)
\Longrightarrow \lim_{s\to\infty}\mathcal{N}_{1}(s)=0.
\end{equation*}
From~\eqref{est:y18},~\eqref{est:Solitonb} and~\eqref{est:Bootsoliton}, we deduce that 
\begin{equation*}
\mathcal{N}_{2}\lesssim \mathcal{N}^{\frac{92}{93}}_{1}\ \mbox{and}\ |b|\lesssim \mathcal{N}_{1}\Longrightarrow \lim_{s\to \infty} \left(\mathcal{N}_{2}(s)+|b(s)|\right)=0.
\end{equation*}
Next, from~\eqref{est:Solitonb},~\eqref{est:Bootsoliton} and (i) of Lemma~\ref{le:Conmon}, we see that 
\begin{equation*}
\int_{0}^{\infty}\left(|b(s)|+\mathcal{N}_{0}(s)\right)\lesssim \int_{0}^{\infty}\left(\mathcal{N}_{0}(s)+\mathcal{N}^{2}_{1}(s)\right)\dd s_{1}\lesssim \delta(\alpha).
\end{equation*}
It follows from~\eqref{est:Bootsoliton} and (iv) of Lemma~\ref{le:Conmon} that 
\begin{equation*}
\int_{0}^{\infty}|\lambda_{0s}(s)|\dd s\lesssim 
\int_{0}^{\infty}\left(|b(s)|+\mathcal{N}_{0}(s)\right)\dd s\lesssim \delta(\alpha).
\end{equation*}
Based on above estimate, ~\eqref{est:smallnessinital} and~\eqref{est:Bootsoliton}, we know that there exists $\lambda_{\infty}\in \R$ with $|\lambda_{\infty}-1|\lesssim \delta(\alpha)$ such that
\begin{equation*}
\lim_{s\to \infty}\lambda_{0}(s)=\lambda_{\infty}\Longrightarrow
\lim_{s\to \infty}\lambda(s)=\lambda_{\infty}.
\end{equation*}
Then, from the above estimates, (ii) of Lemma~\ref{le:modu1} and $t\sim s$, 
\begin{equation*}
x_{1t}=\frac{x_{1s}}{\lambda^{3}}=\frac{1+o(1)}{\lambda^{2}_{\infty}}\Longrightarrow
x_{1}(t)=\frac{t}{\lambda^{2}_{\infty}}\left(1+o(1)\right).
\end{equation*}
Last, from~\eqref{est:Bootsoliton}, (ii) of Lemma~\ref{le:modu1} and~(iv) of Lemma~\ref{le:modu2}, we have 
\begin{equation*}
\left|x_{2s}-\left(\lambda J_{3}\right)_{s}\right|\lesssim |b|\mathcal{N}_{0}^{\frac{1}{2}}+b^{2}+\mathcal{N}_{0}\lesssim b^{2}+\mathcal{N}_{0}.
\end{equation*}
It follows from (i) and (ii) of Lemma~\ref{le:Conmon} that 
\begin{equation*}
\int_{0}^{\infty}\left|x_{2s}-\left(\lambda J_{3}\right)_{s}\right|\dd s
\lesssim \int_{0}^{\infty}\left(b^{2}(s)+\mathcal{N}_{0}(s)\right)\dd s \lesssim \delta(\alpha).
\end{equation*}
Based on the above estimate and $|\lambda(s) J_{3}(s)|\lesssim |J_{3}(s)|\lesssim \mathcal{N}_{0}^{\frac{1}{2}}\to 0$ as $s\to \infty$, we know that there exists $x_{2,\infty}\in \R$ such that 
\begin{equation*}
\lim_{s\to \infty}\left(x_{2}(s)-\lambda(s)J_{3}(s)\right)=x_{2,\infty}\Longrightarrow
\lim_{s\to \infty}x_{2}(s)=x_{2,\infty}.
\end{equation*}
Combining the above estimates with $t\sim s$, we complete the proof of~\eqref{est:Solitoncase}.

\subsubsection{The Exit case.} Assume that $t^{*}>t_1^{*}$ and $ b(t^*_1)\le -C^*\widetilde{\mathcal{N}}_{1}(t_{1}^{*})$.

\smallskip
Step 1. Closing the bootstrap.
First of all, using the same argument as in the Soliton case, the following estimates hold on $[0,s_{1}^{*}]$:
\begin{equation}\label{est:BootExit1}
\begin{aligned}
|b(s)|+\|\varepsilon(s)\|_{L^2}+\mathcal{N}_2(s)+\int_0^{s}\mathcal{N}_1(s_{1})\dd s_{1}\lesssim\delta(\alpha),\\
|\lambda(s)-1|\lesssim \delta(\alpha)\quad \mbox{and}\quad 
\int_{\R}\int_{0}^{\infty}y_1^{100}\varepsilon^2(s,y)\dd y_{1} \dd y_{2}\le 5.
\end{aligned}
\end{equation}
In particular, we have $t_1^*<t^{**}\leq t^*$. Now, we claim that $t^{*}=t^{**}<T$. To prove this, we use a slightly different bootstrap argument than the one used in the Soliton case to improve the bootstrap assumptions~\eqref{est:Boot1}--\eqref{est:Boot3} on $[t_{1}^{*},t^{**}]$.

We denote 
\begin{equation*}
\ell^{*}=\frac{b(s^*_1)}{\lambda^{\theta}(s^*_1)}<0.
\end{equation*}
From~\eqref{est:BootExit1}, we deduce that~$|\ell^{*}|\lesssim \delta(\alpha)$. From~\eqref{est:BootExit1}, (iii) of Lemma~\ref{le:Conmon}, $|b(s^{*}_{1})|\ge C^*|{N}_{1}(s_{1}^{*})|$,
 and the definition of $C^*$, for all $s\in [s_{1}^{*},s^{**}]$, we obtain
 \begin{equation*}
\begin{aligned}
\bigg|\frac{b(s)}{\lambda^{\theta}(s)}-\ell^*\bigg|
&\le 
K\left(\frac{\mathcal{N}_1(s^*_1)}{\lambda^{\theta}(s^*_1)}+\frac{b^2(s^*_1)}{\lambda^{\theta}(s^*_1)}+\frac{b^2(s)}{\lambda^{\theta}(s)}\right)\\
&\le  \frac{|\ell^{*}|}{100}+\delta(\alpha)|\ell^{*}|+\kappa\frac{|b(s)|}{\lambda^{\theta}(s)},
\end{aligned}
\end{equation*}
which implies immediately
\begin{equation}\label{est:Exitblambda}
2\ell^*\le\frac{b(s)}{\lambda^{\theta}(s)}\le \frac{\ell^*}{2}<0\quad \mbox{and}\quad b(s)<0.
\end{equation}
It follows from~\eqref{est:BootExit1} and (ii) of Lemma~\ref{le:Conmon} that 
\begin{equation*}
\frac{b(s)}{\lambda^{\theta}(s)}+\frac{\mathcal{N}_{2}(s)}{\lambda^{\theta}(s)}\lesssim \delta(\alpha).
\end{equation*}
On the other hand, from (iii) of Lemma~\ref{le:Conmon} and $b(s)<0$ on $[s^{*}_{1},s^{**}]$, we obtain 
\begin{equation*}
\frac{\lambda_{0s}}{\lambda_0}\gtrsim-\mathcal{N}_{0}\gtrsim -\kappa\Longrightarrow
\frac{\lambda_{0}(s_{2})}{\lambda_{0}(s_{1})}-1\gtrsim-\kappa,\quad \mbox{for all}\ s_{1}^{*}\le s_{1}<s_{2}\le s^{**}.
\end{equation*}
Therefore, from $|J_{1}|\lesssim \kappa\ll 1$ on $[s^{*}_{1},s^{**}]$, we directly have
\begin{equation}\label{est:Exitlambda12}
\frac{\lambda(s_{2})}{\lambda(s_{1})}-1\gtrsim -\kappa,\quad \mbox{for all}\ s_{1}^{*}\le s_{1}<s_{2}\le s^{**}\Longrightarrow \lambda(s)\ge \frac{1}{2}\ \mbox{on}\ [s^{*}_{1},s^{**}].
\end{equation}
Then, using a similar argument to the one in the soliton case, for all $s\in [s_{1}^{*},s^{**}]$, we have
\begin{equation*}
\int_{\R^{2}}\varepsilon^{2}(s)\Phi\dd y\le 
\int_{\R}\int_{0}^{\infty}y_{1}^{100}\varepsilon^{2}(s^{*}_{1},y)\dd y_{1}\dd y_{2}+\delta(\kappa)\le 7.
\end{equation*}
Last, for all $t\in[t^{*}_1,t^{*})$, we have $u(t)\in\mathcal{T}_{\alpha^*}$. Therefore, from Proposition~\ref{Prop:decomposition},
\begin{equation*}
|b(t)|\lesssim\delta(\alpha^*)\ll \kappa,\quad \mbox{on}\  [t_{1}^{*},t^{*}].
\end{equation*} 
 It follows from~\eqref{est:smallnessinital}, (i) of Lemma~\ref{le:modu1} and~(ii) of Lemma~\ref{le:Conmon} that 
\begin{equation*}
\|\varepsilon(s)\|_{L^2}+|b(s)|+\mathcal{N}_2(s)\lesssim\delta(\alpha)+\delta(\alpha^*)\ll \kappa,\quad \mbox{on}\ [s_{1}^{*},s^{**}].
\end{equation*}
Combining the above estimates, we improve the bootstrap assumptions~\eqref{est:Boot1}--\eqref{est:Boot3} on $[s_{1}^{*},s^{**}]$ and thus we conclude that $t^{*}=t^{**}$.

\smallskip
Step 2. Proof of~\eqref{est:Exitcase1} and~\eqref{est:Exitcase2}. First, from~\eqref{est:Exitblambda} and (iv) of Lemma~\ref{le:Conmon}, we have 
\begin{equation*}
\frac{|\ell^{*}|}{3}-C\frac{\mathcal{N}_{0}}{\lambda^{\theta}}\le \lambda_{0}^{2-\theta}\lambda_{0t}\le 3|\ell^{*}|+C\frac{\mathcal{N}_{0}}{\lambda^{\theta}}.
\end{equation*}
Here, $C$ is a universal constant independent with $\alpha$ and $\alpha^{*}$. Integrating the above estimate over $[t_{1}^{*},t]$ for any $t\in [t_{1}^{*},t^{*}]$, we deduce that 
\begin{equation*}
\begin{aligned}
\frac{|\ell^{*}|}{3}(t-t_{1}^{*})-C\int_{t_{1}^{*}}^{t}\frac{\mathcal{N}_{0}(t_{1})}{\lambda^{\theta}(t_{1})}\dd t_{1}&\le \frac{\lambda_{0}^{3-\theta}(t)-\lambda_{0}^{3-\theta}(t^{*}_{1})}{3-\theta},\\
3|\ell^{*}|(t-t_{1}^{*})+C\int_{t_{1}^{*}}^{t}\frac{\mathcal{N}_{0}(t_{1})}{\lambda^{\theta}(t_{1})}\dd t_{1}
&\ge \frac{\lambda_{0}^{3-\theta}(t)-\lambda_{0}^{3-\theta}(t^{*}_{1})}{3-\theta}.
\end{aligned}
\end{equation*}
Note that, from~\eqref{est:Boot1},~\eqref{est:Exitlambda12} and (ii) of Lemma~\ref{le:Conmon},
\begin{equation*}
\begin{aligned}
\int_{t_{1}^{*}}^{t}\frac{\mathcal{N}_{0}(t_{1})}{\lambda^{\theta}(t_{1})}\dd t_{1}
&\lesssim \int_{s^{*}_{1}}^{s}\lambda^{3-\theta}(s_{1})\mathcal{N}_{0}(s_{1})\dd s_{1}\\
&\lesssim \lambda^{3-\theta}(s)\int_{s_{1}^{*}}^{s}\mathcal{N}_{0}(s_{1})\dd s_{1}\lesssim \delta(\kappa)\lambda^{3-\theta}(t).
\end{aligned}
\end{equation*}
Combining the above two estimates with $\lambda_{0}\approx \lambda$, we obtain 
\begin{equation*}
\frac{1}{5}\left(|\ell^{*}|(t-t_{1}^{*})+\lambda_{0}^{3-\theta}(t^{*}_{1})\right)
\le \lambda^{3-\theta}(t)
\le 5\left(|\ell^{*}|(t-t_{1}^{*})+\lambda_{0}^{3-\theta}(t^{*}_{1})\right).
\end{equation*}
Based on the above estimate and~\eqref{est:Exitblambda}, 
\begin{equation}\label{est:Exitblu}
\begin{aligned}
-10|\ell^{*}|\left(|\ell^{*}|(t-t_{1}^{*})+\lambda_{0}^{3-\theta}(t^{*}_{1})\right)^{\frac{\theta}{3-\theta}}
&\le b(t),\\
-\frac{1}{10}|\ell^{*}|\left(|\ell^{*}|(t-t_{1}^{*})+\lambda_{0}^{3-\theta}(t^{*}_{1})\right)&\ge  b(t).
\end{aligned}
\end{equation}
The above lower and upper estimate of $b(t)$ is enough to show that $t^{*}<T$. 

Indeed, for the sake of contradiction assume that $t^{*}=T$. Then from~\eqref{est:Boot1},~\eqref{est:smallnessinital} and~\eqref{est:Exitlambda12}, we obtain 
\begin{equation*}
\left\|\nabla \varepsilon(t)\right\|_{L^{2}}\lesssim \delta(\kappa)\Longrightarrow
\left\|\nabla u(t)\right\|_{L^{2}}=\frac{\|Q\|_{L^{2}}+\delta(\kappa)}{\lambda(t)}\lesssim 1.
\end{equation*}
It follows from the blow-up criterion~\eqref{equ:Blowcri} that $T=\infty$. Therefore, from~\eqref{est:Exitblu}, we obtain $b(t)\to -\infty$ as $t\to \infty$ which contradictory with $|b(t)|\lesssim \delta(\kappa)$. This means that $t^{*}<T$ and thus~\eqref{est:Exitcase1} holds at time $t^{*}$ from the definition of $\mathcal{T}_{\alpha^{*}}$.

Last, from~\eqref{est:smallnessinital} and (i) of Lemma~\ref{le:modu1}, we have
 $\left(\alpha^{*}\right)^{2}\lesssim |b(t^{*})|$, and thus from~\eqref{est:Exitblambda} and $|\ell^{*}|\lesssim \delta(\alpha)$, we complete the proof of~\eqref{est:Exitcase2}.

\subsubsection{The Blow-up case.} Assume that $t^{*}>t_1^{*}$ and 
$b(t_{1}^{*})\ge C^*\widetilde{\mathcal{N}}_{1}(t_{1}^{*})$.

\smallskip
Step 1. Closing the bootstrap. First of all, using the same argument as in the Soliton case, the following estimates hold on $[0,s_{1}^{*}]$:
\begin{equation}\label{est:BootBlowup1}
\begin{aligned}
|b(s)|+\|\varepsilon(s)\|_{L^2}+\mathcal{N}_2(s)+\int_0^{s}\mathcal{N}_1(s_{1})\dd s_{1}\lesssim\delta(\alpha),\\
|\lambda(s)-1|\lesssim \delta(\alpha)\quad \mbox{and}\quad 
\int_{\R}\int_{0}^{\infty}y_1^{100}\varepsilon^2(s,y)\dd y_{1} \dd y_{2}\le 5.
\end{aligned}
\end{equation}
In particular, we have $t_1^*<t^{**}\leq t^*$. Now, we claim that $t^{*}=t^{**}=T$. To prove this, we use a slightly different bootstrap argument than the one used in the Soliton and Exit cases to improve the bootstrap assumptions~\eqref{est:Boot1}--\eqref{est:Boot3} on $[t_{1}^{*},t^{**}]$.

We denote 
\begin{equation*}
\ell^{*}=\frac{b(s^*_1)}{\lambda^{\theta}(s^*_1)}>0.
\end{equation*}

From~\eqref{est:BootBlowup1}, we deduce that~$0<\ell^{*}\lesssim \delta(\alpha)$. Based on \eqref{est:BootBlowup1}, (iii) of Lemma~\ref{le:Conmon} and a similar argument to the one in the Exit case, we obtain
\begin{equation}\label{est:Blowblambda}
\frac{1}{2}\ell^*\le\frac{b(s)}{\lambda^{\theta}(s)}\le 2\ell^*\quad \mbox{and}\quad b(s)>0.
\end{equation}
It follows from~\eqref{est:BootBlowup1} and (ii) of Lemma~\ref{le:Conmon} that 
\begin{equation}\label{est:Blowblambda2}
\frac{b(s)}{\lambda^{\theta}(s)}+\frac{\mathcal{N}_{2}(s)}{\lambda^{\theta}(s)}\lesssim \delta(\alpha).
\end{equation}
On the other hand, from (iii) of Lemma~\ref{le:Conmon} and $b(s)>0$ on $[s^{*}_{1},s^{**}]$, we obtain 
\begin{equation*}
\frac{\lambda_{0s}}{\lambda_0}\lesssim \mathcal{N}_{0}\lesssim \kappa\Longrightarrow
\frac{\lambda_{0}(s_{2})}{\lambda_{0}(s_{1})}-1\lesssim \kappa,\quad \mbox{for all}\ s_{1}^{*}\le s_{1}<s_{2}\le s^{**}.
\end{equation*}
Therefore, from $|J_{1}|\lesssim \kappa\ll 1$ on $[s^{*}_{1},s^{**}]$, we directly have
\begin{equation}\label{est:Blowlambda12}
\frac{\lambda(s_{2})}{\lambda(s_{1})}-1\lesssim \kappa,\quad \mbox{for all}\ s_{1}^{*}\le s_{1}<s_{2}\le s^{**}\Longrightarrow \lambda(s)\le 2\ \mbox{on}\ [s^{*}_{1},s^{**}].
\end{equation}
Therefore, from~\eqref{est:Blowblambda2} and (i) of Lemma~\ref{le:modu1}, we deduce that 
\begin{equation}\label{est:Blowebn}
\|\varepsilon(s)\|_{L^{2}}+|b(s)|+\mathcal{N}_{2}(s)\lesssim \delta(\alpha).
\end{equation}

Last, integrating the estimate in Proposition~\ref{Prop:decay} over $[s^{*}_{1},s]$ for any $s\in [s^{*}_{1},s^{**}]$ and then using~\eqref{est:BootBlowup1} and (i) and (ii) of Lemma~\ref{le:Conmon}, we obtain
\begin{equation*}
\begin{aligned}
\int_{\R^{2}}\varepsilon^{2}(s)\Phi\dd y
&\le C \int_{s_{1}^{*}}^{s}\frac{\lambda^{100}(s_{1})}{\lambda^{100}(s)}\left(b^{2}(s_{1})+\mathcal{N}_{1}(s_{1})\right)\dd s_{1}\\
&+\frac{\lambda^{100}(s_{1}^{*})}{\lambda^{100}(s)}\int_{\R^{2}}\varepsilon^{2}(s_{1}^{*})\Phi\dd y\le 2+\delta(\alpha)\le \frac{5+\delta(\alpha)}{\lambda^{100}(s)}.
\end{aligned}
\end{equation*}
Combining the above estimate, we improve the bootstrap estimates~\eqref{est:Boot1}--\eqref{est:Boot3} and thus we conclude that $t^{*}=t^{**}$. In particular, we also conclude that $t^{*}=t^{**}=T$ since~\eqref{est:Blowebn} improve the estimate in the definition of $\mathcal{T}_{\alpha^{*}}$ provided that $0<\alpha\ll \alpha^{*}$.

\smallskip 
Step 2. Proof of~\eqref{est:Blowupcase}. Similar to the case of Exit, for all $t\in[t^*_1,T)$, we have
\begin{equation*}
-3\ell^{*}-C\frac{\mathcal{N}_{0}}{\lambda^{\theta}}\le\lambda_0^{2-\theta}\lambda_{0t}\le -\frac{1}{3}\ell^*+C\frac{\mathcal{N}_{0}}{\lambda^{\theta}}.
\end{equation*}
Here, $C$ is a universal constant independent with $\alpha$ and $\alpha^{*}$. Integrating the above estimate over $[t_{1}^{*},t]$ for any $t\in [t_{1}^{*},t^{*}]$, we deduce that 
\begin{equation*}
\begin{aligned}
-3\ell^{*}(t-t^{*}_1)-C\int_{t^{*}_1}^{t}\frac{\mathcal{N}_{0}(t_{1})}{\lambda^{\theta}(t_{1})}\dd t_{1}
&\le\frac{\lambda^{3-\theta}_0(t)-\lambda^{3-\theta}_0(t_1^*)}{3-\theta},\\
-\frac{1}{3}\ell^{*}(t-t^{*}_1)+C\int_{t^{*}_1}^{t}\frac{\mathcal{N}_{0}(t_{1})}{\lambda^{\theta}(t_{1})}\dd t_{1}
&\ge \frac{\lambda^{3-\theta}_0(t)-\lambda^{3-\theta}_0(t_1^*)}{3-\theta}.
\end{aligned}
\end{equation*}
From~\eqref{est:Blowlambda12} and (ii) of Lemma~\ref{le:Conmon}, we know that
\begin{equation*}
\int_{t_{1}^{*}}^{t}\frac{N_{0}(t_{1})}{\lambda^{\theta}(t_{1})}\dd t_{1}
\lesssim \int_{s_{1}^{*}}^{s}\lambda^{3-\theta}(s_{1})\mathcal{N}_{0}(s_{1})\dd s_{1}\lesssim \delta(\kappa).
\end{equation*}
Combining the above estimates, we obtain
\begin{equation*}
0\le \lambda_{0}^{3-\theta}(t)\le \lambda_{0}^{3-\theta}(t_{1}^{*})-\left(1-\frac{\theta}{3}\right)\ell^{*}(t-t_{1}^{*})+\delta(\kappa),
\end{equation*}
which directly implies $T<\infty$.

Then, from~\eqref{est:smallnessinital},~\eqref{est:Blowlambda12},~\eqref{est:Blowebn} and (i) of Lemma~\ref{le:modu1}, we find
\begin{equation*}
\|\nabla \varepsilon\|_{L^{2}}^{2}\lesssim \lambda^{2}E(u_{0})+\delta(\alpha)\lesssim \delta(\alpha).
\end{equation*}
It follows that 
\begin{equation*}
\|\nabla u(t)\|_{L^{2}}^{2}=\frac{\|\nabla Q\|^{2}_{L^{2}}+\delta(\alpha)}{\lambda^{2}(t)},\quad \mbox{for all}\ t\in [0,T).
\end{equation*}
Based on the above identity, blow-up criterion~\eqref{equ:Blowcri} of the Cauchy-problem and~\eqref{est:Blowblambda2},
\begin{equation*}
\lim_{t\uparrow T}\lambda(t)=0\Longrightarrow \lim_{t\uparrow T}\left(\lambda(t)+|b(t)|+\mathcal{N}_{2}(t)\right)=0.
\end{equation*}
Using again (i) of Lemma~\ref{le:modu1}, we conclude that 
\begin{equation*}
\lim_{t\uparrow T}\left(\lambda(t)+|b(t)|+\mathcal{N}_{2}(t)+\|\nabla \varepsilon\|_{L^{2}}^{2}\right)=0.
\end{equation*}
Then, from~\eqref{est:Blowblambda2}, (iii) of Lemma~\ref{le:modu2} and (ii) of Lemma~\ref{le:Conmon},
\begin{align*}
\int_0^{\infty}\left|\frac{\dd}{\dd s}\left(\frac{b}{\lambda^{\theta}}e^J\right)\right|\dd s\lesssim  \int_0^{\infty}\frac{|b(s)|^3+\mathcal{N}_0(s)}{\lambda^{\theta}(s)}\dd s<\infty.
\end{align*}
On the other hand, based on the above estimates, we directly obtain
\begin{equation*}
\lim_{t\uparrow T}\left(\left|J(t)\right|+\left|\frac{\lambda_{0}(t)}{\lambda(t)}-1\right|\right)=0.
\end{equation*}
Hence, there exists $0<\ell_0\lesssim\delta(\alpha)$ such that 
\begin{equation*}
\lim_{t\uparrow T}\frac{b(t)}{\lambda^{\theta}(t)}=\lim_{t\rightarrow T}\frac{b(t)}{\lambda_0^{\theta}(t)}=\ell_0.
\end{equation*}
It follows from (ii) and (iv) of Lemma~\ref{le:Conmon} that
\begin{equation*}
\begin{aligned}
&\left|\frac{\lambda_{0}^{3-\theta}(t)}{3-\theta}-\int_{t}^{T}\frac{b(t_{1})}{\lambda_{0}^{\theta}(t_{1})}\dd t_{1}\right|\\
&=\left|\int_{s}^{s(T)}\lambda_{0}^{3-\theta}(s_{1})
\left(\frac{\lambda_{0s}(s_{1})}{\lambda_{0}(s_{1})}+b(s_{1})\right)\dd s_{1}\right|\\
&\lesssim \lambda^{3-\theta}(s(t))\int_{s(t)}^{s(T)}\left(b^{2}(s_{1})+\mathcal{N}_{0}(s_{1})\right)\dd s_{1}\to 0\quad \mbox{as}\ t\to T.
\end{aligned}
\end{equation*}
Therefore, we conclude that 
\begin{equation*}
\begin{aligned}
\lim_{t\uparrow T}\frac{\lambda(t)}{(T-t)^{\frac{1}{3-\theta}}}&=\left((3-\theta)\ell_0\right)^{\frac{1}{3-\theta}},\\
\lim_{t\uparrow T}\frac{b(t)}{(T-t)^{\frac{\theta}{3-\theta}}}&=\left((3-\theta)^{\theta}\ell_{0}^{3}\right)^{\frac{1}{3-\theta}}.
\end{aligned}
\end{equation*}
Last, from (ii) of Lemma~\ref{le:modu1}, we obtain $x_{1t}\sim \lambda^{-2}$ which implies
\begin{equation*}
x_1(t)\sim (T-t)^{-\frac{\theta-1}{3-\theta}},\quad \text{as }t\rightarrow T.
\end{equation*}
The proof of Proposition~\ref{Prop:rigidity} is complete.

\section{End of the proof of Theorem~\ref{MT2}}\label{S:Endproof2}

In this section, we will give a complete proof of Theorem~\ref{MT2}. First, we recall the following variational property of the ground state $Q$.
\begin{lemma}[Variation property of $Q$]\label{le:orbi}
There exists $\alpha_{1}>0$ such that the following hold. For any $0<\alpha_{2}<\alpha_{1}$ and  $u_0\in H^1$ with
\begin{equation*}
E(u_0)\le\alpha_{2}\int_{\mathbb{R}^2}|\nabla u_0|^2\dd y\quad\mbox{and}\quad 
 \left|\int_{\mathbb{R}^2}(u_0^2-Q^2)\dd y\right|\le \alpha_{2}.
\end{equation*}
Then we have 
\begin{equation*}
\inf_{\substack{\lambda_{0}>0\\x_{0}\in \R^{2}}}\left\|u_0(\cdot)-\frac{\sigma_{0}}{\lambda_0}Q\left(\frac{\cdot-x_0}{\lambda_0}\right)\right\|_{L^2}\le \delta(\alpha_{2})\quad \mbox{where}\ \sigma_{0}\in \{-1,1\}.
\end{equation*}
\end{lemma}

\begin{proof}
The proof is based on a standard variational argument. We refer to~\cite[Lemma 1]{Merle} for the details of the proof.
\end{proof}

For the case of $E_{0}<0$, from Lemma~\ref{le:orbi}, we know that (Exit) is not possible for initial data with negative energy. Then, by (i) of Lemma~\ref{le:modu1} and Proposition~\ref{Prop:rigidity}, we obtain
\begin{equation*}
\lambda^2(t)|E_0|+\|\nabla \varepsilon(t)\|_{L^{2}}^{2}\lesssim |b(t)|+\mathcal{N}_0(t)\rightarrow 0,\quad \mbox{as}\ t\to T,
\end{equation*}
which implies the corresponding solution $u(t)$ belongs to the Blow-up regime.

\smallskip
For the case of $E_0=0$, from Lemma~\ref{le:orbi}, we also know that (Exit) is not possible.
Then,  the proof proceeds by contradiction. We assume that the corresponding solution $u(t)$ belongs to the Soliton regime. A contradiction then follows from the subsequent discussion. We refer to~\cite[\S5]{MMR} for a similar proof.

From Proposition~\ref{Prop:rigidity}, we can choose $t_0$ large enough such that
\begin{equation*}
|\lambda(t)-1|+|x_{1t}(t)-1|\le \frac{1}{100},\quad \mbox{for all} \ t\ge t_{0}.
\end{equation*}
We define the smooth function $G_{0}\in C^{\infty}$ with $G'_{0}\le 0$ as follows,
\begin{equation*}
G_0(x_1)=
\begin{cases}
1,&\text{if }x_1<-2,\\
1-(x_1+2)^{100}e^{-\frac{1}{x_1+2}},&\text{if }-2< x_1\leq -\frac{19}{10},\\
(x_1+1)^{100}e^{\frac{1}{x_1+1}},&\text{if }-\frac{11}{10}\leq x_1< -1,\\
0,&\text{if }x_1>-1.
\end{cases}
\end{equation*}
We also define the smooth function $G_{1}\in C^{\infty}$ with $G'_{1}\le 0$ as follows,
\begin{equation*}
G_1(x_1)=
\begin{cases}
1,&\text{if }x_1<-3,\\
1-(x_1+3)^{100}e^{-\frac{1}{x_1+3}},&\text{if }-3< x_1\leq -\frac{29}{10},\\
(x_1+2)^{100}e^{\frac{1}{x_1+2}},&\text{if }-\frac{21}{10}\le x_1< -2,\\
0,&\text{if }x_1>-2.
\end{cases}
\end{equation*}
From the fact that 
\begin{equation*}
\frac{\dd x^{n}_{1}}{\dd x_{1}}=nx_{1}^{n-1}\quad \mbox{and}\quad 
\frac{\dd e^{-\frac{1}{x_{1}}}}{\dd x_{1}}=\frac{1}{x_{1}^{2}}e^{-\frac{1}{x_{1}}},
\end{equation*}
we obtain
\begin{equation*}
\begin{aligned}
|G_0''|\lesssim |G_0'|^{\frac{48}{49}},\quad |G_0'''|\lesssim |G_0'|^{\frac{47}{49}}\quad \mbox{and}\quad |G_0'|\lesssim |G_0|^{\frac{49}{50}},\\
|G_1''|\lesssim |G_1'|^{\frac{48}{49}},\quad |G_1'''|\lesssim |G_1'|^{\frac{47}{49}}\quad \mbox{and}\quad |G_1'|\lesssim |G_1|^{\frac{49}{50}}.
\end{aligned}
\end{equation*}

For all $t_{0}\le t_{1}\le t$ and $x_0\gg1$, we set 
\begin{equation*}
\widetilde{G}_{0}(t,x_{1})=G_{0}(\widetilde{x}) \quad \mbox{and}\quad 
\widetilde{G}_{1}(t,x_{1})=G_{1}(\widetilde{x}),
\end{equation*}
where
\begin{equation*}
L(t)=\frac{t-t_1}{4}+x_0\quad 
\mbox{and}\quad 
\widetilde{x}=\frac{x_{1}-x_{1}(t)}{L(t)}.
\end{equation*}
Moreover, we denote
\begin{equation*}
\begin{aligned}
M_{x_{0}}(t)&=\frac{1}{2}\int_{\R^{2}}|\partial_{x_{1}}u(t)|^{2}\widetilde{G}_{1}(t,x_{1})\dd x,\\
E_{x_0}(t)&=\int_{\mathbb{R}^2}\left(\frac{1}{2}|\nabla u(t)|^2-\frac{1}{4}|u(t)|^4\right)\widetilde{G}_{0}(t,x_{1})\dd x.
\end{aligned}
\end{equation*} 
We are in a position to complete the proof of Theorem~\ref{MT2}.
\begin{proof}[End of the proof of Theorem~\ref{MT2}]
We split the proof into the following three steps.

\smallskip
Step 1. Control of $E_{x_{0}}(t)$. By an elementary computation, we have 
\begin{equation*}
\begin{aligned}
\frac{\dd}{\dd t}{E}_{x_{0}}
&=-\frac{1}{L}\int_{\R^{2}}\left(\left(\partial_{x_{1}}^{2}u\right)^{2}+\left(\partial_{x_{1}x_{2}}u\right)^{2}\right)G'_{0}(\widetilde{x})\dd x\\
&-\frac{1}{4L}\int_{\R^{2}}\left(\frac{1}{2}|\nabla u(t)|^2-\frac{1}{4}|u(t)|^4\right)
G'_{0}(\widetilde{x})\widetilde{x}\dd x\\
&-\frac{1}{2L}\int_{\R^{2}}\left(\Delta u+u^{3}\right)^{2}G'_{0}(\widetilde{x})\dd x+\frac{1}{2L^{3}}\int_{\R^{2}}\left(\partial_{x_{1}}u\right)^{2}G'''_{0}(\widetilde{x})\dd x\\
&-\frac{x_{1t}}{L}\int_{\R^{2}}\left(\frac{1}{2}|\nabla u(t)|^2-\frac{1}{4}|u(t)|^4\right)G'_{0}(\widetilde{x})\dd x+\frac{3}{L}\int_{\R^{2}}u^{2}\left(\partial_{x_{1}}u\right)^{2}G'_{0}(\widetilde{x})\dd x.
\end{aligned}
\end{equation*}
First, from ${\rm{Supp}} \ G'_{0}\subset[-2,-1]$ and $x_{1t}\sim 1$, we have 
\begin{equation*}
-\frac{1}{8L}\int_{\R^{2}}|\nabla u(t)|G'_{0}(\widetilde{x})\dd x
-\frac{x_{1t}}{2L}\int_{\R^{2}}|\nabla u(t)|G'_{0}(\widetilde{x})\widetilde{x}\dd x
\ge \frac{1}{16L}\int_{\R^{2}}|\nabla u(t)|G'_{0}(\widetilde{x})\dd x.
\end{equation*}
Then, from Proposition~\ref{Prop:decomposition} and \S\ref{SSS:Soliton}, we find
\begin{equation*}
\begin{aligned}
&\int_{\R}\int_{\widetilde{x}<-1}\left(|\nabla u|^{2}+u^{2}\right)\dd x_{1}\dd x_{2}\\
&\lesssim \int_{\R}\int^{-\frac{L}{\lambda(t)}}_{-\infty}\left(|\nabla Q|^{2}+|\nabla (bP\phi_{b})|^{2}+|\nabla \varepsilon|^{2}\right)\dd y_{1}\dd y_{2}\\
&+ \int_{\R}\int_{-\infty}^{-\frac{L}{\lambda(t)}}\left(Q^{2}+(bP\phi_{b})^{2}+\varepsilon^{2}\right)\dd y\lesssim e^{-\frac{t-t_{1}+x_{0}}{50}}+b(t)+\mathcal{N}_{0}(t)\to 0,\ \mbox{as}\ t\to \infty.
\end{aligned}
\end{equation*}
Based on the above estimate and the 2D Sobolev embedding, we obtain
\begin{equation*}
\lim_{t\to \infty}E_{x_{0}}(t)=0,\quad \mbox{for any}\ x_{0}\gg 1.
\end{equation*}
On the other hand, it is easily checked that 
\begin{equation}\label{est:G0123}
\begin{aligned}
\left|\partial_{x_{1}}\widetilde{G}_{0}\right|&\lesssim L^{-1}\left|\widetilde{G}_{0}\right|^{\frac{48}{49}},\\
\left|\partial_{x_{1}}^{2}\widetilde{G}_{0}\right|\lesssim L^{-\frac{50}{49}}\left|\partial_{x_{1}}\widetilde{G}_{0}\right|^{\frac{48}{49}}\quad &\mbox{and}\quad 
\left|\partial_{x_{1}}^{3}\widetilde{G}_{0}\right|\lesssim L^{-\frac{100}{49}}\left|\partial_{x_{1}}\widetilde{G}_{0}\right|^{\frac{47}{49}}.
\end{aligned}
\end{equation}
Therefore, based on a similar argument to the proof for Lemma~\ref{le:weight2D},
\begin{equation*}
\begin{aligned}
&\left\|u^{2}\sqrt{|\partial_{x_{1}}\widetilde{G}_{0}|}\right\|_{L^{\infty}}\\
&\lesssim L^{-\frac{100}{49}}\int_{\R^{2}}u^{2}|\partial_{x_{1}}\widetilde{G}_{0}|^{\frac{45}{98}}\dd x+L^{-\frac{50}{49}}\|u\|_{L^{2}}\left(\int_{\R^{2}}|\nabla u|^{2}|\partial_{x_{1}}\widetilde{G}_{0}|\dd x\right)^{\frac{47}{98}}\\
&+\left(\int_{\R}\int_{\widetilde{x}<-1}\left(|\nabla u|^{2}+u^{2}\right)\dd x\right)^{\frac{1}{2}}\left(\int_{\R^{2}}\left(|\nabla u|^{2}+|\partial_{x_{1}}\partial_{x_{2}}u|^{2}\right)|\partial_{x_{1}}\widetilde{G}_{0}|\dd x\right)^{\frac{1}{2}}.
\end{aligned}
\end{equation*}
Then, using the Cauchy-Schwarz inequality, we obtain
\begin{equation*}
\left|\int_{\R^{2}}(\partial_{x_{1}}u)^{2}\sqrt{|\partial_{x_{1}}\widetilde{G}_{0}|}\dd x\right|
\lesssim \left(\int_{\R}\int_{\widetilde{x}<-1}|\nabla u|^{2}\dd x\right)^{\frac{1}{2}}\left(\int_{\R^{2}}|\nabla u|^{2}|\partial_{x_{1}}\widetilde{G}_{0}|\dd x\right)^{\frac{1}{2}}.
\end{equation*}
It follows from~\eqref{est:G0123} that 
\begin{equation*}
\begin{aligned}
&\left|\frac{1}{L}\int_{\R^{2}}u^{2}\left(\partial_{x_{1}}u\right)^{2}G'_{0}(\widetilde{x})\dd x\right|\\
&\lesssim 
\left(\left|\int_{\R^{2}}(\partial_{x_{1}}u)^{2}\sqrt{|\partial_{x_{1}}\widetilde{G}_{0}|}\dd x\right|\right)
\left\|u^{2}\sqrt{|\partial_{x_{1}}\widetilde{G}_{0}|}\right\|_{L^{\infty}}\\
&\lesssim \frac{1}{L^{3}}+\frac{o(1)}{L}\int_{\R^{2}}\left(|\nabla u|^{2}+|\partial_{x_{1}}\partial_{x_{2}}u|^{2}\right)|{G}'_{0}(\widetilde{x})|\dd x.
\end{aligned}
\end{equation*}
Here, we use the fact that 
\begin{equation*}
\lim_{t\to \infty}\left(L^{-1}(t)+\int_{\R}\int_{\widetilde{x}<-1}\left(|\nabla u|^{2}+u^{2}\right)\dd x\right)=0.
\end{equation*}
Then, based on a similar argument to the proof for Lemma~\ref{le:weight2D}, we obtain
\begin{equation*}
\begin{aligned}
\left|\frac{1}{L}\int_{\R^{2}}u^{4}G_{0}'(\widetilde{x})\dd x\right|
&\lesssim 
\left(\int_{\R}\int_{\widetilde{x}<-1}u^{2}\dd x\right)\left(\int_{\R^{2}}|\nabla u|^{2}|\partial_{x_{1}}\widetilde{G}_{0}|\dd x\right)\\
&+\left(\int_{\R}\int_{\widetilde{x}<-1}u^{2}\dd x\right)
\left(\int_{\R^{2}}u^{2}\left|\frac{\partial^{2}_{x_{1}}\widetilde{G}_{0}}{\sqrt{\partial_{x_{1}}\widetilde{G}_{0}}}\right|^{2}\dd x\right)\\
&\lesssim \frac{1}{L^{3}}+\frac{o(1)}{L}\int_{\R^{2}}\left(|\nabla u|^{2}+|\partial_{x_{1}}\partial_{x_{2}}u|^{2}\right)|{G}'_{0}(\widetilde{x})|\dd x.
\end{aligned}
\end{equation*}
Combining the above estimates, one has
\begin{equation*}
\frac{\dd }{\dd t}E_{x_{0}}(t)\ge -\frac{1}{20L}\int_{\R^{2}}\left(|\nabla u|^{2}+|\partial_{x_{1}}\partial_{x_{2}}u|^{2}\right)G'_{0}(\widetilde{x})\dd x-\frac{C}{L^{3}}.
\end{equation*}
Here, $C>1$ is a universal constant independent with $L$. Integrating the above estimate over $[t_{1},t]$ and then using the fact that $E_{x_{0}}(t)\to 0$ as $t\to \infty$, 
\begin{equation}\label{est:Ex0}
E_{x_{0}}(t_{1})-\int_{t_{1}}^{\infty}\frac{1}{L(t)}\int_{\R^{2}}\left(|\nabla u|^{2}+|\partial^{2}_{x_{1}}u|^{2}+|\partial_{x_{1}}\partial_{x_{2}}u|^{2}\right)G'_{0}(\widetilde{x})\dd x\dd t\lesssim \frac{1}{x^{2}_{0}}.
\end{equation}
\smallskip
Step 2. Refined control of localized $\dot{H}^{1}$ norm. First, using again~\eqref{est:G0123} and a similar argument to the proof of Lemma~\ref{le:weight2D}, 
\begin{equation*}
\begin{aligned}
\int_{\R^{2}}u^{4}\widetilde{G}_{0}\dd x
&\lesssim \left(\int_{\R}\int_{\widetilde{x}<-1}u^{2}\dd x\right)\left(\int_{\R^{2}}|\nabla u|^{2}|\widetilde{G}_{0}|\dd x\right)\\
&+\left(\int_{\R}\int_{\widetilde{x}<-1}u^{2}\dd x\right)
\left(\int_{\R^{2}}u^{2}\left|\frac{\partial_{x_{1}}\widetilde{G}_{0}}{\sqrt{\widetilde{G}_{0}}}\right|^{2}\dd x\right)\\
&\lesssim \frac{1}{L^{2}}+o(1)\int_{\R^{2}}|\nabla u|^{2}\widetilde{G}_{0}\dd x\lesssim 
\frac{1}{x_{0}^{2}}+o(1)\int_{\R^{2}}|\nabla u|^{2}\widetilde{G}_{0}\dd x.
\end{aligned}
\end{equation*}
It follows from~\eqref{est:Ex0} that 
\begin{equation}\label{est:endnablau}
\int_{\R}\int_{\widetilde{x}<-2}|\nabla u|^{2}\dd x\lesssim 
\int_{\R^{2}}|\nabla u|^{2}\widetilde{G}_{0}\dd x\lesssim E_{x_{0}}(t)+\frac{1}{x_{0}^{2}}\lesssim \frac{1}{x^{2}_{0}}.
\end{equation}
Using again~\eqref{est:Ex0},
\begin{equation*}
\begin{aligned}
&\frac{\dd}{\dd x_{0}}\left(\int_{t_{0}}^{\infty}\int_{\R^{2}}
\left(
|\nabla u|^{2}+(\partial_{x_{1}}^{2}u)^{2}+(\partial_{x_{1}}\partial_{x_{2}}u)^{2}
\right)G_{0}(\widetilde{x})\dd x\dd t
\right)\\
&=-\int_{t_{0}}^{\infty}\int_{\R^{2}}
\left(
|\nabla u|^{2}+(\partial_{x_{1}}^{2}u)^{2}+(\partial_{x_{1}}\partial_{x_{2}}u)^{2}
\right)G'_{0}(\widetilde{x})\frac{\widetilde{x}}{L}\dd x\dd t\\
&\ge \int_{t_{0}}^{\infty}\int_{\R^{2}}\left(
|\nabla u|^{2}+(\partial_{x_{1}}^{2}u)^{2}+(\partial_{x_{1}}\partial_{x_{2}}u)^{2}
\right)G'_{0}(\widetilde{x})\frac{2}{L}\dd x\dd t\gtrsim -\frac{1}{x_{0}^{2}}.
\end{aligned}
\end{equation*}
Integrating the above estimate over $[x_{0},\infty)$, we see that 
\begin{equation}\label{est:endG0x}
\int_{t_{1}}^{\infty}\int_{\R^{2}}
\left(
|\nabla u|^{2}+(\partial_{x_{1}}^{2}u)^{2}+(\partial_{x_{1}}\partial_{x_{2}}u)^{2}
\right)G_{0}(\widetilde{x})\dd x\dd t\lesssim \frac{1}{x_{0}}.
\end{equation}
Here, we use the fact that 
\begin{equation*}
\lim_{x_{0}\to \infty}\int_{t_{1}}^{\infty}\int_{\R^{2}}
\left(
|\nabla u|^{2}+(\partial_{x_{1}}^{2}u)^{2}+(\partial_{x_{1}}\partial_{x_{2}}u)^{2}
\right)G_{0}(\widetilde{x})\dd x\dd t=0.
\end{equation*}
Note that, 
\begin{equation}\label{est:G1123}
\begin{aligned}
\left|\partial_{x_{1}}\widetilde{G}_{1}\right|&\lesssim L^{-1}\left|\widetilde{G}_{1}\right|^{\frac{48}{49}},\\
\left|\partial_{x_{1}}^{2}\widetilde{G}_{1}\right|\lesssim L^{-\frac{50}{49}}\left|\partial_{x_{1}}\widetilde{G}_{1}\right|^{\frac{48}{49}}\quad &\mbox{and}\quad 
\left|\partial_{x_{1}}^{3}\widetilde{G}_{1}\right|\lesssim L^{-\frac{100}{49}}\left|\partial_{x_{1}}\widetilde{G}_{1}\right|^{\frac{47}{49}}.
\end{aligned}
\end{equation}
By an elementary computation, 
\begin{equation*}
\begin{aligned}
\frac{\dd}{\dd t}M_{x_{0}}
&=-\frac{1}{2L}\int_{\R^{2}}\left(3(\partial^{2}_{x_{1}}u)^{2}+(\partial_{x_{1}}\partial_{x_{2}}u)^{2}\right)G'_{1}(\widetilde{x})\dd x-3\int_{\R^{2}}u(\partial_{x_{1}}u)^{3}G_{1}(\widetilde{x})\dd x
\\
&+\frac{1}{2L^{3}}\int_{\R^{2}}(\partial_{x_{1}}u)^{2}G'''_{1}(\widetilde{x})\dd x+\frac{3}{2L}\int_{\R^{2}}u^{2}(\partial_{x_{1}}u)^{2}G'_{1}(\widetilde{x})\dd x\\
&-\frac{x_{1t}}{2L}\int_{\R^{2}}(\partial_{x_{1}}u)^{2}G'_{1}(\widetilde{x})\dd x
-\frac{1}{8L}\int_{\R^{2}}(\partial_{x_{1}}u)^{2}G'_{1}(\widetilde{x})\widetilde{x}\dd x.
\end{aligned}
\end{equation*}
First, from Supp $G'_{1}\subset [-3,-2]$ and $x_{1t}\sim 1$, we have 
\begin{equation*}
-\frac{x_{1t}}{2L}\int_{\R^{2}}(\partial_{x_{1}}u)^{2}G'_{1}(\widetilde{x})\dd x
-\frac{1}{8L}\int_{\R^{2}}(\partial_{x_{1}}u)^{2}G'_{1}(\widetilde{x})\widetilde{x}\dd x\ge
-\frac{1}{20L}\int_{\R^{2}}(\partial_{x_{1}}u)^{2}G'_{1}(\widetilde{x})\dd x.
\end{equation*}
Then, from~\eqref{est:Solitonb} and (ii) of Lemma~\ref{le:Conmon}, we deduce that 
\begin{equation*}
\begin{aligned}
&\int_{t_{1}}^{\infty}\left|\frac{1}{L^{3}(t)}\int_{\R^{2}}(\partial_{x_{1}}u)^{2}G'''_{1}(\widetilde{x})\dd x\right|\dd t\\
&\lesssim \frac{1}{x_{0}^{3}}\int_{t_{1}}^{\infty}\left(e^{-\frac{t-t_{1}+x_{0}}{50}}+|b(t)|+
\mathcal{N}_{0}(t)\right)\dd t\lesssim \frac{1}{x_{0}^{3}}.
\end{aligned}
\end{equation*}
Based on~\eqref{est:endnablau} and a similar argument to the one in the Lemma~\ref{le:weight2D}, 
\begin{equation*}
\begin{aligned}
&\left|\frac{1}{L}\int_{\R^{2}}u^{2}(\partial_{x_{1}}u)^{2}G'_{1}(\widetilde{x})\dd x\right|\\
&\lesssim 
\left(\left|\int_{\R^{2}}(\partial_{x_{1}}u)^{2}\sqrt{|\partial_{x_{1}}\widetilde{G}_{1}|}\dd x\right|\right)
\left\|u^{2}\sqrt{|\partial_{x_{1}}\widetilde{G}_{1}|}\right\|_{L^{\infty}}\\
&\lesssim \frac{1}{L^{5}}+\frac{1}{x^{2}_{0}}\int_{\R^{2}}\left(|\nabla u|^{2}+|\partial_{x_{1}}\partial_{x_{2}}u|^{2}\right)|{G}'_{1}(\widetilde{x})|\dd x.
\end{aligned}
\end{equation*}
It follows from~\eqref{est:endG0x} that 
\begin{equation*}
\int_{t_{1}}^{\infty}\left|\frac{1}{L}\int_{\R^{2}}u^{2}(\partial_{x_{1}}u)^{2}G'_{1}(\widetilde{x})\dd x\right|\dd t\lesssim \frac{1}{x_{0}^{3}}.
\end{equation*}
Then, using again~\eqref{est:G1123} and a similar argument to the one in the Lemma~\ref{le:weight2D},
\begin{equation*}
\begin{aligned}
&\left|\int_{\R^{2}}u^{2}(\partial_{x_{1}}u)^{2}G_{1}(\widetilde{x})\dd x\right|\\
&\lesssim 
\left(\left|\int_{\R^{2}}(\partial_{x_{1}}u)^{2}\sqrt{\widetilde{G}_{1}}\dd x\right|\right)
\left\|u^{2}\sqrt{\widetilde{G}_{1}}\right\|_{L^{\infty}}\\
&\lesssim \frac{1}{L^{5}}+\frac{1}{x_{0}}\int_{\R^{2}}\left(|\nabla u|^{2}+|\partial_{x_{1}}\partial_{x_{2}}u|^{2}\right){G}_{1}(\widetilde{x})\dd x.
\end{aligned}
\end{equation*}
Then, from~\eqref{est:endnablau}, we see that 
\begin{equation*}
\begin{aligned}
\left|\int_{\R^{2}}(\partial_{x_{1}}u)^{4}G_{1}(\widetilde{x})\dd x\right|
&\lesssim \left(\int_{\widetilde{x}<-2}(\partial_{x_{1}}u)^{2}\dd x\right)
\left(\int_{\R^{2}}\left((\partial^{2}_{x_{1}}u)^{2}+(\partial_{x_{1}}\partial_{x_{2}}u)^{2}\right)G_{1}(\widetilde{x})\dd x\right)\\
&+ \left(\int_{\widetilde{x}<-2}(\partial_{x_{1}}u)^{2}\dd x\right)
\left(\int_{\R^{2}}(\partial_{x_{1}}u)^{2}\frac{\partial_{x_{1}}\widetilde{G}_{1}}{\sqrt{\widetilde{G}_{1}}}\dd x\right)\\
&\lesssim \frac{1}{x_{0}^{2}}\int_{\R^{2}}\left( 
(\partial^{2}_{x_{1}}u)^{2}
+(\partial_{x_{1}}\partial_{x_{2}}u)^{2}
\right){G}_{1}(\widetilde{x})\dd x
+\frac{1}{x^{2}_{0}L^{2}}.
\end{aligned}
\end{equation*}
Therefore, from the Cauchy-Schwarz inequality, 
\begin{equation}
\begin{aligned}
&\left|\int_{\R^{2}}u(\partial_{x_{1}} u)^{3}G_{1}(\widetilde{x})\dd x\right|\\
&\lesssim \left(\int_{\R^{2}}u^{2}(\partial_{x_{1}} u)^{2}G_{1}(\widetilde{x})\dd x\right)^{\frac{1}{2}}
\left(\int_{\R^{2}}(\partial_{x_{1}} u)^{4}G_{1}(\widetilde{x})\dd x\right)^{\frac{1}{2}}\\
&\lesssim \frac{1}{x_{0}^{\frac{3}{2}}}\int_{\R^{2}}\left( 
(\partial^{2}_{x_{1}}u)^{2}
+(\partial_{x_{1}}\partial_{x_{2}}u)^{2}
\right){G}_{1}(\widetilde{x})\dd x+\frac{1}{L^{5}}+\frac{1}{x_0^\frac{3}{2}L^2}.
\end{aligned}
\end{equation}
Combining the above estimates, we obtain
\begin{align*}
\frac{\dd}{\dd t}M_{x_{0}} &\gtrsim -\frac{1}{x_{0}^{\frac{3}{2}}}\int_{\R^{2}}\left( 
(\partial^{2}_{x_{1}}u)^{2}
+(\partial_{x_{1}}\partial_{x_{2}}u)^{2}
\right){G}_{1}(\widetilde{x})\dd x-\frac{1}{L^{5}}-\frac{1}{x_0^\frac{3}{2}L^2}.
\end{align*}
Integrating the above estimate over $[t_{1},\infty)$, we obtain
\begin{align*}
M_{x_0}(t_1)\lesssim \frac{1}{x_0^{\frac{5}{2}}}+\frac{1}{x_0^{\frac{3}{2}}}\int_{t_1}^{+\infty}\left(\int_{\R^{2}}\left( 
(\partial^{2}_{x_{1}}u)^{2}
+(\partial_{x_{1}}\partial_{x_{2}}u)^{2}
\right){G}_{1}(\widetilde{x})\dd x\right)\dd t\lesssim \frac{1}{x_0^{\frac{5}{2}}},
\end{align*}
which implies
\begin{equation}\label{eq:-3u_x}
\int_{\mathbb{R}}\int_{\widetilde{x}<-3}|\partial_{x_1} u|^2\, \dd x\leq G_{x_0}\lesssim \frac{1}{x_0^{\frac{5}{2}}}.
\end{equation}
\smallskip
Step 3. Conclusion. 
Using the Sobolev embedding with $x_2$ fixed, for all $y_0\gg1$,
\begin{align*}
&|u(t,x_1(t)-2y_0,x_2)|^2\\
&\lesssim \left(\int_{-\infty}^{x_{1}(t)-y_0}(\partial_{x_1}u)^2(t,x_1,x_2)\dd x_1\right)^{\frac{1}{2}}\left(\int_{-\infty}^{x_{1}(t)-y_0}u^2(t,x_1,x_2)\dd x_1\right)^{\frac{1}{2}}.
\end{align*}
It follows from the H\"older's inequality and \eqref{eq:-3u_x} that 
\begin{equation*}
\begin{aligned}
&\int_{\mathbb{R}}|u(t,x_1(t)-2y_0,x_2)|^2\dd x_2\\
&\leq \left(\int_{\mathbb{R}}\int_{\infty}^{x_1(t)-y_0}(\partial_{x_1}u)^2\dd x\right)^{\frac{1}{2}}\left(\int_{\mathbb{R}}\int_{-\infty}^{x_1(t)-y_0}u^2\dd x\right)^{\frac{1}{2}}\lesssim \frac{1}{y_0^{\frac{5}{4}}}.
\end{aligned}
\end{equation*}
Integrating the above estimate with respect to $y_0$ over $[\frac{x_{0}}{2},\infty)$, we obtain
\begin{equation}\label{eq:decayx0}
\int_{\R}\int_{-\infty}^{x_1(t)-x_0}u^2\,\dd x \lesssim \frac{1}{x^\frac{1}{4}_0}.
\end{equation}
For any given $x_0\gg 1$, we split the $L_2$ norm of $u$ into the following two pieces,
\begin{equation*}
\int_{\R^{2}}u^2\,\dd x=
 \int_{\R}\int_{-\infty}^{x_1(t)-x_0}u^2\,\dd x+ \int_{\R}\int^{\infty}_{x_1(t)-x_0}u^2\,\dd x.
\end{equation*}
The first term above will go to $0$ as $x_0\rightarrow\infty$ by \eqref{eq:decayx0}. Then, for the second term, 
\begin{align*}
&\left|\int_{\R} \int^{\infty}_{x_1(t)-x_0}u^2\,\dd x-\int_{\R^{2}} Q^2\,\dd x\right|\\
&=\left|\int_{\R}\int^{\infty}_{-\frac{x_{0}}{\lambda(t)}}|Q+b\phi_bP+\varepsilon|^2\,\dd y-\int_{\R^{2}}Q^2\,\dd y\right|\\
&\lesssim
\left(\int_{\R} \int^{\infty}_{-2x_0}\varepsilon^2\,\dd y+|b|^2\int_{\R^{2}}\phi_b^2P^2\,\dd y\right)^{\frac{1}{2}}
+\int_{\R} \int_{-\infty}^{\frac{x_{0}}{2}}Q^2\,\dd y\\
&+\int_{\R} \int^{\infty}_{-2x_0}\varepsilon^2\,\dd y+|b|^2\int_{\R^{2}}\phi_b^2P^2\,\dd y
\lesssim e^{\frac{10x_0}{B}}\mathcal{N}_0(t)+|b(t)|+e^{-\frac{x_0}{10}}.
\end{align*}
We know $\mathcal{N}_0(t)$ and $|b(t)|$ go to $0$ as $t\rightarrow\infty$. We can take a fixed $x_0$ to make  $e^{-\frac{x_0}{10}}$ and $x^{-\frac{1}{4}}_{0}$ as small as we want, and then pass $t\rightarrow\infty$. Therefore, we obtain
\begin{equation*}
\lim_{t\rightarrow+\infty}\int_{\R^{2}}|u(t)|^2 \dd x=\int_{\R^{2}}Q^2\dd x,
\end{equation*}
 which leads a contradiction, and so the proof of the zero energy case is completed.
\end{proof}

\appendix
\section{Numerical computations}\label{App:A}
In this appendix, we provide some details for the numerical computations for the value of $\theta$. These numerical computations were carried out by Mathematica 13.1.

Recall that, $Q$ is the unique nonnegative radial solution with exponential decay to the following second-order elliptic equation:
\begin{equation*}
-\Delta Q+Q-Q^3=0\quad \mbox{on}\ \R^{2}.
\end{equation*}
Recall also that, we denote by $\Lambda$ the scaling operator $\Lambda={\rm{Id}}+x\cdot \nabla$.

Our numerical computations are used to compute
\begin{equation*}
  \theta=2\bigg(\int_{\R}\frac{\big|\widehat{F}(\xi)\big|^{2}}{1+|\xi|^{2}}\dd \xi\bigg)
\bigg/
\left(\int_{\R}\big|\widehat{F}(\xi)\big|^{2}\dd \xi\right).
\end{equation*}
Here, the functions $F$ and $\widehat{F}$ are given by 
\begin{equation*}
F(y_{2})=\int_{\R}\Lambda Q(y_{1},y_{2})\dd y_{1}\quad \mbox{and}\quad 
\widehat{F}(\xi)=\frac{1}{\sqrt{2\pi }}\int_{\mathbb{R}}F(y_{2})e^{-iy_{2}\xi}\dd y_{2}.
\end{equation*}
\subsection{Numerical computations of $Q$} We numerically compute $Q$ in the polar coordinates
\begin{equation*}
    -R_{rr}-\frac{1}{r}R_r+R-R^3=0,
\end{equation*}
with $R_r(0)=0$ and $R(r)\rightarrow0$ as $r\rightarrow0$. To perform the numerical computation, we truncate the system to $r\in[0,L]$ and set $R(L)=0$. Then we employ the non-spectral renormalization method to iterate and obtain an approximated solution.  For full details, we refer to Section 3 of Chapter 28 in Fibich \cite{Fib} which also contains references on the convergence of the non-spectral renormalization method. Due to the exponential decay of $Q$, the error caused by the truncation is very small. After obtaining the numerical solution in $r$, we use the standard interpolation to recover the numerical solution in the $(x_{1},x_{2})$ coordinate.

\subsection{Fourier transforms}
We use the default numerical integrations in Mathematica to integrate the numerical solution obtained above to find an approximation of $F$.  Following this, we apply the FFT (Fast Fourier Transform) in Mathematica to compute the Fourier transform of the approximated $F$. It is necessary to renormalize the constants to align with our conventions of Fourier transforms; see the codes in the next section for details.

\subsection{Mathematica code}
With these numerical computations, one has
\begin{center}
\begin{tabular}{ |c|c|c|c|c|} 
 \hline
 \multicolumn{5}{|c|}{Numerical values of $\theta$} \\
\hline
Grid size \textbackslash  $L$ & \qquad 5 \qquad & \qquad10 \qquad& \qquad15\qquad & \qquad20\qquad\\
\hline

0.05 & 1.65849  & 1.65849  & 1.66112 & 1.66112\\
\hline
0.02  & 1.67766 & 1.65741 & 1.66006 & 1.66095\\ 
\hline 
0.01 & 1.67862 & 1.65703 & 1.66112 & 1.66032\\ 
\hline
\end{tabular}
\end{center}
We also plot the graph of \(\widehat{F}(\xi)\) with \(L=20\) and a grid size of $0.02$ as a reference.
\smallskip
\begin{center}
\includegraphics[scale=0.17]{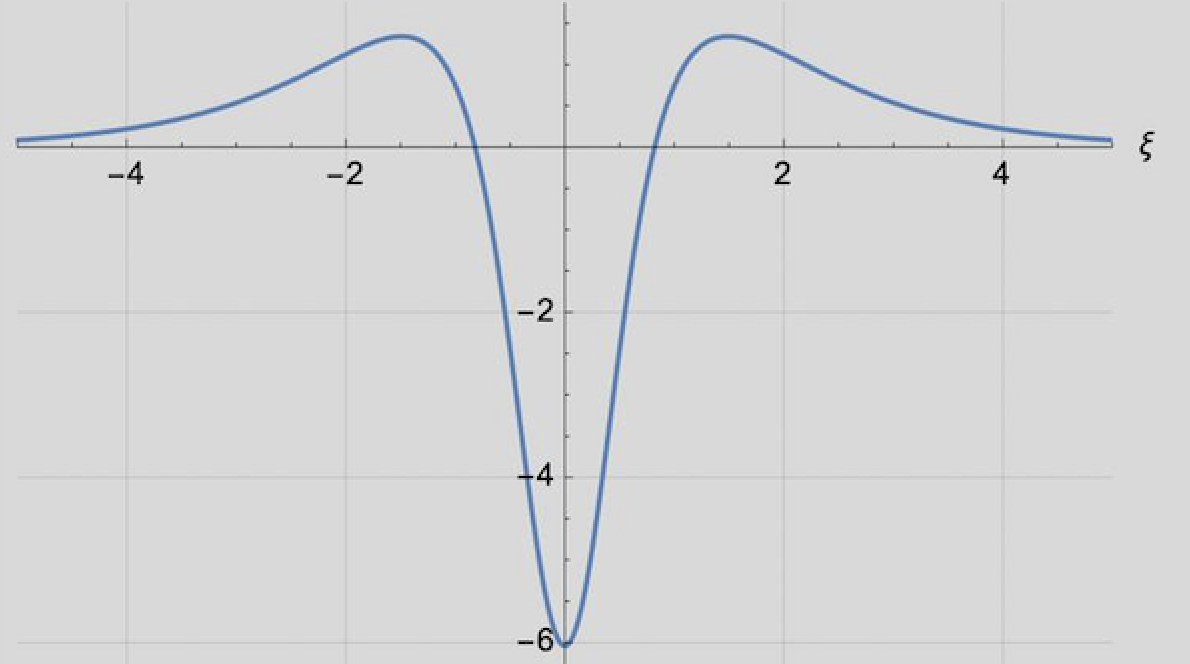}
\end{center}

Picking $L=20$ and the grid size to be $0.01$, we record
\begin{equation*}
    \theta=2\bigg(\int_{\R}\frac{\big|\widehat{F}(\xi)\big|^{2}}{1+|\xi|^{2}}\dd \xi\bigg)
\bigg/
\left(\int_{\R}\big|\widehat{F}(\xi)\big|^{2}\dd \xi\right)\approx 1.66032.
\end{equation*}
For the sake of completeness, we provide the code for computations in Mathematica. 
{\footnotesize
\begin{lstlisting}[extendedchars=true,language=Mathematica]
dr=0.05;
m=0;
r=Range[0,rmax=20,dr];
R=r Exp[-r^2]//N;
R::usage="Initial guess."; 

rdr=r dr;
i0=Max[Round[Min[1/rmax,0.1]/dr],1];
M=Length[r];
L1=SparseArray[{Band[{2,3}]->1./dr^2,Band[{2,1}]->1./dr^2,Band[{1,1}]->-(2./dr^2)-m^2/Max[r,dr/10]^2,{1,2}->2./dr^2},{M,M}];
L2=SparseArray[{Band[{i0+1,i0+2}]->1./(2. dr r[[i0+1;;-2]]),Band[{i0+1,i0}]->-(1./(2. dr r[[i0+1;;All]]))},{M,M}];
L3[R_,{d_,\[Sigma]_}]:=Module[{R0=R[[1]],R0p2,R0p4},R0p2=(R0 (1-Abs[R0]^(2 \[Sigma])))/d;R0p4=(3. R0p2 (1-(2 \[Sigma]+1) Abs[R0]^(2 \[Sigma])))/(2.+d);SparseArray[{Band[{1,1}]->R0p2+(r[[1;;i0]]^2 R0p4)/6.},{M,M}]];
L[R_,{d_,\[Sigma]_}]:=L1+L2+L3[R,{d,\[Sigma]}]-IdentityMatrix[M];
iL[R_,{d_,\[Sigma]_}]:=Inverse[L[R,{d,\[Sigma]}]];
L[R_]:=L[R,{2,1}];
L3[R_]:=L3[R,{2,1}];
iL[R_]:=iL[R,{2,1}]; 

SL[R_]:=Total[rdr*Abs[R]^2];
SR[R_,{d_,\[Sigma]_}]:=-Total[rdr*R iL[R,{d,\[Sigma]}].(Abs[R]^(2\[Sigma]) R)];
SR[R_,{d_,\[Sigma]_},iLR_]:=-Total[rdr*R iLR];
SR[R_]:=SR[R,{2,1}]; 

R0set[R_]:=If[m==0,R,ReplacePart[R,1->0.]];
newR[R_,{d_,\[Sigma]_}]:=With[{iLR=R0set[iL[R0set[R],{d,\[Sigma]}].(Abs[R]^(2 \[Sigma]) R)]},-Abs[SL[R]/SR[R,{d,\[Sigma]},iLR]]^(((2 \[Sigma]+1)/(2 \[Sigma]))) iLR];
newR[R_]:=newR[R,{2,1}]; 

solver[initR_,{d_,\[Sigma]_},OptionsPattern[{MaxIterations->1000,Tolerance->N[1*^-10]}]]:=With[{f=Function[{nR},{nR[[1]]+1,nR[[3]],newR[nR[[3]],{d,\[Sigma]}]}]},FixedPoint[f,{0,0,initR},OptionValue[MaxIterations],SameTest->(Max[Abs[#1[[3]]-#2[[3]]]]<OptionValue[Tolerance]&)]];
solver[initR_]:=solver[initR,{2,1}];
solver::usage="Solver given inital guess. Returns {# of iterations, 2nd last result, last result}";

simplesolver[initR_]:=With[{s=solver[R]},Print["Converged after "<>TextString[s[[1]]]<>" steps"];Print["Max error is "<>TextString[Max[Abs[s[[2]]-s[[3]]]]]];s[[3]]];
simplesolver::usage="A simplified solver. Example: Q=simplesolver[R];"; 

s=solver[R,{2,1}];
Qp=s[[2]];
Q=s[[3]];

nD[x_List]:=Join[Join[{x[[2]]-x[[1]]},1/2 (x[[3;;All]]-x[[1;;-3]])],{x[[-1]]-x[[-2]]}];
nD::usage="Discrete derivative in r.";
\[CapitalLambda][R_]:=R+(r nD[R])/dr;
normR[R_]:=Total[rdr Abs[R]^2]; 

interpR[R_]:=With[{f=Interpolation[Transpose[{r,R}]]},{x}|->Piecewise[{{f[x],0<=x<=Max[r]}},0]];
interpR::usage="Evaluate function on a rectangular grid.";

\[CapitalLambda]Q=ParallelTable[interpR[\[CapitalLambda][Q]][Norm[{x,y}]],{x,rpm},{y,rpm}];
\[CapitalLambda]Q::usage="\[CapitalLambda]Q(x,y).";
g=((#+Reverse[#])/2&)[Total[\[CapitalLambda]Q dr]];
g::usage="The g(x,y). Also make it duely even."; 


fftshift[x_]:=With[{ls=Floor[Dimensions[x]/2]},RotateLeft[x,ls]];
ifftshift[x_]:=With[{ls=Floor[Dimensions[x]/2]},RotateRight[x,ls]];
fftshift::usage="Shift r=0 to first element."
ifftshift::usage="Shift r=0 to central element. ifftshift@fftshift=identity."; 

d\[Xi]=N[(2 \[Pi])/((2 M-1) dr)];
\[Xi]=Range[1-M,M-1] d\[Xi];
g\[Xi]:=Sqrt[Length[g]/(2Pi)]dr Re[ifftshift[Fourier[fftshift[g]]]];
g\[Xi]::usage="g\[Xi] fourier transforms, g[\[Xi]]:=1/Sqrt[2\[Pi]]\[Integral]g(r)E^Ix\[Xi]\[DifferentialD]r. Utilizes the fact that g(r) is an even function, hence take real part of g[\[Xi]] in the end."; 

theta[Q]=1.66032;

ListPlot[Transpose[{\[Xi],g\[Xi]}],AxesLabel->{"\[Xi]","g[\[Xi]]"}]
ListPlot[Transpose[{\[Xi],g\[Xi]}],AxesLabel->{"\[Xi]","g[\[Xi]]"},PlotRange->{{-5,5},All}] 
\end{lstlisting}
}

\end{document}